\documentclass{amsart} \usepackage{hyperref}
\usepackage{enumerate} 
\usepackage{upref}
\usepackage{verbatim} 
\usepackage{color}
\usepackage{graphicx}
\usepackage[textsize=tiny]{todonotes}

\newcommand*{\mailto}[1]{\href{mailto:#1}{\nolinkurl{#1}}}

\usepackage[latin1]{inputenc}
\usepackage{amssymb, amsmath, amsfonts, amsthm}
\usepackage[all]{xy}

\usepackage{marginnote}

\DeclareMathOperator{\id}{Id}
\DeclareMathOperator{\meas}{meas}

\newcommand{\dott}{\, \cdot\,}
\newcommand{\epsi}{\varepsilon}

\newcommand{\quot}{{\F/\Gr}}
\newcommand{\Gr}{G}

\newcommand{\D}{\ensuremath{\mathcal{D}}}
\newcommand{\G}{\ensuremath{\mathcal{G}}}
\newcommand{\F}{\ensuremath{\mathcal{F}}}

\newcommand{\inv}{{^{-1}}}
\newcommand{\abs}[1]{\left\vert#1\right\vert}
\newcommand{\abss}[1]{\vert#1\vert}

\newcommand{\Real}{\mathbb R}
\newcommand{\Natural}{\mathbb N}

\newcommand{\norm}[1]{\left\Vert#1\right\Vert}

\newcommand{\Linf}{{L^\infty(\Real)}}

\newcommand{\muac}{\mu_{\text{\rm ac}}}

\DeclareMathOperator{\sgn}{sgn}
\newcommand{\nn}{\nonumber}
\newcommand{\mus}{\mu_{\text{\rm s}}}

\newcommand{\dx}{\,dx}
\newcommand{\PP}{\mathcal{P}}
\DeclareMathOperator{\Ima}{Im}

\DeclareMathOperator{\artanh}{artanh}

\newtheorem{theorem}{Theorem}[section]

\newtheorem{lemma}[theorem]{Lemma}
\newtheorem{definition}[theorem]{Definition}
\newtheorem{proposition}[theorem]{Proposition}

\newtheorem{remark}[theorem]{Remark}
\newtheorem{corollary}[theorem]{Corollary}
\numberwithin{equation}{section}

\allowdisplaybreaks

\begin{document}

\title[Global solutions with nonvanishing asymptotics]{A continuous interpolation between
  conservative and dissipative solutions for the two-component Camassa--Holm system}

\author[K. Grunert]{Katrin Grunert}
\address{Department of Mathematical Sciences\\ Norwegian University of Science and Technology\\ NO-7491 Trondheim\\ Norway}
\email{\mailto{katring@math.ntnu.no}}
\urladdr{\url{http://www.math.ntnu.no/~katring/}}

\author[H. Holden]{Helge Holden}
\address{Department of Mathematical Sciences\\
  Norwegian University of Science and Technology\\
  NO-7491 Trondheim\\ Norway\\ {\rm and} Centre of
  Mathematics for Applications\\ University of Oslo\\
  NO-0316 Oslo\\ Norway}
\email{\mailto{holden@math.ntnu.no}}
\urladdr{\url{http://www.math.ntnu.no/~holden/}}

\author[X. Raynaud]{Xavier Raynaud}
\address{Department of Mathematical Sciences\\
  Norwegian University of Science and Technology\\
  NO-7491 Trondheim\\ Norway}
\email{\mailto{xavierra@cma.uio.no}}
\urladdr{\url{http://folk.uio.no/xavierra/}}

%\date{\today} 
\thanks{Research supported in part by the
  Research Council of Norway  and by the Austrian Science Fund (FWF) under Grant No.~J3147.}  

\begin{abstract}
  We introduce a novel solution concept, denoted $\alpha$-dissipative solutions,  that provides a continuous interpolation between conservative and dissipative solutions of the 
  Cauchy problem for the two-component Camassa--Holm system on the line with vanishing  asymptotics. All the $\alpha$-dissipative solutions are global weak solutions of  the same equation in Eulerian coordinates, yet they exhibit rather distinct behavior at wave breaking. The solutions are constructed after a transformation into Lagrangian variables, where the solution is carefully modified at wave breaking. 
\end{abstract}
\maketitle

\section{Introduction}\label{sec:intro}

We consider the Cauchy problem for the two-component
Camassa--Holm (2CH) system  given by
\begin{subequations}
  \label{eq:chsys2A}
  \begin{align}
    \label{eq:chsys21A}
    u_t-u_{txx}+\kappa u_x+3uu_x-2u_xu_{xx}-uu_{xxx}+\eta\rho\rho_x&=0,\\ \label{eq:chsys22A}
    \rho_t+(u\rho)_x&=0, 
  \end{align}
\end{subequations}
with initial data $u|_{t=0}=u_0$ and $\rho|_{t=0}=\rho_0$. Here, $\kappa\in\Real$ and
$\eta\in(0,\infty)$ are given parameters. We are interested in global weak solutions for general
initial data
\begin{equation}
  \label{eq:nonvanlimA}
  u_0\in H^1(\Real)\quad \text{ and }\quad \rho_0\in L^2(\Real).
\end{equation}
The 2CH system was introduced by Olver and Rosenau  \cite[Eq.~(43)]{OlverRosenau} (see also \cite{ChenLiuZhang,AratynGomesZimerman,Ivanov}), and derived in the context of water waves by  Constantin and Ivanov \cite{ConstantinIvanov:2008}. In this paper also the question of wave breaking is analyzed. The scalar
CH equation, which corresponds to the case where $\rho(t,x)=\rho_0(x)=0$, was introduced by Camassa
and Holm in the fundamental paper \cite{CH:93}, and its analysis has been pervasive. Other generalizations of the Camassa--Holm equation 
exist, see, e.g., \cite{ChenLiuZhang,ChenLiu2010,FuQu:09,GuanYin2010,Kuzmin}. 

The 2CH system experiences wave breaking in the sense that the spatial derivative of $u$ becomes
unbounded while keeping its $H^1(\Real)$ norm finite. This gives rise to a dichotomy between
so-called conservative and dissipative solutions, which complicates the issue of wellposedness of
the Cauchy problem. This issue has been studied extensively \cite{GHR4,fritz60,GHR5,WangHuangChen,TianWangZhou}.  Analysis of blow-up and existence of global solutions for the 2CH system can be found in, e.g., \cite{GuanKarlsenYin,GuiLiu2010,GuiLiu2011,GuanYin2011,GuoZhou2010,HolmNaraighTronci,HuYin}.
  
 In this article, we introduce a novel class of solutions parametrized by
$\alpha\in[0,1]$. The parameter $\alpha$ determines the amount of dissipation for the corresponding
class of solutions. If $\alpha=0$, there is no dissipation and we obtain the conservative solutions,
meaning that, when a collision, i.e.,  wave breaking, occurs, the energy contained in the collision is entirely
redistributed in the system after the collision. If $\alpha=1$, we obtain the (fully) dissipative
solutions, where all the energy contained in a collision vanishes from the system. The intermediate
values of $\alpha$ give the fraction of the energy contained in the collision which is
dissipated. The remaining energy is given back after the collision.

For simplicity, in this introduction, we consider first the CH equation with $\kappa=0$.  However, in
the text proper, we analyze the full 2CH system. Dissipation occurs when the solution blows up. The
problem of blow-up can be studied explicitly in the case of multipeakon solutions but since this
example is well-known, we refer to, e.g., \cite{HolRay:06b} where this is well described, rather
than presenting the details here. The upshot of the analysis is that the solution $u$ has to be
augmented by an additional variable in the form of a measure, denoted $\mu$, that describes the energy.  For $u_0\in H^1(\Real)$, we let $\mu=u_x^2 dx$. For smooth
solutions to the CH equation, the following conservation law for the energy holds
\begin{equation}
  \label{eq:conslaw}
  (u^2+u_x^2)_t + (u(u^2+u_x^2))_x = (u^3-2Pu)_x,
\end{equation}
which implies that the total energy, i.e., the $H^1(\Real)$ norm of $u$, is preserved. Here, $P$ is an
integrated term which is defined below, see \eqref{eq:CH_a}. When blow-up occurs, the energy density
$(u^2+u_x^2)\,dx$ becomes singular, that is, it becomes a measure containing a singular part. This
measure has to be augmented to the solution $u$ in order to be able to define the continuation after
blow-up.

The proper way to continue the solution after blow-up is to rewrite the equation in terms of new
variables, denoted Lagrangian variables, where the CH equation appears as a system of ordinary
differential equations taking values in a Banach space in such a way that the blow-up in the
original Eulerian variables \eqref{eq:chsys2A} evaporates
\cite{BreCons:05,BreCons:05a,HolRay:07,HolRay:09}. In the present literature the analysis has been distinct for the two
classes of solutions. Our new solution concept governed by the parameter $\alpha$ allows for a
continuous interpolation between the conservative
and dissipative solutions. At the same time it allows a uniform treatment of all cases.  We denote
these solutions as $\alpha$-dissipative solutions.

Let us describe more precisely the construction of the $\alpha$-dissipative solutions. After
applying the inverse Helmholtz operator $(1-\partial_{xx})^{-1}$, the CH equation can be rewritten as
\begin{equation} \label{eq:CH_a}
u_t+uu_x+P_x=0,\quad P-P_{xx}=u^2+\frac12 u_x^2.
\end{equation}
The pattern of  blow-up is known \cite{cons_esc1}: The solution remains continuous
while the derivative $u_x$ tends to minus infinity at the blow-up point. For this reason, the
blow-up for the CH equation is often characterized as \textit{wave breaking} and we will use this
term extensively in this paper. Wave breaking occurs precisely when the characteristics,
$y=y(t,\xi)$, given by
\begin{equation} \label{eq:chardefintro}
  y_t(t,\xi)=u(t,y(t,\xi)),
\end{equation}
have a critical point, i.e., $y_\xi(t,\xi)=0$.  For a given ``particle'', labeled by $\xi$, the
characteristic $y(t,\xi)$ denotes the trajectory of $\xi$ and
\begin{equation*}
  \tau_1(\xi)=\begin{cases}
    \sup \{t\in \Real_+\mid y_\xi(t^\prime,\xi)>0 \text{ for all } 0< t^\prime <t\},\quad & \text{ if $\{\dots\}\neq\emptyset$}, \\
    \infty, \quad &  \text{ otherwise},
  \end{cases}
\end{equation*}
denotes the time of the first wave breaking for $\xi$. For dissipative solutions, we would set
$y_\xi(t,\xi)=0$ for $t>\tau_1(\xi)$ while for conservative solutions we would continue to use
\eqref{eq:chardefintro}. Typically in a collision taking place at time $t_c$, the trajectories of
different particles meet, say $y(t_c,\xi_1)=y(t_c,\xi)=y(t_c,\xi_2)$ for $\xi\in[\xi_1, \xi_2]$. In
the dissipative case, the particles remain together. The energy which, in the case of conservative
solutions sends  the collided particles  apart, is entirely dissipated in the dissipative case. To keep
track of the part of the energy that accumulates at collision points, we introduce the function
\begin{equation*}
  h(t,\xi)=u_x^2(t,y(t,\xi))y_\xi(t,\xi).
\end{equation*}
The time evolution of $h$ is given by
\begin{equation*}
  h_t(t,\xi)=2(U^2(t,\xi)-P(t,\xi))U_\xi(t,\xi),
\end{equation*}
where the function
\begin{equation}
  \label{eq:lagvel}
  U(t,\xi)=u(t,y(t,\xi))
\end{equation}
denotes the Lagrangian velocity. We write the CH equation as a system of ordinary differential
equations in Lagrangian coordinates
\begin{equation}\label{eq:ode1}
  \begin{aligned}
    y_t& =U , & U_t& =-Q,  & h_t& =2(U^2-P)U_\xi, \\
    y_{t,\xi}&=U_\xi, &  U_{t,\xi}&= \frac12 h+(U^2-P)y_\xi, & &
  \end{aligned}
\end{equation}
where $P$ and $Q$ are integrated terms, enjoying higher regularity, given by \eqref{eq:Plag2} and
\eqref{eq:Qlag2}, respectively.  The control on the level of dissipation, which depends on $\alpha$,
is determined by the Lagrangian variables at the times of collision. At collision time $\tau_1(\xi)$, for
the particle $\xi$, we decompose $h$ in two parts
\begin{equation*}
  h(\tau_1(\xi), \xi) = \alpha h(\tau_1(\xi),\xi) + (1-\alpha) h(\tau_1(\xi),\xi).
\end{equation*}
For $\alpha$-dissipative solutions, the first part is dissipated while the second is redistributed
to the system.  We introduce $\bar h$, which denotes the effective part of the energy, that is, the
part which effectively amounts for the energy that is left after a collision. Before the first
collision, $h$ and $\bar h$ coincide, but at collision time, $\bar h$ is discontinuous and we set
\begin{equation}
  \label{eq:discontbarh1}
  \bar h(\tau_1(\xi),\xi) = (1-\alpha)\lim_{t\uparrow \tau_1(\xi)}\bar h(t,\xi),
\end{equation}
while $h$ remains continuous in time. In fact, it should be enough only to  consider $\bar h$ instead
of $h$, however, the variable $h$, because of its time continuity property, is so useful in the proofs
that we keep it as one of the variables for the governing equations. The same particle may
experience additional   collisions later. Thus, we construct the sequence
\begin{equation*}
  0< \tau_1(\xi)<\tau_2(\xi)<\dots <\tau_j(\xi)<\cdots
\end{equation*}
of collision times. For a given $\xi$, the sequence $\tau_i(\xi)$ does not accumulate and there
exists a lower bound for the time separating two collisions, see Corollary \ref{cor:tau}. At each
$\tau_j(\xi)$ we reset $\bar h$, i.e.,
\begin{equation}
  \label{eq:discbarh}
  \bar h(\tau_j(\xi),\xi) = (1-\alpha)\lim_{t\uparrow \tau_j(\xi)}\bar h(t,\xi).
\end{equation}
The equations in Lagrangian coordinates we will consider are given by
\begin{subequations}\label{eq:sysdiss2}
  \begin{align}
    y_t& =U, \\ U_t&=-Q,\\
    y_{t,\xi}&=U_\xi,\\
    U_{t,\xi}&= \frac12 \bar h+(U^2-P)y_\xi,\\
    h_t& =2(U^2-P)U_\xi, 
  \end{align}
\end{subequations}
where $P$ and $Q$ are given by \eqref{eq:Plag1} and \eqref{eq:Qlag1}, respectively.  The initial characteristics are given by 
$y(\xi)=\sup\left\{y \mid \mu((-\infty,y))+y<\xi\right\}$. Note that, since
$\bar h$ is discontinuous, the system of ordinary differential equations \eqref{eq:sysdiss2} is
discontinuous.

Now we want to obtain a global solution of the system \eqref{eq:sysdiss2}, properly formulated. We
consider the vector $\Theta=(\zeta,U,\zeta_\xi,U_\xi, \bar h , h)\in L^\infty(\Real)\times E^5$
where $E=L^2(\Real)\cap L^\infty(\Real)$, and, for technical reasons, we prefer to work with
$\zeta=y-\id$.  In order to obtain a global solution that respects the intrinsic structure of the system, we have to restrict the initial data
appropriately, and we only consider initial data in the set $\G$ given by Definition
\ref{def:G}. Short time existence is proved by an iteration argument (see Theorem \ref{th:short0}),
and existence of a global solution in  $\G$, is proved in Theorem \ref{th:global}.

The next task is then to return to Eulerian coordinates where the solution $(u(t),\mu(t))$ for each
positive time $t$ satisfies $u(t)\in H^1(\Real)$, as well as being a weak, global solution of
\eqref{eq:CH_a}, and $\mu(t)$ is a nonnegative Radon measure such that
$\muac(t)=u_x^2(t,\dott)\,dx$.  When $u$ is a smooth solution, $\mu=\muac$ but, at a blow-up time
$t_c$, the singular part of $\mu$, which we denote $\mus$, amounts for the singular part of the
energy, as we have
\begin{equation*}
  \lim_{t\uparrow t_c}\big((u^2(t,x)+u_x^2(t,x))\,dx\big) = \mus(t_c) +  (u^2(t_c,x)+u_x^2(t_c,x))\,dx.
\end{equation*}

The next problem is that of \emph{relabeling}; there are several
distinct Lagrangian solutions corresponding to one and the same solution in Eulerian variables,
similar to the fact that there are several distinct parametrizations of one and the same curve. We
identify the precise set $G$ of relabeling functions, see Definition \ref{def:Grelab}, and we show that the flow respects the relabeling, see Theorem \ref{th:sgS}.
The return to Eulerian variables is contained in Definition \ref{th:umudef}, where we
define\footnote{We denote the push-forward of the measure $d\sigma$ by the function $y$ as
  $\nu=y_\#(d\sigma)$ where $\nu(A)=\sigma\big(y^{-1}(A)\big)$.}
\begin{align*}
  u(x)&=U(\xi)\text{ for any }\xi\text{ such that  }  x=y(\xi),\\
  \mu&=y_\#(\bar h(\xi)\,d\xi).
\end{align*}
Finally, we show that the solution is a global weak solution of the CH equation, and that we have (see Theorem \ref{th:energi})
\begin{equation}
  (u^2+\mu)_t+(u(u^2+\mu))_x\le(u^3-2Pu)_x
\end{equation}
in the sense of distributions.

Until now, we have focused on the CH equation, that is, the case where $\rho(t,x)=\rho_0(x)=0$ which
implies that \eqref{eq:chsys21A} and \eqref{eq:chsys22A} decouple. For the 2CH system in the general
case, when $\rho_0 \neq 0$, we observe the same regularisation properties as in the conservative
case presented in \cite{GHR4}, namely that, if $\rho_0(x)>0$ for all $x$, then the solution retains the same
level of regularity as the one it has initially, no collision occurs and
\begin{equation}
  \label{eq:consenerg}
  E(t) = E(0)
\end{equation}
for all times $t$, where $E(t)=\sqrt{\norm{u(t,\cdot)}_{H^1}^2+\norm{\rho(t,\cdot)}_{L^2}^2}$. For general initial data, if $\alpha = 0$, the identity \eqref{eq:consenerg} holds only for \emph{almost} every time $t$
while, if $\alpha>0$, the function $E(t)$ is then non-increasing almost
everywhere, that is,
\begin{equation}
  \label{eq:dissenerg}
  E(t)\leq E(t')
\end{equation}
for $t>t'$ where $t$ and $t'$ belong to a given set of full measure, see Theorems \ref{th:energi}
and \ref{thm:find_alpha}.

\medskip Finally, we present in Section \ref{sec:AP} detailed calculations for the explicit example
of a peakon-antipeakon solution. Here one can see the interplay between Eulerian and Lagrangian
variables, the role and use of relabeling, as well as an explicit description of the behavior at
wave breaking.

\section{Lagrangian setting}\label{sec:lag}

We consider the Cauchy problem for the two component Camassa--Holm system with arbitrary 
$\kappa\in\Real$ and $\eta\in(0,\infty)$, given by
\begin{subequations}
  \label{eq:chsys2}
  \begin{align}
    \label{eq:chsys21}
    u_t-u_{txx}+\kappa u_x+3uu_x-2u_xu_{xx}-uu_{xxx}+\eta\rho\rho_x&=0,\\ \label{eq:chsys22}
    \rho_t+(u\rho)_x&=0, 
  \end{align}
\end{subequations}
with initial data $u|_{t=0}=u_0$ and $\rho|_{t=0}=\rho_0$, such that $u\in H^1(\Real)$ and $\rho\in L^2(\Real)$. A close look reveals that, if $(u(t,x),\rho(t,x))$ is a solution of the two-component Camassa--Holm system \eqref{eq:chsys2}, then we easily find that 
\begin{equation}
  v(t,x)=u(t,x),\quad \text{ and }\quad \tau(t,x)=\sqrt{\eta}\rho(t,x), 
\end{equation}
solves the two-component Camassa--Holm system with $\eta=1$. Therefore, without loss of generality,
we assume in what follows that $\eta=1$. Our analysis does not extend to the case with $\eta$ negative. For results in that 
case, see, e.g., \cite{eschlechyin:07}. In addition, we
only consider the case $\kappa=0$ as one can make
the same conclusions for $\kappa\neq0$ with
slight modifications.\footnote{The general case with $\kappa\in\Real$, which is related to the case where the solution $u,\rho$ has non-vanishing asymptotics, is treated in \cite{GHR3,GHR4,GHR5}.}

In the remainder of this section we will introduce the set of Lagrangian coordinates we want to work with and the corresponding Banach space.

\subsection{Reformulation of the 2CH system in Lagrangian coordinates}

The 2CH system with $\kappa=0$ can be rewritten as the following system in Eulerian coordinates\footnote{For $\kappa$ nonzero \eqref{eq:ch2} is simply 
  replaced by $P-P_{xx}=u^2+\kappa u+\frac12 u_x^2+\frac12 \rho^2$.} 
\begin{subequations}\label{eq:ch}
  \begin{align}\label{eq:ch1}
    u_t+uu_x +P_x&=0,\\ \label{eq:ch3}
    \rho_t+(u\rho)_x&=0,\\ \label{eq:ch2}
    P-P_{xx}&=u^2+\frac12 u_x^2+\frac12\rho^2,
  \end{align}
\end{subequations}
where $P$ and $P_x$ are given by 
\begin{equation}\label{eq:repPeul}
  P(t,x)=\frac12\int_\Real e^{-\vert x-z\vert}\big(u^2+\frac12 u_x^2+\frac12 \rho^2\big)(t,z)dz,
\end{equation}
and
\begin{equation}\label{eq:Pxeul}
  P_x(t,x)=-\frac12\int_\Real \sgn{(x-z)}e^{-\vert x-z\vert}\big(u^2+\frac12 u_x^2+\frac12 \rho^2\big)(t,z) dz.
\end{equation}

In order to reformulate the system \eqref{eq:ch} in Lagrangian variables we define the characteristics $y(t,\xi)$ as the solution of 
\begin{equation}\label{eq:chardef}
  y_t(t,\xi)=u(t,y(t,\xi))
\end{equation}
for a given $y(0,\xi)$. The Lagrangian velocity is given by $U(t,\xi)=u(t,y(t,\xi))$ and we find using \eqref{eq:ch1} that 
\begin{equation}
  U_t(t,\xi)=-Q(t,\xi),
\end{equation}
where $Q(t,\xi)=P_x(t,y(t,\xi))$ is given by 
\begin{equation}\label{eq:Qlag2}
  Q(t,\xi)= -\frac14 \int_\Real \sgn{(\xi-\eta)}e^{-\vert y(t,\xi)-y(t,\eta)\vert}(2U^2y_\xi+h)(t,\eta)d\eta,
\end{equation}
where
we have introduced $h=(u_x^2+\rho^2)\circ y\, y_\xi$, or
\begin{equation} \label{eq:h}
  h(t,\xi)=\big(u_x^2(t,y(t,\xi))+\rho^2(t,y(t,\xi))\big)y_\xi(t,\xi).
\end{equation}
The time evolution of $h(t,\xi)$ is given by 
\begin{equation}
  h_t(t,\xi)=2(U^2(t,\xi)-P(t,\xi))U_\xi(t,\xi), 
\end{equation}
where $P(t,\xi)=P(t,y(t,\xi))$ is given by 
\begin{equation}\label{eq:Plag2}
  P(t,\xi)=\frac14 \int_\Real e^{-\vert y(t,\xi)-y(t,\eta)\vert}(2U^2y_\xi+h)(t,\eta)d\eta. 
\end{equation}
Last, but not least, the Lagrangian density
\begin{equation}
  r(t,\xi)=\rho\circ y\, y_\xi(t,\xi)=\rho(t,y(t,\xi))y_\xi(t,\xi) \label{eq:r}
\end{equation}
is preserved  with respect to time, i.e., 
\begin{equation}\label{eq:rt}
  r_t=0,
\end{equation}
according to \eqref{eq:ch3}.

We have formally reformulated the 2CH system \eqref{eq:ch} in Eulerian coordinates  as the following system of ordinary differential equations in Lagrangian variables
\begin{subequations}\label{eq:sysdiss1}
  \begin{align}
    y_t& =U, \\
    U_t&=-Q,\\
    y_{t,\xi}&=U_\xi,\\
    U_{t,\xi}&= \frac12 h+(U^2-P)y_\xi,\\
    h_t& = 2(U^2-P)U_\xi, \\
    r_t&=0,
  \end{align}
\end{subequations}
where $P$ and $Q$ are given by  \eqref{eq:Plag2} and \eqref{eq:Qlag2}, respectively.

\subsection{The new solution concept: $\alpha$-dissipative solutions}

Wave breaking for the 2CH system means that $u_x$ becomes pointwise unbounded from below, which is equivalent, in this case, to saying that $y_{\xi}$ becomes zero. Let therefore $\tau_1(\xi)$ denote the first time when $y_\xi(t,\xi)$ vanishes at the point $\xi$, i.e.,
\begin{equation}\label{eq:taudef}
  \tau_1(\xi)=\sup \{t\in \Real_+\mid y_\xi(t^\prime,\xi)>0 \text{ for all } 0< t^\prime <t\}
\end{equation}
if there exists some $t>0$ such that $y_\xi(t^\prime,\xi)>0$ for all $t'\in(0,t)$ and $y_\xi(t,\xi)=0$. 
Otherwise we set $\tau_1(\xi)=\infty$. For conservative solutions we would continue $y_\xi(t,\xi)$ past wave
breaking according to the definition \eqref{eq:chardef}, while for dissipative solutions one sets
$y(t,\xi)$ constant in $\xi$ (not in time), i.e., $y_\xi(t,\xi)=0$, after wave breaking. It turns
out that the proper way to interpolate between the two solutions is by using the variable $h(t,\xi)$
given by \eqref{eq:h}. 
For $\alpha\in[0,1]$, we extend the solution past wave breaking by instantaneously  reducing the function $h(t,\xi)$ by a factor $(1-\alpha)$ at wave breaking. More precisely, we 
introduce an extra energy variable, $\bar h$, which corresponds to the
energy which is actually contained in the system and which coincides with $h$ until wave breaking occurs for the first time. At each collision, $\bar h$ is going to be discontinuous in time (for $\alpha>0$) as we set\footnote{We use the notation $\Phi(x\pm0)=\lim_{\varepsilon\downarrow0}\Phi(x\pm\varepsilon)$.}
\begin{equation*}
  \bar h(\tau_1(\xi), \xi) = (1 - \alpha) \bar h (\tau_1(\xi)-0, \xi).
\end{equation*}
The energy variable $h$
remains continuous in time as we set
\begin{equation*}
  h(\tau_1(\xi), \xi) = h (\tau_1(\xi)-0, \xi).
\end{equation*}
We define by induction the times $\tau_n(\xi)$, for $\xi$ fixed, where collisions occur. Let
\begin{equation}\label{eq:taudef_rec}
  \tau_n(\xi) = \sup \{t\in (\tau_{n-1}(\xi),\infty) \mid y_\xi(t^\prime,\xi)>0 \text{ for all } 
  \tau_{n-1}(\xi)< t^\prime < t\},
\end{equation}
if there exists some $t>\tau_{n-1}(\xi)$ such that $y_\xi(t^\prime,\xi)>0$ for all $t'\in(\tau_{n-1}(\xi),t)$ and $y_\xi(t,\xi)=0$. We set $\tau_n(\xi)=\infty$ otherwise. For convenience we let $\tau_0(\xi)=0$ for all $\xi\in\Real$.  Then, as above, we impose 
\begin{equation}\label{eq:alpdiss}
  \bar h(\tau_n(\xi), \xi) = (1 - \alpha) \bar h (\tau_n(\xi)-0, \xi)\quad
  \text{ and }\quad  
  h(\tau_n(\xi), \xi) = h (\tau_n(\xi)-0, \xi).
\end{equation}
We denote by $l_j$ the change in $\bar h$ due to the collision, that is,
\begin{equation}
  \label{eq:deflj}
  l_j(\xi) = \bar h (\tau_j(\xi)-0, \xi)-\bar h(\tau_j(\xi), \xi) = \alpha \bar h (\tau_j(\xi)-0, \xi).
\end{equation}

\begin{remark}
  The sequence $\tau_n(\xi)$ is increasing and  can a priori accumulate. However, we will show
  that this does not happen, see Corollary~\ref{cor:tau}. 
\end{remark}

\begin{definition}\label{def:sol}  An $\alpha$-dissipative solution in Lagrangian coordinates is given by the
  functions $(y, U,y_\xi,U_\xi, \bar h, h, r)$ such that
  \begin{equation*} y-\id \in L^\infty([0,T], W^{1,\infty}(\Real)),\ U \in L^\infty([0,T], H^1(\Real)),
  \end{equation*}
\begin{equation*} y_\xi - 1,\ U_\xi, \ r \in W^{1,\infty}([0,T], L^2(\Real)\cap L^\infty(\Real))
  \end{equation*} 
\begin{equation*}
\bar h\in L^\infty([0,T], L^2(\Real)),\ h\in W^{1,\infty}([0,T], L^1(\Real)\cap L^\infty(\Real))
\end{equation*}
  and measurable functions $\tau_1(\xi)<\tau_2(\xi)<\dots$,  either finitely many or $\tau_n(\xi)\to\infty$ as $n\to\infty$, given by \eqref{eq:taudef} and \eqref{eq:taudef_rec}, which satisfy, for almost every $\xi\in\Real$,
  \begin{subequations}\label{eq:sysdiss}
    \begin{align} 
    y_t(t,\xi)& =U(t,\xi), \\
    U_t(t,\xi)&=-Q(t,\xi),\\
      y_{t,\xi}(t,\xi)&=U_\xi(t,\xi),\\ U_{t,\xi}(t,\xi)&= \frac12 \bar
      h(t,\xi)+(U^2(t,\xi)-P(t,\xi))y_\xi(t,\xi),\\
       \label{eq:sysdiss-4}
      h_t(t,\xi)& =2(U^2(t,\xi)-P(t,\xi))U_\xi(t,\xi), \\
       \label{eq:sysdiss-5}
      \bar h_t(t,\xi)& =h_t(t,\xi),\\ r_t(t,\xi)&=0, \\
   \intertext{for $t\in[\tau_{n-1}(\xi), \tau_{n}(\xi))$ and}
    \label{eq:defderatjump}
    X(\tau_n(\xi),\xi) &= X(\tau_n(\xi)-0,\xi), \\
    \bar{{h}}(\tau_n(\xi),\xi) &=(1- \alpha)\bar{{h}}(\tau_n(\xi)-0,\xi), \label{eq:defderatjumpA}
  \end{align}
 \end{subequations}
  for $X=(y,  U,  y_\xi,  U_\xi, h, r)$. In
  \eqref{eq:sysdiss}, the functions $P$ and $Q$ are given by 
  \begin{align}\label{eq:Plag1} 
  P(t,\xi)&=\frac14 \int_{\Real} e^{-\vert
      y(t,\xi)-y(t,\eta)\vert}(2U^2y_\xi+\bar h)(t,\eta)d\eta  \\
  \intertext{and}
 \label{eq:Qlag1} 
 Q(t,\xi)&=-\frac14 \int_{\Real} \sgn{(\xi-\eta)}e^{-\vert
      y(t,\xi)-y(t,\eta)\vert}(2U^2y_\xi+\bar h)(t,\eta)d\eta,
  \end{align} 
  respectively.
\end{definition}

\begin{remark}
 Note that due to the above considerations, we can represent $\bar h(t,\xi)$ in the following way
\begin{equation}\label{eq:tautau}
  \bar h(t,\xi)=
    h(t,\xi)-\sum_{j=0}^n l_j(\xi), \quad \text{for }t\in[\tau_n(\xi),\tau_{n+1}(\xi)), 
\end{equation}
where we recursively define $l_j(\xi)=\alpha\big(h(\tau_j(\xi),\xi)-\sum_{k=0}^{j-1} l_k(\xi)\big)$ for $j\in\Natural$, and $ l_0(\xi)= h(0,\xi)-\bar h(0,\xi)\geq 0$ and 
$\tau_0(\xi)=0$. In particular, we have $0\leq\bar h(t,\xi)\leq h(t,\xi)$.
\end{remark}

\begin{remark} We will here try to explain the strategy behind the lengthy existence proof in Lagrangian variables.  Our starting point is the formulation  \eqref{eq:sysdiss1} in Lagrangian variables. We replace the mixed derivatives $y_{t,\xi}$ and  $U_{t,\xi}$ by new variables, namely $q=y_\xi$ and $w=U_\xi$, which turns \eqref{eq:sysdiss1} into a system of ordinary differential equations. We show the existence of a solution by an iterative argument, as part of  the proof of Theorem \ref{th:short0}. To secure a global solution and to make sure that the underlying structure is preserved, e.g., that the functions $q$ and $w$ satisfy $q=y_\xi$ and $w=U_\xi$, respectively, we have to restrict the set of initial data to the set $\G$, cf., Definition \ref{def:G}. The existence of global solutions then follows in the standard way by showing that  the solution remains bounded.  This would then yield the solution in Lagrangian variables in the conservative case. However, to construct the $\alpha$-dissipative solutions we need to monitor  $y_{\xi}(t,\xi)$ carefully  as a function of $t$ for each fixed $\xi$. At the first occasion when $y_{\xi}(t,\xi)=0$, that is, when $t=\tau_1(\xi)$, we read off the values of the dependent variables, and scale the variable $\bar h$ (which equals $h$ up to $\tau_1(\xi)$) by the factor $1-\alpha$. The system of ordinary differential equations is then restarted at $t=\tau_1(\xi)$ and runs according to \eqref{eq:sysdiss} until the next time $y_{\xi}(t,\xi)$ vanishes. Again the function  $\bar h$ is rescaled, and the system restarted.  This construction is performed for each $\xi\in \Real$. As the system of ordinary differential equations is discontinuous, the global existence proof requires careful estimates, see Lemmas \ref{lem:2.3}, \ref{lem:breakN},  \ref{lem:PQ}, \ref{lem:G}--\ref{lem:contrgamma}, \ref{lem:globest}.  

The function $g$, introduced below in Definition \ref{def:Omega}, plays a subtle role in our considerations. It is used in Lemma~\ref{lem:G}, when identifying $\kappa_{1-\gamma}$, cf.~\eqref{eq:defKgamma}, as the set of points which will experience wave breaking in the near future. However, it will play an even more vital role in the (future) construction of a Lipschitz metric for this system, see, e.g., \cite{BHR, GHRb:10}. A close look at $g$ and $\bar h$ reveals that the function $\bar h$ drops suddenly at breaking time while the function $g$ models the loss of energy in a continuous way. Thus $g$ will play a major role in (future) investigations about the stability of solutions.

\end{remark}

We introduce the following notation for the Banach spaces that are frequently used.  Let
\begin{equation*}
  E= L^2(\Real)\cap L^\infty(\Real), 
\end{equation*}
together with the norm
\begin{equation*}
  \norm{f}_{E}=\norm{f}_{L^2}+\norm{f}_{L^\infty},
\end{equation*}
and let
\begin{align*}
  W&=\left[ L^2(\Real)\right]^4,&
  \bar W&= E^4,\\
  V&= L^\infty(\Real) \times L^2(\Real)\times W,& \bar V& = L^\infty(\Real)\times E\times \bar W.
\end{align*}
For any function $f\in C([0,T], B)$ for $T\geq 0$ and $B$ a normed space, we denote
\begin{equation*}
  \norm{f}_{L_T^1B}=\int_0^T\norm{f(t,\dott)}_B dt\quad \text{ and }\quad \norm{f}_{L_T^\infty B}=\sup_{t\in[0,T]}\norm{f(t,\dott)}_B.
\end{equation*}
\begin{definition}\label{def:Omega}
  For $x=(x_1,\dots,x_7)\in \Real^7$, we define the functions $g_1,g_2,g\colon\Real^7\to \Real$ by
   \begin{align*}
    g_1(x)&=\vert x_4\vert +2x_3,\\
    g_2(x)&= x_3+x_5,
  \end{align*}
  and 
  \begin{equation}\label{eq:defg}
    g(x)=\begin{cases} 
      \alpha g_1(x)+(1-\alpha)g_2(x), &\quad \text{if $x\in \Omega_1$,}\\
      g_2(x), & \quad \text{otherwise,} \\
    \end{cases}
  \end{equation}
  where  $\Omega_1$ is the set where $g_1\leq g_2$, $x_4$ is nonpositive, and $x_7=0$, thus
  \begin{equation*}
    \Omega_1=\{ x\in \Real^7\mid \vert x_4\vert+2x_3\leq x_3+x_5 \text{, } x_4\leq 0, \text{ and } x_7=0\}.
  \end{equation*}
 \end{definition}

We  identify $x=(x_1,\dots,x_7)$ with $\Theta=(y, U,y_\xi, U_\xi,\bar h, h, r)$.

\begin{remark} \label{rem:g}
  In the case of conservative solutions, i.e., $\alpha=0$, we have $0<g(\Theta)(t,\xi)=g_2(\Theta)(t,\xi)$ and
  $h(t,\xi)=\bar h(t,\xi)$ for all $\xi\in\Real$ and $t\in\Real$. In the case of dissipative
  solutions, i.e., $\alpha=1$, we infer $0<g(\Theta)(t,\xi)$ and $h(t,\xi)=\bar h(t,\xi)$ before wave
  breaking, while $0=g(\Theta)(t,\xi)$ and $\bar h(t,\xi)=0$ thereafter. The function $g(\Theta)(t,\xi)$ is
  introduced in such a way that it describes the loss of energy in a continuous way, in contrast to
  $\bar h(t,\xi)$,  which drops suddenly at wave breaking.
\end{remark}

\begin{definition} \label{def:G} The set $\G$ consists of all $\Theta=(y,U, y_\xi,U_\xi,\bar h, h, r)$ such that
  \begin{subequations}\label{eq:lagcoord}
    \begin{align}
      \label{eq:lagcoord1}
      &X=(\zeta, U,\zeta_\xi, U_\xi, h, r)\in \bar V,\\
      \label{eq:lagcoord2}
      &g(\Theta)-1\in E,\\  
      \label{eq:lagcoord8}
      & h\in L^1(\Real),\\
      \label{eq:lagcoord3}
      &y_\xi\geq 0, \quad h\geq 0, \quad \bar h\geq 0 \text{  almost everywhere}, \\
      \label{eq:lagcoord4}
      &\lim_{\xi\to -\infty} \zeta(\xi)=0,\\
      \label{eq:lagcoord5}
      &\frac{1}{y_\xi+h}\in L^\infty(\Real),\\ 
      \label{eq:lagcoord6}
      &y_\xi \bar h=U_\xi^2+r^2 \text{ almost everywhere},\\ 
      \label{eq:lagcoord7}
      & h\geq \bar h \text{ almost everywhere},
    \end{align}
  \end{subequations}
  where we denote $y(\xi)=\zeta(\xi)+\xi$.
\end{definition}
The condition \eqref{eq:lagcoord4} will be valid as long as the solution exists since in that case
we must have $\lim_{\xi\to-\infty}U(t,\xi)=0$ by construction. In addition, it should be noted that,
due to the definition of $g(\Theta)$, the relation \eqref{eq:lagcoord2} is valid for any $\Theta$ that
satisfies \eqref{eq:lagcoord1} since $0\leq \bar h\leq h$. 

Making the identifications $q=y_\xi$ and $w=U_\xi$, we obtain
\begin{subequations}\label{eq:ODEsys}
  \begin{align}
    y_t& =U, \\ 
    U_t&=-Q(\Theta),\\
    q_{t}&=w,\\
    w_{t}&= \frac12 \bar h+(U^2-P(\Theta))q,\\
    h_t& =2(U^2-P(\Theta))w,\\
    r_t&=0,
  \end{align}
\end{subequations}
where $P(\Theta)$ and $Q(\Theta)$ are given by 
\begin{align}\label{eq:Plag3} 
  P(t,\xi)&=\frac14 \int_{\Real} e^{-\vert y(t,\xi)-y(t,\eta)\vert}(2U^2q+\bar h)(t,\eta)d\eta, \\
\intertext{and}
\label{eq:Qlag3}
  Q(t,\xi)&=-\frac14 \int_{\Real} \sgn{(\xi-\eta)}e^{-\vert y(t,\xi)-y(t,\eta)\vert}(2U^2q+\bar h)(t,\eta)d\eta,
\end{align}
respectively.

The definition of $\tau_1$ given by \eqref{eq:taudef} (after replacing $y_\xi$ by the corresponding
variable $q$) is not appropriate for $q\in C([0,T],\Linf)$, and, in addition, it is not clear from
this definition if $\tau_1$ is measurable. Thus we replace this definition by the following one. Let
$\{t_i\}_{i=1}^\infty$ be a dense countable subset of $[0,T]$.  Define
\begin{equation*}
  A_t=\bigcup_{n\in \Natural}\bigcap_{t_i\leq
    t}\Big\{\xi\in\Real\mid q(t_i,\xi)>\frac1n\Big\}. 
\end{equation*}
The sets $A_t$ are measurable for all $t$, and we have $A_{t'}\subset A_t$ for $t\leq t'$. We
consider a dyadic partition of the interval $[0,T]$ (that is, for each $n$, we consider the set
$\{2^{-n}iT\}_{i=0}^{2^n}$) and set
\begin{equation*}
  \tau^n_1(\xi)=\sum_{i=0}^{2^n}\frac{iT}{2^n}\chi_{i,n}(\xi),
\end{equation*}
where $\chi_{i,n}$ is the indicator function of the set $A_{2^{-n}iT}\setminus A_{2^{-n}(i+1)T}$.
 The function $\tau^n_1$ is by construction
measurable. One can check that $\tau^n_1(\xi)$ is increasing with respect to $n$, it is also bounded
by $T$. Hence, we can define
\begin{equation*}
  \tau_1(\xi)=\lim_{n\to\infty}\tau^n_1(\xi),
\end{equation*}
and $\tau_1$ is a measurable function. The next lemma gives the main property of $\tau_1$.

\begin{lemma}\label{lem:2.3}
  If, for every
  $\xi\in\Real$, $q(t,\xi)$ is positive and continuous with
  respect to time, then 
  \begin{equation}
    \label{eq:sectau}
    \tau_1(\xi)=
\begin{cases}
\sup\{t\in\Real^+\mid q(t',\xi)>0\text{ for all }0< t'<t\}, & \text{ if $\{\dots\}\neq\emptyset$}, \\
\infty, & \text{otherwise}. 
\end{cases}
  \end{equation}
  that is, we retrieve the definition \eqref{eq:taudef}.
\end{lemma}
\begin{proof}  See \cite{HolRay:09}.
 \end{proof}

One can represent $\tau_n(\xi)$ with $n=2,3,\dots$ similarly. Indeed, let $\{t_i\}_{i=1}^\infty$ be
a dense countable subset of $[0,T]$. Define inductively
\begin{equation*}
  A_{n,t}=\bigcup_{m\in\Natural}\bigcap_{t_i\leq
    t}\Big\{\xi\in\Real\mid \tau_{n-1}(\xi)\leq t_i, \quad q(t_i,\xi)>\frac1m\Big\}, \quad n=2, 3, \dots. 
\end{equation*}
As before,  the sets $A_{n,t}$ are measurable for all $t$, and, in particular, $A_{n,t'}\subset
A_{n,t}$ for $t\leq t'$.  We consider a dyadic partition of the interval $[0,T]$, and set
\begin{equation*}
  \tau_n^m(\xi)=\sum_{i=0}^{2^m} \frac{iT}{2^m}\chi_{i,n,m}(\xi),
\end{equation*}
where $\chi_{i,n,m}$ is the indicator function of the set $A_{n,2^{-m}iT}\setminus
A_{n,2^{-m}(i+1)T}$. The function $\tau_n^m(\xi)$ is by construction
measurable. One can check that $\tau_n^m(\xi)$ is increasing with respect to $m$ and bounded by
$T$. Hence we define
\begin{equation*}
  \tau_n(\xi)=\lim_{m\to\infty} \tau_n^m(\xi), 
\end{equation*}
and $\tau_n(\xi)$ is a measurable function. Concluding as in the proof of Lemma~\ref{lem:2.3}, one
obtains the following result.

\begin{lemma} \label{lem:breakN}
  If, for every $\xi\in\Real$, $q(t,\xi)$ is positive and continuous with respect to time, then
  \begin{equation}
    \label{eq:deftauj}
    \tau_n(\xi)=
\begin{cases}
\sup\{t\in(\tau_{n-1}(\xi),\infty)\mid q(t',\xi)>0
    \text{ for all } t'\in (\tau_{n-1}(\xi), t) \}, & \text{if $\{\dots\}\neq\emptyset$}, \\
   \infty, & \text{otherwise},
\end{cases}
  \end{equation}
 for $n=2,3,\dots$.
\end{lemma}

\begin{remark}
  In the case of conservative solutions we actually do not need to define $\tau_j(\xi)$ for
  $\xi\in\Real$ because we do not redefine our system \eqref{eq:sysdiss} after wave breaking.
\end{remark}

So far we have identified $q$ with $y_\xi$. However, $y_\xi$ does not decay fast enough at infinity
to belong to $L^2(\Real)$, but $y_\xi-1=\zeta_\xi$ will be in $L^2(\Real)$, and we therefore
introduce $v=q-1$. In the case of conservative solutions, we know that $Q(\Theta)$ and $P(\Theta)$ are
Lipschitz continuous on bounded sets and that $Q(\Theta)$ and $P(\Theta)$ can be bounded by a constant
depending on the bounded set. A slightly different result is true when describing
$\alpha$-dissipative solutions. Define
\begin{equation}\label{eq:defBMfix}
  B_M=\{ \Theta \mid \norm{X}_{\bar V}+\norm{h}_{L^1}+\norm{\frac{1}{q+h}}_{L^\infty}\leq M \text{, } q\bar h=w^2+r^2 \text{, }\bar h\leq h \text{, and } q, \bar h\geq 0\text{ a.e.}\}.
\end{equation}

In addition,  it should be pointed out that for any $\Theta\in C([0,T],B_M)$ the set of all points which
experience wave breaking within a finite time interval $[0,T]$ is bounded, since
\begin{equation}
\begin{aligned}
  \meas(\{\xi\in\Real\mid q(t,\xi)=0\})&\leq \int_\Real \frac{h}{q+h}(t,\xi)d\xi\\
 & \leq \norm{\frac{1}{q+h}}_{L^\infty_TL^\infty}\norm{h}_{L^\infty_TL^1}\leq C(M),
  \end{aligned}
\end{equation}
for all $t\in[0,T]$, where $C(M)$ denotes some constant only depending on $M$.

\begin{lemma}\label{lem:PQ}
  (i) For all $\Theta\in C([0,T],B_M)$, we have 
  \begin{equation}
    \norm{Q(\Theta)}_{L^\infty_T E}+\norm{P(\Theta)}_{L^\infty_TE}\leq C(M)
  \end{equation}
  for a constant $C(M)$ which only depends on $M$.\\
  (ii) 
  For any $\Theta$ and $\tilde \Theta$ in $C([0,T],B_M)$, we have
  \begin{align}
    \norm{Q(\Theta)-Q(\tilde \Theta)}_{L^1_T E}&+\norm{P(\Theta)-P(\tilde \Theta)}_{L^1_TE} \nn\\  \label{eq:lt1bdqp}
    & \leq C(M) \Big(T\norm{X-\tilde X}_{L^\infty_T \bar V} +\int_0^T \int_\Real\vert \bar h(t,\xi)-\bar{\tilde{h}}(t,\xi)\vert d\xi dt\Big).
  \end{align}
  Here $C(M)$ denotes a constant which only depends on $M$.
\end{lemma}

\begin{proof}
 We will only establish the estimates for $P(\Theta)$ as the ones for $Q(\Theta)$ can be obtained using the same methods with only slight modifications. The main tool for proving the stated estimates will be Young's inequality which we recall here for the sake of completeness.
  For any $f\in L^p(\Real)$ and $g\in L^q(\Real)$ with $1\leq p,q,r\leq \infty$, we have
  \begin{equation}
    \norm{f\star g}_{L^r}\leq \norm{f}_{L^p}\norm{g}_{L^q}, \quad \text{ if } \quad 1+\frac{1}{r}=\frac{1}{p}+\frac{1}{q}.
  \end{equation}

(\textit{i}): By definition we have
\begin{equation}\label{Phelp}
 P(\Theta)(t,\xi)= \frac14\int_\Real e^{-\vert y(t,\xi)-y(t,\eta)\vert} (2U^2q+\bar h)(t,\eta)d\eta.
\end{equation}
So far we do not know if $y(t,\xi)$ is an increasing function or not, thus we will split the integral above into three as follows. By assumption we have that $\norm{y(t,\xi)-\xi}_{L^\infty_T L^\infty}\leq M$, thus
\begin{equation*}
(\xi-\eta)-2M\leq y(t,\xi)-y(t,\eta)=(y(t,\xi)-\xi)+(\xi-\eta)-(y(t,\eta)-\eta)\leq (\xi-\eta)+2M,
\end{equation*}
and, in particular,
\begin{align*}
 y(t,\xi)-y(t,\eta)\geq 0 &\quad \text{if} \quad \eta\leq \xi-2M,\\ 
y(t,\xi)-y(t,\eta)\leq 0& \quad \text{if} \quad \eta\geq \xi+2M.
\end{align*}
Hence we can rewrite \eqref{Phelp} as
\begin{align*}
 P(\Theta)(t,\xi)& =\frac14 \int_{-\infty}^{\xi-2M} e^{-(y(t,\xi)-y(t,\eta))}(2U^2q+\bar h)(t,\eta)d\eta\\
& \quad +\frac14 \int_{\xi-2M}^{\xi+2M} e^{-\vert y(t,\xi)-y(t,\eta)\vert}(2U^2q+\bar h)(t,\eta)d\eta\\
&\quad +\frac14 \int_{\xi+2M}^\infty e^{-(y(t,\eta)-y(t,\xi))} (2U^2q+\bar h)(t,\eta)d\eta\\
& = I_1(t,\xi)+I_2(t,\xi)+I_3(t,\xi).
\end{align*}
Let $f(\xi)= \chi_{\{\xi>2M\}}e^{-\xi}$. Then we have 
\begin{align*}
 \norm{I_1(t,\xi)}_{L^\infty_T E}& =\norm{\frac14 \int_{-\infty}^{\xi-2M} e^{-\zeta(t,\xi)}e^{-(\xi-\eta)}e^{\zeta(t,\eta)}(2U^2q+\bar h)(t,\eta) d\eta}_{L^\infty_T E}\\ 
& \leq\frac14  e^{\norm{\zeta}_{L^\infty_TL^\infty}}\norm{(f \star[e^\zeta(2U^2q+\bar h)])(t,\xi)}_{L^\infty_T E}\\ 
& \leq C(M)( \norm{f}_{L^1}+\norm{f}_{L^2})\norm{e^\zeta(2U^2q+ \bar h)}_{L^\infty_TL^2} \\
& \leq C(M),
\end{align*}
since $\bar h(t,\xi)\leq h(t,\xi)$.
Similarly one can estimate $\norm{I_3(t,\xi)}_{L^\infty_T E}$ by replacing the function $f(\xi)$ by the function $g(\xi)= \chi_{\{\xi<-2M\}}e^{\xi}$.
As far as $I_2(t,\xi)$ is concerned, we conclude as follows
\begin{align*}
& \norm{I_2(t,\xi)}_{L^\infty_T E}\\
 &\qquad \leq \norm{\frac14\int_{\xi-2M}^{\xi+2M} e^{-\vert y(t,\xi)-y(t,\eta)\vert} (2U^2q+\bar h)(t,\eta)d\eta}_{L^\infty_T E}\\ 
&\qquad \leq \norm{\frac14\int_{\xi-2M}^{\xi+2M} (e^{-(y(t,\xi)-y(t,\eta))}+e^{-(y(t,\eta)-y(t,\xi))})(2U^2q+\bar h)(t,\eta)d\eta}_{L^\infty_T E}\\
&\qquad \leq \norm{\frac14\int_{\xi-2M}^{\xi+2M} e^{-(y(t,\xi)-y(t,\eta))}(2U^2q+\bar h)(t,\eta)d\eta}_{L^\infty_T E}\\
&\qquad \quad + \norm{\frac14\int_{\xi-2M}^{\xi+2M} e^{-(y(t,\eta)-y(t,\xi))}(2U^2q+\bar h)(t,\eta)d\eta}_{L^\infty_T E}.
\end{align*}
Following closely the argument we used for $I_1(t,\xi)$, yields
\begin{equation}
 \norm{I_2(t,\xi)}_{L^\infty_T E}\leq C(M).
\end{equation}

(\textit{ii}): As before we split the integral into three parts and investigate each of them separately. 
We start with 
\begin{align*}
B_1(t,\xi)&=\frac14 \int_{-\infty}^{\xi-2M} \big(e^{-(y(t,\xi)-y(t,\eta))}(2U^2q+\bar h)(t,\eta)\\
&\qquad \qquad\qquad\qquad\qquad\qquad-e^{-(\tilde y(t,\xi)-\tilde y(t,\eta))} (2\tilde U^2\tilde q+\bar{\tilde h})(t,\eta)\big)d\eta\\ 
& =\frac14 \Big(e^{-\zeta(t,\xi)}-e^{-\tilde\zeta(t,\xi)}\Big)\int_{-\infty}^{\xi-2M} e^{-(\xi-\eta)}e^{\zeta(t,\eta)}(2U^2q+\bar h)(t,\eta)d\eta\\ 
& \quad +\frac14  e^{-\tilde\zeta(t,\xi)} \int_{-\infty}^{\xi-2M} e^{-(\xi-\eta)}\Big(e^{\zeta(t,\eta)}2U^2q(t,\eta)-e^{\tilde\zeta(t,\eta)}2\tilde U^2\tilde q(t,\eta)\Big)d\eta\\
&\quad +\frac14  e^{-\tilde\zeta(t,\xi)} \int_{-\infty}^{\xi-2M} e^{-(\xi-\eta)}\Big(e^{\zeta(t,\eta)}\bar h(t,\eta)-e^{\tilde\zeta(t,\eta)}\bar{\tilde h}(t,\eta)\Big)d\eta.
\end{align*}
Let $f(\xi)= \chi_{\{\xi>2M\}}e^{-\xi}$, then 
\begin{align*}
 \norm{B_1(t,\xi)}_{L^1_T E}& \leq C(M)T\norm{\zeta-\tilde\zeta}_{L^\infty_T E} + C(M)T(\norm{f}_{L^1}+\norm{f}_{L^2})\norm{X-\tilde X}_{L^\infty_T V}\\
& \quad + C(M)(\norm{f}_{L^\infty}+\norm{f}_{L^2})\\
&\qquad\qquad \times\Big(T\norm{X-\tilde X}_{L^\infty_T V}+\int_0^T  \int_{\Real}\vert \bar h(t,\xi)-\bar{\tilde{h}}(t,\xi)\vert d\xi dt\Big)\\ 
& \leq C(M)\Big(T\norm{X-\tilde X}_{L^\infty_T V}+\int_0^T  \int_{\Real}\vert \bar h(t,\xi)-\bar{\tilde{h}}(t,\xi)\vert d\xi dt\Big).
\end{align*}
$B_3(t,\xi)$, which corresponds to $I_3(t,\xi)$ in (\textit{i}), can be investigated similarly. 
As far as $B_2(t,\xi)$ is concerned, we have
\begin{align*}
B_2(t,\xi)& =\frac14 \int_{\xi-2M}^{\xi+2M} \Big(e^{-\vert y(t,\xi)-y(t,\eta)\vert }(2U^2q+\bar h)(t,\eta)\\
&\qquad\qquad\qquad\qquad\qquad\qquad\qquad -e^{-\vert\tilde y(t,\xi)-\tilde y(t,\eta)\vert} (2\tilde U^2\tilde q+\bar{\tilde h})(t,\eta)\Big)d\eta\\ 
& =\frac14 \int_{\xi-2M}^{\xi+2M} \Big(e^{-\vert y(t,\xi)-y(t,\eta)\vert}-e^{-\vert \tilde y(t,\xi)-\tilde y(t,\eta)\vert}\Big)(2U^2q+\bar h)(t,\eta)d\eta\\ 
& \quad + \frac14 \int_{\xi-2M}^{\xi+2M}e^{-\vert \tilde y(t,\xi)-\tilde y(t,\eta)\vert} (2U^2q+\bar h-2\tilde U^2\tilde q-\bar {\tilde h})(t,\eta)d\eta\\ 
& =\frac14 \int_{\xi-2M}^{\xi+2M}e^{-\vert y(t,\xi)-y(t,\eta)\vert}\\
&\qquad\qquad\qquad\times\Big(1-e^{-\vert \tilde y(t,\xi)-\tilde y(t,\eta)\vert +\vert y(t,\xi)-y(t,\eta)\vert}\Big)(2U^2q+\bar h)(t,\eta)d\eta\\ \nn
& \quad +\int_{\xi-2M}^{\xi+2M}e^{-\vert \tilde y(t,\xi)-\tilde y(t,\eta)\vert}2(U^2q-\tilde U^2\tilde q)(t,\eta)d\eta\\ 
& \quad +\int_{\xi-2M}^{\xi+2M}e^{-\vert\tilde y(t,\xi)-\tilde y(t,\eta)\vert}(\bar h-\bar{\tilde h})(t,\eta)d\eta\\ 
& =B_{21}(t,\xi)+B_{22}(t,\xi)+B_{23}(t,\xi).
\end{align*}
$\norm{B_{22}(t,\xi)}_{L^1_T E}$ and $\norm{B_{23}(t,\xi)}_{L^1_T E}$ can be estimates using Young's inequality, while 
$\norm{B_{21}(t,\xi)}_{L^1_T E}$ requires more careful estimates. Since $\xi-2M\leq \eta\leq \xi+2M$, we have 
\begin{align}
 \vert \vert y(t,\xi)-y(t,\eta)\vert -\vert \tilde y(t,\xi)-\tilde y(t,\eta)\vert \vert&\leq \vert y(t,\xi)-\tilde y(t,\xi)\vert +\vert y(t,\eta)-\tilde y(t,\eta)\vert \\ \nn
& \leq 2\norm{y-\tilde y}_{L^\infty_T L^\infty}
\end{align}
and 
\begin{align}
 \vert \vert y(t,\xi)-y(t,\eta)\vert -\vert \tilde y(t,\xi)-\tilde y(t,\eta)\vert\vert
&\leq \vert y(t,\xi)-y(t,\eta)\vert +\vert \tilde y(t,\xi)-\tilde y(t,\eta)\vert\\ \nn
&\leq 4\norm{y-\id}_{L^\infty_T L^\infty}+2\vert \xi-\eta\vert  
\leq 8M.
\end{align}
Hence 
\begin{align}
 \vert 1- e^{-\vert \tilde y(t,\xi)-\tilde y(t,\eta)\vert+\vert y(t,\xi)-y(t,\eta)\vert}\vert 
& \leq \vert \int_{-\vert \tilde y(t,\xi)-\tilde y(t,\eta)\vert +\vert y(t,\xi)-y(t,\eta)\vert}^0 e^xdx\vert \\ \nn 
& \leq C(M)\norm{y-\tilde y}_{L^\infty _T L^\infty}
\end{align}
and 
\begin{align*}
& \norm{B_{21}(t,\xi)}_{L^1_T E} \\
 &\qquad\leq C(M)\norm{y-\tilde y}_{L^\infty_T L^\infty}\norm{\frac14 \int_{\xi-2M}^{\xi+2M} e^{-\vert y(t,\xi)-y(t,\eta)\vert}(2U^2q+\bar h)(t,\eta)d\eta}_{L^1_T E}\\  
&\qquad \leq C(M)T\norm{y-\tilde y}_{L^\infty_T L^\infty}\Big(\norm{\frac14\int_{\xi-2M}^{\xi+2M}e^{-(y(t,\xi)-y(t,\eta))}(2U^2q+\bar h)(t,\eta)d\eta}_{L^\infty_T E}\\ 
&\qquad \qquad +\norm{\frac14\int_{\xi-2M}^{\xi+2M}e^{-(y(t,\eta)-y(t,\xi))} (2U^2q+\bar h)(t,\eta)d\eta}_{L^\infty_T E}\Big)\\  
&\qquad \leq C(M)T \norm{y-\tilde y}_{L^\infty_T L^\infty}. 
\end{align*}
Thus putting everything together, we have
\begin{equation}
 \norm{B_{2}(t,\xi)}_{L^1_T E}\leq C(M)\Big(T\norm{X-\tilde X}_{L^\infty_T V}+\int_0^T  \int_{\Real}\vert \bar h(t,\xi)-\bar{\tilde{h}}(t,\xi)\vert d\xi dt\Big). 
\end{equation}
\end{proof}

\begin{remark}
  (\textit{i}): In the case of conservative solutions, i.e., $\alpha=0$, we have $h(t,\xi)=\bar h(t,\xi)$ and hence 
  \begin{equation*}
    \int_0^T  \int_{\Real}\vert \bar h(t,\xi)-\bar{\tilde{h}}(t,\xi)\vert d\xi dt\leq TC(M)\norm{X-\tilde X}_{L^\infty_T \bar V}, 
  \end{equation*}
  after using that $h=U_\xi^2+r^2-h\zeta_\xi$ together with the Cauchy--Schwarz inequality.\\
  \noindent
  (\textit{ii}): In the case of dissipative solutions, i.e., $\alpha=1$, we get, since $\bar h(t,\xi)=0$ for $t\geq \tau_1(\xi)$, that 
  \begin{align*}
    \int_0^T  \int_{\Real}& \vert \bar h(t,\xi)-\bar{\tilde{h}}(t,\xi)\vert d\xi dt\\ \nn 
    & \leq C(M)\Big(T\norm{X-\tilde X}_{L^\infty_T \bar V}\\
    &\quad+\int_\Real\big(\int_{\tau_1}^{\tilde\tau_1} \tilde h(t,\xi)\chi_{\{\tilde\tau_1>\tau_1\}}(\xi)dt +\int_{\tilde\tau_1}^{\tau_1} h(t,\xi)\chi_{\{\tau_1>\tilde\tau_1\}}(\xi)dt\big)d\xi\Big).
  \end{align*}
  Here we used the same argument as in (\textit{i}) together with an application of Fubini's
  theorem. In particular, this means that the norm estimates here imply the ones in \cite{GHR5},
  where the dissipative case is studied, and vice versa.
\end{remark}

To show short-time existence of solutions we will use an iteration argument for the following system
of ordinary differential equations.  Denote generically $(\zeta, U,q,w,\bar h,h,r)$ by $\Theta$, $(\zeta, U,q,w,h,r)$ by $X$, and $(q,w,h,r)$
by $Z$, thus $X=(\zeta,U,Z)$. 
Then, we define the mapping
\begin{equation*}
  \mathcal{P}\colon C([0,T], B_M) \to C([0,T], B_M)
\end{equation*}
as follows: Given $\Theta_0\in \mathcal{G}\cap B_{M_0}$ and $\Theta\in C([0,T], B_M)$, we can compute $P(\Theta)$ and $Q(\Theta)$ using \eqref{eq:Plag3}
and \eqref{eq:Qlag3}. Then, we define $\tilde \Theta=\mathcal{P}(\Theta)$ as follows. Given $\xi\in\Real$, we
set $\tilde \Theta(0,\xi)=\Theta_0(\xi)$ and $\tilde \Theta(t,\xi)$ on $[\tilde\tau_n(\xi), \tilde\tau_{n+1}(\xi)]$ as the solution of the
system of ordinary differential equations
\begin{subequations}\label{eq:sysfix2}
  \begin{align}
    \tilde \zeta_t(t,\xi)&=\tilde U(t,\xi), \\ 
    \tilde U_t(t,\xi)&=-Q(\Theta)(t,\xi),  \label{eq:sysfix2A}\\
    \tilde q_t(t,\xi)&=\tilde w(t,\xi),\\
    \tilde w_t(t,\xi)&= \frac12\bar{\tilde{h}}(t,\xi)+(U^2(t,\xi)-P(\Theta)(t,\xi))\tilde q(t,\xi),\\
    \tilde h_t(t,\xi)& = 2(U^2(t,\xi)-P(\Theta)(t,\xi))\tilde w(t,\xi), \\
    \bar{\tilde{h}}_t(t,\xi)& = \tilde h_t(t,\xi), \\
    \tilde r_t(t,\xi)&=0,
  \end{align}
\end{subequations}
which satisfies, at $t=\tilde\tau_n(\xi)$,
\begin{equation}
  \label{eq:defderatjump2}
  \tilde X(\tilde \tau_n(\xi),\xi) = \tilde X(\tilde\tau_n(\xi)-0,\xi)
  \quad  \text{ and }\quad
  \bar{\tilde{h}}(\tilde\tau_n(\xi),\xi) = (1-\alpha)\bar{\tilde{h}}(\tilde\tau_n(\xi)-0,\xi).
\end{equation}
We write
$\bar{\tilde {Z}}_t=F(\Theta)\bar{\tilde {Z}}$, where $\bar{\tilde{Z}}=(\tilde q,\tilde w, \bar{\tilde{h}},
\tilde{r})$ for all times $t$, where no wave breaking occurs, i.e., for $t\in [\tilde\tau_n(\xi),
\tilde\tau_{n+1}(\xi))$. 
So far we have not excluded that the sequence $\tilde\tau_n(\xi)$ might have an accumulation point $\tilde\tau_\infty(\xi)$. Later on we will see that this is not possible, see Lemma~\ref{lem:break}. If the sequence $\tilde\tau_n(\xi)$ were to have an accumulation point  $\tilde\tau_\infty(\xi)$, we define $\tilde \Theta$ as the
solution of
\begin{align*} \tilde y_t(t,\xi)& =\tilde U(t,\xi) , \quad \tilde U_t(t,\xi)=-Q(t,\xi),\\
  \tilde q_{t}(t,\xi)&= \tilde w_{t}(t,\xi)= \bar{\tilde{h}}_t(t,\xi) = \tilde r_t(t,\xi) =0,\\
  \tilde h(t,\xi)&=\tilde h(\tilde\tau_\infty(\xi),\xi),
\end{align*}
for $t\in[\tilde\tau_\infty(\xi),T]$.  

The following set will play a key role
in the context of wave
breaking, since it contains all points which will experience wave breaking in the near future,
\begin{equation}\label{eq:defKgamma}
  \kappa_{1-\gamma}=\{\xi\in\Real\mid \frac{\bar h_0}{q_0+\bar h_0}(\xi)\geq 1-\gamma\text{, } w_0(\xi)\leq 0, \text{ and } r_0(\xi)=0\}, \quad \gamma\in [0,\frac12].
\end{equation}
Note that 
\begin{align*}
  \frac{\bar h_0}{q_0+\bar h_0}(\xi)\geq 1-\gamma & \Longleftrightarrow \gamma\geq 1-\frac{\bar h_0}{q_0+\bar h_0}(\xi)=\frac{q_0}{q_0+\bar h_0}(\xi)\\ \nn 
  & \Longleftrightarrow (1-\gamma)q_0(\xi)\leq \gamma \bar h_0(\xi)\leq \gamma h_0(\xi),
\end{align*}
which implies that $\frac{q_0}{q_0+h_0}(\xi)\leq \gamma$, and hence $\frac{h_0}{q_0+h_0}(\xi)\geq 1-\gamma$.
In particular, we have that 
\begin{equation}
  \meas(\kappa_{1-\gamma})\leq \frac{1}{1-\gamma}\int_\Real \frac{h_0}{h_0+q_0}(\xi)d\xi\leq \frac{1}{1-\gamma}\norm{\frac{1}{q_0+h_0}}_{L^\infty}\norm{h_0}_{L^1},
\end{equation}
and therefore the set $\kappa_{1-\gamma}$ has finite measure if we choose $\gamma\in [0,\frac12]$, and, in particular, $\meas(\kappa_{1-\gamma})\leq C(M)$.

\begin{lemma}\label{lem:G}
  Given $\Theta_0\in \mathcal{G} \cap B_{M_0}$ for some constant $M_0$, given $\Theta\in C([0,T], B_M)$, we denote by $\tilde \Theta=(\tilde \zeta,\tilde U,\tilde v,\tilde w,\bar{\tilde{h}},\tilde h,\tilde r)=\mathcal{P}(\Theta)$ with initial data $\Theta_0$. Let 
\begin{equation*}
\bar M=\norm{Q(\Theta)}_{L^\infty_TL^\infty}+\norm{P(\Theta)}_{L^\infty_TL^\infty}+\norm{U}_{L^\infty_TL^\infty}^2.
\end{equation*}
Then the following statements hold:

  (\textit{i}) For all $t$ and almost all $\xi$ 
  \begin{equation}\label{eq:G1}
    \tilde q(t,\xi)\geq 0,\quad \tilde h(t,\xi)\geq 0, \quad \bar{\tilde{h}}(t,\xi)\geq 0,
  \end{equation}
  and 
  \begin{equation}\label{eq:G2}
    \tilde q\bar{\tilde{h}} =\tilde w^2+\tilde{r}^2.
  \end{equation}
  Thus, $\tilde q(t,\xi)=0$ implies $\tilde w(t,\xi)=0$ and $\tilde{r}(t,\xi)=0$. Recall that $\tilde q=\tilde v+1$.

  (\textit{ii}) We have 
  \begin{equation}\label{eq:G3}
    \norm{\frac{1}{\tilde q+\tilde h}(t,\dott)}_{L^\infty}\leq 2e^{C(\bar M)T}\norm{\frac{1}{q_0+h_0}}_{L^\infty},
  \end{equation}
  and 
  \begin{equation}
    \label{eq:G3b}
    \norm{(\tilde q+\tilde h)(t,\dott)}_{L^\infty}\leq 2e^{C(\bar M)T}\norm{q_0+h_0}_{L^\infty},
  \end{equation}
  for all $t\in[0,T]$ and a constant $C(\bar M)$ which depends only on $\bar M$. In particular, $\tilde q+\tilde h$ remains bounded strictly away from zero.

  (\textit{iii}) There exists a $\gamma\in(0,\frac12)$ depending only on $\bar M$ such that if $\xi\in \kappa_{1-\gamma}$, then $\tilde \Theta(t,\xi)\in \Omega_1$, where  
  $ \Omega_1$ is given in Definition \ref{def:Omega}, for all $t\in [0,\min(\tilde\tau_1(\xi),T)]$, $\frac{\tilde q}{\tilde q+\bar{\tilde{h}}}(t,\xi)$ is a decreasing function with respect to time for $t\in[0,\min(\tilde\tau_1(\xi),T)]$ and $\frac{\tilde w}{\tilde q+\bar{\tilde{h}}}(t,\xi)$ is an increasing function with respect to time for $t\in[0,\min(\tilde\tau_1(\xi),T)]$. Thus we infer that 
  \begin{equation}
    \frac{w_0}{q_0+\bar h_0}(\xi)\leq\frac{\tilde w}{\tilde q+\bar{\tilde{h}}}(t,\xi)\leq 0\quad \text{and} \quad 0\leq \frac{\tilde q}{\tilde q+\bar{\tilde{h}}}(t,\xi)\leq \frac{q_0}{q_0+\bar{h}_0}(\xi),
  \end{equation}
  for $t\in[0,\min(\tilde\tau_1(\xi),T)]$.
  In addition, for $\gamma$ sufficiently small, depending only on $\bar M$ and $T$, we have 
  \begin{equation}\label{eq:G5}
    \kappa_{1-\gamma}\subset \{\xi\in\Real\mid 0\leq \tilde \tau_1(\xi)<T\}.
  \end{equation}

  (\textit{iv})
  Moreover, for any given $\gamma\in(0,\frac12)$, there exists $\hat T>0$ such that 
  \begin{equation}\label{eq:G4}
    \{\xi\in\Real\mid 0<\tilde\tau_1(\xi)<\hat T\}\subset \kappa_{1-\gamma}.
  \end{equation}
\end{lemma}

\begin{proof}
  (\textit{i}) Since $\Theta_0\in\G$,  equations
  \eqref{eq:G1} and \eqref{eq:G2} hold for almost every
  $\xi\in\Real$ at $t=0$. We consider such a $\xi$
  and will drop it in the notation. From
  \eqref{eq:sysfix2}, we have, on the one hand,
  \begin{equation*}
    (\tilde q\bar{\tilde{h}})_t=\tilde q_t\bar{\tilde{h}}+\tilde q\bar{\tilde{h}}_t=\tilde w\bar{\tilde{h}}+2(U^2-P(\Theta))\tilde w\tilde q, \quad t\in(\tilde\tau_n, \tilde\tau_{n+1})
  \end{equation*}
  and, on the other hand,
  \begin{equation*}
    (\tilde w^2+\tilde{r}^2)_t=2\tilde w\tilde w_t= \tilde w\bar{\tilde{h}}+2(U^2-P(\Theta))\tilde w\tilde q,\quad t\in(\tilde\tau_n, \tilde\tau_{n+1}).
  \end{equation*}
  Thus, 
  \begin{equation}
    \label{eq:derpropeq0}
    (\tilde q\bar{\tilde{h}}-\tilde w^2-\tilde{r}^2)_t=0    
  \end{equation}
  and since $\tilde q(0)\bar{\tilde {h}}(0)=\tilde w^2(0)+\tilde{r}^2(0)$, we have $\tilde
  q(t)\bar{\tilde{h}}(t)=\tilde w^2(t)+\tilde{r}^2(t)$ for all $t\in[0,\tilde \tau_1)$. We show by
  induction that it holds for $t\in[\tilde \tau_{n-1},\tilde\tau_{n}]$ for each $n\geq 1$, where $\tilde\tau_0=0$. We have $\tilde
  q(\tilde \tau_n-0)=q(\tilde \tau_n)=0$ so that, by \eqref{eq:defderatjump2},
  \begin{equation*}
    0=\tilde q(\tilde\tau_n)\bar{\tilde {h}}(\tilde\tau_n-0)=\tilde w^2(\tilde\tau_n)+\tilde{r}^2(\tilde\tau_n).
  \end{equation*}
  Hence, $\tilde w(\tilde\tau_n)= \tilde{r}(\tilde\tau_n)=0$ and 
  \begin{equation*}
    \tilde q(\tilde\tau_n)\bar{\tilde {h}}(\tilde\tau_n)=0=\tilde w^2(\tilde\tau_n)+\tilde{r}^2(\tilde\tau_n)
  \end{equation*}
  so that \eqref{eq:G2} holds for $t=\tilde\tau_n$. By \eqref{eq:derpropeq0}, we obtain that
  \eqref{eq:G2} holds also on the whole interval $[\tilde\tau_{n},\tilde\tau_{n+1}]$. From the definition of
  $\tilde\tau_1$ we have that $\tilde q(t)>0$ on $[0,\tilde\tau_1)$ and $\tilde
  q(\tilde\tau_1)=\tilde w(\tilde\tau_1)=\tilde{r}(\tilde\tau_1)=0$ and $\bar{\tilde{h}}(\tau_1)\geq
  0$.  Hence $\tilde w(t)$ becomes positive at time $\tilde\tau_1$,  and therefore $\tilde q(t)$ is
  increasing. Since whenever $\tilde q(t)=0$, we have that $\tilde w$ changes sign from negative to
  positive, it follows that $\tilde q(t)\geq0$ for $t\geq0$. From \eqref{eq:G2} it follows that, for
  $t\in[0,\tilde\tau_1)$, $\bar{\tilde{h}}(t)=\frac{\tilde w^2+\tilde{r}^2}{\tilde q}(t)$ and
  therefore $\bar{\tilde{h}}(t)\geq0$. By the continuity of $\bar{\tilde{h}}$ (with respect to time) we
  have $\lim_{t\uparrow\tilde\tau_1}\bar{\tilde{h}}(t)\geq0$ and, using \eqref{eq:defderatjump2} and \eqref{eq:G2} we
  have $\bar{\tilde{h(t)}}\geq0$ for all $t\in [0,\tilde\tau_2)$. The claim now follows  by induction.

  (\textit{ii}) We consider a fixed $\xi$ that we suppress in the notation. We denote by
  $\abs{\tilde Z}_2=(\tilde q^2+\tilde w^2+\tilde h^2+\tilde{r}^2)^{1/2}$ the Euclidean norm of
  $\tilde Z=(\tilde q,\tilde w,\tilde h,\tilde r)$. Since $0\leq \bar{\tilde{h}}\leq \tilde h$, we have
  \begin{equation*}
    \frac{d}{dt}\abss{\tilde Z}_2^{-2}=-2\abss{\tilde Z}_2^{-4}\,  \tilde Z\,  \frac{d\tilde Z}{dt}\leq C(\bar M)\abss{\tilde Z}_2^{-2}
  \end{equation*}
  for a constant $C(\bar M)$ which depends only on $\bar M$. 
Applying Gronwall's lemma, we obtain $\abs{\tilde Z(t)}_2^{-2}\leq e^{C(\bar
    M)T}\abs{Z(0)}_2^{-2}$. Hence,
  \begin{equation}
    \label{eq:mindec}
    \frac{1}{\tilde q^2+\tilde w^2+\tilde h^2+\tilde{r}^2}(t)\leq e^{C(\bar M)T}\frac{1}{q_0^2+w_0^2+h_0^2+ r_0^2}.
  \end{equation}
  Using \eqref{eq:G2}, we have
  \begin{equation*}
    \tilde q^2+\tilde w^2+\tilde h^2+\tilde{r}^2\leq\tilde q^2+\tilde q\tilde h+\tilde h^2.
  \end{equation*}
  Hence, \eqref{eq:mindec} yields
  \begin{equation*}
    \frac1{(\tilde q+\tilde h)^2}(t)\leq\frac{1}{\tilde q^2+\tilde q\tilde h+\tilde h^2}(t)\leq e^{C(\bar M)T}\frac{1}{q_0^2+h_0^2}\leq 2e^{C(\bar M)T}\frac{1}{(q_0+h_0)^2}. 
  \end{equation*}
  The second claim can be shown similarly.

  (\textit{iii}) Let us consider a given $\xi\in \kappa_{1-\gamma}$. We are going to determine an upper bound on $\gamma$ depending only on $\bar M$ such that the conclusions of (\textit{iii}) hold. For $\gamma$ small enough we have $\Theta_0(\xi)\in\Omega_1$ as otherwise $g_2(\Theta_0(\xi))=q_0(\xi)+\bar h_0(\xi)$ and
  \begin{equation*}
    1=\frac{g_2(\Theta_0(\xi))}{q_0(\xi)+\bar h_0(\xi)} <\frac{-w_0(\xi)+2q_0(\xi)}{q_0(\xi)+\bar h_0(\xi)} \leq\sqrt{\gamma}+2\gamma
  \end{equation*}
  would lead to a contradiction.  We claim that there exists a constant $\gamma(\bar M)$ depending
  only on $\bar M$ such that for all $\gamma\leq\gamma (\bar M)$, $\xi\in\Real$, and $t\in [0,T]$,
  \begin{equation}
    \label{eq:gammp1}
    \frac{\tilde q}{\tilde q +\bar{\tilde{h}}}(t,\xi)\leq \gamma\text{ and }\tilde w(t,\xi)=0 \text{ implies }\tilde q(t,\xi)=0, 
  \end{equation}
  and 
  \begin{equation}
    \label{eq:gammp2}
    \frac{\tilde q}{\tilde q +\bar{\tilde{h}}}(t,\xi)\leq \gamma \text{ implies } \left(\frac{\tilde w}{\tilde q+\bar{\tilde{h}}}\right)_t(t,\xi)\geq 0. 
  \end{equation}
  We consider a fixed $\xi\in\Real$ and suppress it in the notation.  If $\tilde w(t)=0$, then
  \eqref{eq:G2} yields $\tilde q(t)\bar{\tilde{h}}(t)=0$. Thus, either $\tilde q(t)=0$ or
  $\bar{\tilde{h}}(t)=0$. Assume that $\tilde q(t)\neq 0$, then $\bar{\tilde{h}}(t)=0$. Hence
  $1-\gamma\leq \frac{\bar{\tilde{h}}(t)}{\tilde q(t)+\bar{\tilde{h}}(t)}=0$, and we are led to a
  contradiction. Hence, $\tilde q(t)=0$, and we have proved \eqref{eq:gammp1}.\\ If $\frac{\tilde
    q}{\tilde q+\bar{\tilde{h}}}(t)\leq\gamma$, we have
  \begin{align}
    \label{eq:lbdwt}
    \Big(\frac{\tilde w}{\tilde q+\bar{\tilde{h}}}\Big)_t 
    &=\frac12+(U^2-P(\Theta)-\frac12)\frac{\tilde q}{\tilde q+\bar{\tilde{h}}} -(2U^2-2P(\Theta)+1)\frac{\tilde w^2}{(\tilde q+\bar{\tilde{h}})^2} \notag \\
    & \geq \frac12 -C(\bar M)\frac{\tilde q}{\tilde q+\bar{\tilde{h}}}-C(\bar M) \frac{\tilde q\bar{\tilde{h}}}{(\tilde q+\bar{\tilde{h}})^2} \notag \\
    & \geq \frac12-C(\bar M)\gamma.
  \end{align}
  Recall that we allow for a redefinition of $C(\bar M)$. By choosing $\gamma(\bar M)\leq(4C(\bar
  M))^{-1}$, we get $\Big(\frac{\tilde w}{\tilde q+\bar{\tilde {h}}}\Big)_t\geq0$, and we have
  proved \eqref{eq:gammp2}. For any $\gamma\leq\gamma(\bar M)$, we consider a given $\xi$ in
  $\kappa_{1-\gamma}$ and again suppress it in the notation.  We define
  \begin{equation*}
    t_0=\sup\{t\in[0,\tilde\tau_1]\mid \frac{\tilde q}{\tilde q+\bar{\tilde{h}}}(t')<2\gamma \text{ and } \tilde w(t')<0 \text{ for all }t'\leq t\}. 
  \end{equation*}
  Let us prove that $t_0=\tilde \tau_1$. Assume the
  opposite, that is, $t_0<\tilde \tau_1$.  Then we
  have either $\frac{\tilde q}{\tilde q+\bar{\tilde{h}}}(t_0)=2\gamma$ or $\tilde
  w(t_0)=0$.  We have $\Big(\frac{\tilde q}{\tilde q+\bar{\tilde{h}}}\Big)_t\leq0$
  on $[0,t_0]$ and $\frac{\tilde q}{\tilde q+\bar{\tilde{h}}}(t)$ is decreasing on
  this interval.  Hence, $\frac{\tilde q}{\tilde q+\bar{\tilde{h}}}(t_0)\leq \frac{\tilde
    q}{\tilde q+\bar{\tilde{h}}}(0)\leq\gamma$, and therefore we must have
  $\tilde w(t_0)=0$. 
  Then  \eqref{eq:gammp1}
  implies $\tilde q(t_0)=0$, and therefore
  $t_0=\tilde\tau_1$, which contradicts our
  assumption. From \eqref{eq:lbdwt} we get, for
  $\gamma$ sufficiently small,
  \begin{equation*} 
    0=\frac{\tilde w}{\tilde q+\bar{\tilde{h}}}(\tilde\tau_1-0)\geq\frac{\tilde w}{\tilde q+\bar{\tilde{h}}}(0)+\frac14\tilde\tau_1,
  \end{equation*}
  and therefore  $\tilde\tau_1\leq 4\sqrt{\gamma}$. By taking $\gamma$ small enough we can
  impose $\tilde\tau_1<T$, which proves
  \eqref{eq:G5}. It is clear from
  \eqref{eq:gammp2} that $\frac{\tilde w}{\tilde q+\bar{\tilde{h}}}$ is increasing. 
  Assume that $\tilde \Theta(t,\xi)$ leaves $\Omega_1$
  for some $t<\min (\tilde\tau_1,T)$.  Then we get
  \begin{equation*}
    1 =\frac{\tilde q(t)+\bar{\tilde{h}}(t)}{\tilde q(t)+\bar{\tilde{h}}(t)}\leq \frac{\abs{\tilde w(t)}+2\tilde q(t)}{\tilde q(t)+\bar{\tilde{h}}(t)}\leq \sqrt{\gamma}+2\gamma
  \end{equation*}
  and, by taking
  $\gamma$ small enough, we are led to a
  contradiction. 

  (\textit{iv}) Without loss of generality we assume $\hat T\leq 1$. From (\textit{iii}) we know that there exists a $\gamma'$ only depending on $\bar M$ such that for $\xi\in\kappa_{1-\gamma'}$, we have that $\frac{\tilde q}{\tilde q+\bar{\tilde{h}}}$ is a decreasing and $\frac{\tilde w}{\tilde q+\bar{\tilde{h}}}$ is an increasing function both with respect to time on $[0,\min(\tilde\tau_1,T)]$. Let $\bar \gamma \leq \min (\gamma,\gamma')$. We consider a fixed $\xi\in\Real$ such that $\tilde \tau_1(\xi)<\hat T$ (which means implicitly $\tilde r(t,\xi)=0$ for all $t$), but $\xi\not\in \kappa_{1-\bar\gamma}$. We will suppress $\xi$ in the notation from now on. Let us introduce 
  \begin{equation}
    t_0=\inf\{t\in[0,\tilde\tau_1)\mid \frac{\bar{\tilde{h}}}{\tilde q+\bar{\tilde{h}}}(\bar t)\geq 1-\bar\gamma \text{ and } \tilde w(\bar t)\leq 0\text{ for all }\bar t\in[t,\tilde\tau_1)\}.
  \end{equation}
  Since $\tilde w_t(\tilde\tau_1)=\frac12 \bar{\tilde{h}}(\tilde\tau_1)\geq 0$ and $\tilde w(\tilde\tau_1)=\tilde q(\tilde\tau_1)=0$, the definition of $t_0$ is well-posed when $\tilde\tau_1>0$, and we have $t_0<\tilde\tau_1$. By assumption $t_0>0$ and $\tilde w(t_0)=0$ or $\frac{\bar{\tilde{h}}}{\tilde q+\bar{\tilde{h}}}(t_0)= 1-\bar\gamma$. We cannot have $\tilde w(t_0)=0$, since it would imply, see \eqref{eq:gammp1}, that $\tilde q(t_0)=0$ and therefore $t_0=\tilde\tau_1$ which is not possible. Thus we must have $\frac{\bar{\tilde{h}}}{\tilde q+\bar{\tilde{h}}}(t_0)=1-\bar\gamma$, and, in particular, $\frac{\tilde q}{\tilde q+\bar{\tilde{h}}}(t_0)=\bar\gamma$. According to the choice of $\bar\gamma$ we have that $\frac{\tilde q}{\tilde q+\bar{\tilde{h}}}(t)\leq \bar\gamma$ for all $t\geq t_0$, and $\frac{\tilde w}{\tilde q+\bar{\tilde{h}}}(t)$ is increasing. Then we have, following the same lines as in \eqref{eq:lbdwt},
  \begin{equation*}\notag
    \Big(\frac{\tilde w}{\tilde q+\bar{\tilde{h}}}\Big)_t \geq \frac12-C(\bar M)\bar\gamma, \nn
  \end{equation*}
  which yields for $0\leq t_0\leq t'\leq \min(\tilde\tau_1,1)$ that
  \begin{align*}
    \frac{\tilde w}{\tilde q+\bar{\tilde{h}}}(t')\geq \frac{\tilde w}{\tilde q+\bar{\tilde{h}}}(t_0)+(t'-t_0)(\frac12-C(\bar M)\bar \gamma).
  \end{align*}
  Since $\frac{\tilde w}{\tilde q+\bar{\tilde{h}}}(t_0)=-\sqrt{\bar\gamma(1-\bar\gamma)}$, we choose $\hat T$ such that $0> -\sqrt{\bar\gamma(1-\bar\gamma)}+\hat T(\frac12-C(\bar M)\bar \gamma)$. Thus $\frac{\tilde w}{\tilde q+\bar{\tilde{h}}}(\hat T)<0$ and therefore all points which experience wave breaking before $\hat T$ are contained in $\kappa_{1-\bar\gamma}$, since any point entering $\kappa_{1-\bar\gamma}$ at a later time cannot reach the origin within the time interval $[0,\hat T]$ according to the last estimate. 
\end{proof}

\begin{lemma}\label{lem:estshort}
  Given $M>0$, there exist $\bar T$ and $\bar M$ such that for all $T\leq \bar T$ and any initial data $\Theta_0\in\mathcal{G}\cap B_M$, $\mathcal{P}$ is a mapping from $C([0,T],B_{\bar M})$ to $C([0,T], B_{\bar M})$.
\end{lemma}

\begin{proof}
  To simplify the notation, we will generically denote by $K(M)$ and $C(\bar M)$ increasing functions of $M$ and $\bar M$, respectively. Without loss of generality, we assume $\bar T\leq 1$. 

  Let $\Theta\in C([0,T], B_{\bar M})$ for a value of $\bar M$ that will be determined at the end as a function of $M$. We assume without loss of generality $\bar M\geq M$.  Let $\tilde \Theta=\mathcal{P}(\Theta)$. From Lemma~\ref{lem:PQ}  we have 
  \begin{equation}
    \norm{Q(\Theta)}_{L^\infty_TE}\leq C(\bar M), \quad \norm{P(\Theta)}_{L^\infty_TE}\leq C(\bar M).
  \end{equation}
  Since $\tilde U_t=-Q(\Theta)$, we get 
  \begin{equation}\label{eq:estest3}
    \norm{\tilde U}_{L^\infty_TE}\leq \norm{U_0}_{E}+T\norm{Q(\Theta)}_{L^\infty_TE}\leq M +TC(\bar M).
  \end{equation}
  Similarly, since, $\tilde \zeta_t=\tilde U$, we get
  \begin{equation}\label{eq:estest2}
    \norm{\tilde\zeta}_{L^\infty_TL^\infty}\leq \norm{\zeta_0}_{L^\infty}+T\norm{\tilde U}_{L^\infty_TL^\infty}\leq M+TC(\bar M).
  \end{equation}

  From \eqref{eq:sysfix2}, by the Minkowski inequality for integrals, we get 
  \begin{subequations}\label{eq:estL1}
    \begin{align}
      \norm{\tilde v(t,\dott)}_E&\leq \norm{v_0}_E+\int_0^t \norm{\tilde w(t',\dott)}_Edt',\\
      \norm{\tilde w(t,\dott)}_E&\leq \norm{w_0}_E+T\norm{P(\Theta)-U^2}_{L^\infty_TE}\\ \nn
      & \quad +\int_0^t \Big(\frac12\norm{\tilde h(t',\dott)}_E+\norm{U^2-P(\Theta)}_{L^\infty_TE}\norm{\tilde v(t',\dott)}_E\Big)dt',\nn \\
      \label{eq:estL12}
      \norm{\tilde h(t,\dott)}_E& \leq \norm{h_0}_E+2\int_0^t\norm{U^2-P(\Theta)}_{L^\infty_TE}\norm{\tilde w(t',\dott)}_Edt',\\
      \norm{\tilde{r}(t,\dott)}_E&\leq \norm{r_0}_E.
    \end{align}
  \end{subequations}
Here we used that $0\leq\bar{\tilde{h}}(t,\xi)\leq \tilde h(t,\xi)$ and that $\tilde h(t,\xi)$ is continuous with respect to time. 
These inequalities imply that
  \begin{equation}
    \norm{\tilde Z(t,\dott)}_{\bar W}\leq K(M)+TC(\bar M)+C(\bar M)\int_0^t \norm{\tilde Z(t',\dott)}_E dt',
  \end{equation}
  and, applying Gronwall's inequality,
  \begin{equation}\label{eq:estest4}
    \norm{\tilde Z}_{L^\infty_T\bar W}\leq (K(M)+TC(\bar M))e^{C(\bar M)T}.
  \end{equation}

  Gathering \eqref{eq:estest3}, \eqref{eq:estest2}, and \eqref{eq:estest4}, we get 
  \begin{equation}\label{eq:estest10}
    \norm{\tilde X}_{L^\infty_T\bar V}\leq (K(M)+TC(\bar M))e^{C(\bar M)T}.
  \end{equation}

Moreover, \eqref{eq:sysfix2} implies that 
\begin{equation}
 \norm{\tilde h(t,\dott)}_{L^1}\leq \norm{\tilde h_0}_{L^1}+2T\norm{U^2-P(\Theta)}_{L^\infty_TL^2}\norm{\tilde w}_{L^\infty_TL^2},
\end{equation}
and hence 
\begin{equation}
 \norm{\tilde h}_{L^\infty_T L^1}\leq (K(M)+TC(\bar M))e^{C(\bar M)T}.
\end{equation}

  From \eqref{eq:G3} we get 
  \begin{equation*}
    \norm{\frac{1}{\tilde q+\tilde h}}_{L^\infty_TL^\infty}\leq K(M) e^{C(\bar M)T}.
  \end{equation*}
  Thus we finally obtain 
  \begin{equation}
    \norm{\tilde X}_{L^\infty_T \bar V}+\norm{\tilde h}_{L^\infty_T L^1}+\norm{\frac{1}{\tilde q+\tilde h}}_{L^\infty_TL^\infty}\leq (K(M)+TC(\bar M))e^{C(\bar M)T}
  \end{equation}
  for some constants $K(M)$ and $C(\bar M)$ that only depend on $M$ and $\bar M$, respectively. We
  now set $\bar M=2K(M)$.  Then we can choose $T$ so small that $(K(M)+C(\bar M)T)e^{C(\bar M)T}\leq
  2K(M)=\bar M$, and therefore $\norm{\tilde X}_{L^\infty_T \bar V}+\norm{\tilde h}_{L^\infty_T L^1}+\norm{\frac{1}{\tilde q+\tilde
      h}}_{L^\infty_TL^\infty}\leq \bar M$.
\end{proof}

Given $\Theta_0\in \mathcal{G}\cap B_M$, there exists $\bar M$, which depends only on $M$, such that
$\mathcal{P}$ is a mapping from $C([0,T], B_{\bar M})$ to $C([0,T], B_{\bar M})$ for $T$ small
enough. Therefore we set
\begin{equation}\label{def:mapP}
  \Ima(\mathcal{P})=\{\mathcal{P}(\Theta)\mid \Theta\in C([0,T],B_{\bar M})\}.
\end{equation}

\begin{lemma}\label{lem:break}
 Given $\Theta_0\in \G \cap B_{M}$, given $\Theta\in C([0,T], B_{\bar M})$, we denote by $\tilde\Theta=\mathcal{P}(\Theta)\in C([0,T],B_{\bar M})$ with initial data 
 $\tilde\Theta|_{t=0}=\Theta_0$. 
 
Then there exists a time $\hat T$ depending on $\bar M$ such that any point $\xi$ can experience
  wave breaking at most once within the time interval $[T_0,T_0+\hat T]$ for any $T_0\ge0$. More
  precisely, given $\xi\in\Real$, we have
   \begin{equation}
   \tilde\tau_{j+1}(\xi)-\tilde\tau_j(\xi)>\hat T   \text{ for all } j. 
   \end{equation}
   In addition,
  for $\hat T$ sufficiently small, we get that in this case $\tilde w(t,\xi)\geq 0$ for all
  $t\in[\tilde\tau_j(\xi), \tilde\tau_j(\xi)+\hat T]$.
\end{lemma}

\begin{proof}
If no wave breaking occurs within $[0,T]$ or $\alpha=1$, there is nothing to prove. Therefore let us assume that $\alpha\in[0,1)$ and for some fixed $\xi\in\Real$ wave breaking occurs. Moreover, let us assume the worst possible case, namely $T_0=\tilde\tau_1(\xi)$, since all other cases follow from this one.
At time $\tilde\tau_1(\xi)$ we have $\tilde q(\tilde\tau_1(\xi),\xi)=\tilde
  w(\tilde\tau_1(\xi),\xi)=0$, and, in particular, $\frac{\bar{\tilde{h}}}{\tilde
    q+\bar{\tilde{h}}}(\tilde\tau_1(\xi),\xi)=1$ and $\tilde r(t,\xi)=0$ for all $t$. Moreover, wave
  breaking can only take place if $\tilde w(t,\xi)\leq 0$ for $\tilde\tau_1(\xi)-\varepsilon\leq
  t\leq \tilde\tau_1(\xi)$, but right after wave breaking $\tilde w(t,\xi)$ is positive, in the case
  where $\alpha<1$. Thus before wave breaking can occur once more at $\xi\in\Real$, $\tilde w(t,\xi)$ has to change sign from positive to negative at some time $t^*>\tilde\tau_1(\xi)$. Hence we will now establish a lower bound on $t^*-\tilde\tau_1(\xi)$, which defines $\hat T$.

   If $\tilde w(t,\xi)$ changes sign for the first time at time $t^*$ for $t>\tilde \tau_1(\xi)$, then \eqref{eq:G2} implies that either
  $\tilde q(t^*,\xi)=0$ (i.e., wave breaking occurs) or $\frac{\bar{\tilde{h}}}{\tilde
    q+\bar{\tilde{h}}}(t^*,\xi)=0$ (i.e., no wave breaking). The first alternative is not possible
  as $\tilde q_t(t,\xi)=\tilde w(t,\xi)>0$ for $t\in(\tau_1(\xi), t^*)$, in the case where $\alpha<1$. Hence,
  $\frac{\tilde q(t^*,\xi)}{\tilde q(t^*,\xi)+\bar{\tilde h}(t^*,\xi)}=1$. Thus if we can establish a lower bound on how long it takes for the function $\frac{\tilde q(t,\xi)}{\tilde q(t,\xi)+\bar{\tilde{h}}(t,\xi)}$, which equals $0$ at time $\tilde \tau_1(\xi)$, to reach $1$ after wave breaking the claim follows.

Observe first that \eqref{eq:G2} implies that
  \begin{equation}
    \Big\vert \frac{\tilde w}{\tilde q+\bar{\tilde{h}}}(t,\xi)\Big\vert \leq\frac{1}{\sqrt{2}}, \quad \text{and}\quad \Big\vert \frac{\tilde{r}}{\tilde q+\bar{\tilde{h}}}(t,\xi)\Big\vert\leq \frac{1}{\sqrt{2}} 
  \end{equation}
  for all $t\in[0,\infty)$ and $\xi\in\Real$. Moreover, according to Lemma~\ref{lem:PQ} (\textit{i}) we have 
\begin{equation}\label{constbreak}
 \norm{P(\Theta)}_{L^\infty_TL^\infty}+\norm{U}^2_{L^\infty_TL^\infty}\leq C(\bar M).
\end{equation}
From 
  \eqref{eq:sysfix2} we get
  \begin{align*}
    \Big(\frac{\tilde q}{\tilde q+\bar{\tilde{h}}}\Big)_t&=\frac{\tilde w}{\tilde q+\bar{\tilde{h}}}\Big(1-2\frac{\tilde q}{\tilde q
      +\bar{\tilde{h}}}(U^2-P(\Theta)+\frac12)\Big)\\ 
    &\leq \frac{1}{\sqrt{2}}\big(1+C(\bar M)\big),
  \end{align*}
  for $t\in[\tilde\tau_1(\xi), \tilde\tau_2(\xi))$.
Hence, integrating the latter equation in time from
  $\tau_1$ to $t^*$ yields $1=\frac{\tilde q}{\tilde q+\bar{\tilde{h}}}(t^*)\leq \frac{1}{\sqrt{2}}
  (t^* - \tau_1(\xi))\big(1+C(\bar M)\big)$. Choosing $\hat T=\sqrt{2}(1+C(\bar M))^{-1}$ concludes the proof.

\end{proof}

We define the \textit{discontinuity residual} as
\begin{equation*}
  \Gamma(\Theta,\tilde \Theta)=\int_0^T\int_\Real \vert \bar h(t,\xi)-\bar{\tilde{h}}(t,\xi)\vert d\xi dt.
\end{equation*}

According to Lemma~\ref{lem:PQ} (\textit{ii}), we have
\begin{multline}\label{eq:estPQGamma}
  \norm{Q(\Theta)-Q(\tilde \Theta)}_{L^1_T
    E}+\norm{P(\Theta)-P(\tilde \Theta)}_{L^1_TE}\\\leq C(\bar M) \Big(T\norm{X-\tilde
    X}_{L^\infty_T \bar V}+\Gamma(\Theta,\tilde
  \Theta)\Big).
\end{multline}

In the next lemma we establish some estimates for
$\Gamma(\Theta,\tilde \Theta)$,
$\Gamma(\mathcal{P}(\Theta),\mathcal{P}(\tilde \Theta))$ and
a quasi-contraction property for $\mathcal{P}$.

\begin{lemma} \label{lem:contrgamma} Given $\Theta$, $\tilde \Theta\in \Ima(\mathcal{P})$,
  $\gamma\in(0,\frac12)$, let $\Theta_2=\mathcal{P}(\Theta)$ and $\tilde \Theta_2=\mathcal{P}(\tilde \Theta)$, then there exists $T>0$ depending on $\bar M$ 
  such that the following inequalities hold\\
   (\textit{i})  
  \begin{equation}
    \label{eq:contrh}
\norm{h_2-\tilde h_2}_{L^\infty_TL^1}\leq C(\bar M)T\Big(\norm{X_2-\tilde X_2}_{L^\infty_T\bar V}+\norm{X-\tilde X}_{L^\infty_T \bar V}+\Gamma(\Theta,\tilde\Theta)\Big),
  \end{equation}
 (\textit{ii})
  \begin{equation}
    \label{eq:contGamma}
    \Gamma(\Theta,\tilde \Theta)\leq C(\bar M)\norm{X-\tilde X}_{L^\infty_T\bar V},
  \end{equation}
  (\textit{iii}) 
  \begin{equation}
    \label{eq:contracGamma}
    \Gamma(\Theta_2,\tilde \Theta_2)
    \leq C(\bar M)\Big(T\big(\norm{X_2-\tilde X_2}_{L^\infty_T\bar V}+\norm{X-\tilde X}_{L^\infty_T\bar V}\big) + \gamma\Gamma(\Theta,\tilde \Theta)\Big),
  \end{equation}
  (\textit{iv}) 
  \begin{align}
    \label{eq:contracP}
    \norm{X_2-\tilde X_2}_{L^\infty_T\bar V}&\leq
    C(\bar M)\Big(T\norm{X-\tilde X}_{L^\infty_T\bar V}+\Gamma(\Theta,\tilde \Theta)\Big),
  \end{align}
  where $C(\bar M)$ denotes some constant which only depends on $\bar M$.
\end{lemma}

\begin{proof}
  Denote by $\Theta_2=\mathcal{P}(\Theta)$ and $\tilde \Theta_2=\mathcal{P}(\tilde \Theta)$ and, abusing the notation, let $\tau_2(\xi)$ and
  $\tilde\tau_2(\xi)$ 
 be the first time when wave
  breaking occurs at the point $\xi\in\Real$ for $\Theta_2$ and $\tilde \Theta_2$, respectively. Given
  $\gamma>0$ we know from Lemma~\ref{lem:G} (\textit{iv}) and Lemma~\ref{lem:break} that there exists $T$
  small enough such that $\{\xi\in\Real\mid 0<\tau_2(\xi)<T \text{ or }0<\tilde\tau_2(\xi)<T\}\subset
  \kappa_{1-\gamma}$ and such that every point experiences wave breaking at most once within the time
  interval $[0,T]$. We consider such $T$. Without loss of generality we can assume that $T\leq 1$
  and $\gamma\leq\gamma(\bar M)$.

(\textit{i}): From \eqref{eq:sysfix2} we get  
\begin{align}
 & \norm{h_2 -\tilde h_2}_{L^\infty_T L^1}\\ \nn
& \quad  \leq\int_0^T \norm{2(U^2-P(\Theta))(w_2-\tilde w_2)(s,\dott)}_{L^1}ds\\ \nn
& \qquad +\int_0^T \norm{2(U^2-P(\Theta)-\tilde U^2+P(\tilde\Theta))\tilde w_2(s,\dott)}_{L^1}ds\\ \nn 
& \quad \leq \int_{0}^T \norm{2(U^2-P(\Theta))(s,\dott)}_{L^2}\norm{(w_2-\tilde w_2)(s,\dott)}_{L^2} ds\\ \nn 
& \qquad +\int_0^T 2\Big(\norm{(U^2-\tilde U^2)(s,\dott)}_{L^2}+\norm{(P(\Theta)-P(\tilde\Theta))(s,\dott)}_{L^2}\Big)\norm{\tilde w_2(s,\dott)}_{L^2}ds\\ \nn 
& \quad \leq C(\bar M)T\Big(\norm{X_2-\tilde X_2}_{L^\infty_T \bar V}+\norm{X-\tilde X}_{L^\infty_T\bar V}+\norm{P(\Theta)-P(\tilde \Theta)}_{L^\infty_T E}\Big)\\ \nn 
& \quad \leq C(\bar M)T\Big(\norm{X_2-\tilde X_2}_{L^\infty_T \bar V}+\norm{X-\tilde X}_{L^\infty_T\bar V}+\Gamma(\Theta,\tilde\Theta)\Big).
\end{align}

As far as the other estimates are concerned, observe first that for $\xi\in \kappa_{1-\gamma}^c$ no wave breaking occurs, and therefore $\vert \bar h_2(t,\xi)-\bar{\tilde{h}}_2(t,\xi)\vert =\vert h_2(t,\xi)-\tilde h_2(t,\xi)\vert$, since $\Theta_2(0,\xi)=\tilde \Theta_2(0,\xi)$. Moreover, using \eqref{eq:G2}, we get $\bar h_2(t,\xi)=w_2^2(t,\xi)+r_2^2(t,\xi)-\bar h_2(t,\xi) v_2(t,\xi)$ and a similar relation holds for $\bar{\tilde{h}}_2(t,\xi)$. Hence 
  \begin{equation}\label{eq:gammanorm}
    \begin{aligned}
      \int_0^T \int_{\kappa_{1-\gamma}^c}&\vert \bar h_2(t,\xi)-\bar{\tilde{h}}_2(t,\xi)\vert d\xi dt \\ 
      & \quad\leq \int_0^T \int_{\kappa_{1-\gamma}^c} \vert (w_2+\tilde w_2)(w_2-\tilde w_2)\vert (t,\xi)d\xi dt\\  
      & \qquad +\int_0^T \int_{\kappa_{1-\gamma}^c}\vert (r_2+\tilde{r}_2)(r_2-\tilde{r}_2)\vert (t,\xi)d\xi dt\\ 
      & \qquad +\int_0^T \int_{\kappa_{1-\gamma}^c} \big(\vert (\bar h_2-\bar{\tilde{h}}_2)v_2\vert (t,\xi)+\bar{\tilde{h}}_2\vert v_2-\tilde v_2\vert (t,\xi) \big)d\xi dt\\ 
      &\quad \leq C(\bar M) T \norm{X_2-\tilde X_2}_{L^\infty_T \bar V},
    \end{aligned}
  \end{equation}
  where we used the Cauchy--Schwarz inequality in the last step. 
  Thus we have 
  \begin{align}\label{eq:3.65}
    \int_0^T \int_{\Real} \vert \bar h_2(t,\xi)-\bar{\tilde{h}}_2(t,\xi)\vert d\xi dt& \leq C(\bar M)T \norm{X_2-\tilde X_2}_{L^\infty_T\bar V}\\ \nn 
    & \quad +\int_0^T \int_{\kappa_{1-\gamma}}\vert \bar h_2(t,\xi)-\bar{\tilde{h}}_2(t,\xi)\vert d\xi dt.
  \end{align}

  (\textit{ii}): Let us consider $\xi\in \kappa_{1-\gamma}$ such that $\tau_2(\xi)\not = \tilde\tau_2(\xi)$. Without loss of generality we assume $0<\tau_2(\xi)<\tilde\tau_2(\xi)\leq T$. 
 Since $\Theta_2(t,\xi)$ and $\tilde \Theta_2(t,\xi)$ both belong to $\Ima(\mathcal{P})$, we have that $\vert h_2(t,\xi)-\tilde h_2(t,\xi)\vert =\vert \bar h_2(t,\xi)-\bar{\tilde{h}}_2(t,\xi)\vert$ for $t\in[0,\tau_2(\xi))$, and especially
  \begin{equation}\label{eq:estttau1}
    \int_0^{\tau_2} \vert \bar h_2(t,\xi)-\bar{\tilde{h}}_2(t,\xi)\vert dt =\int_0^{\tau_2} \vert h_2(t,\xi)-\tilde h_2(t,\xi)\vert dt.
  \end{equation}

  For $t\in[\tau_2(\xi),\tilde\tau_2(\xi))$, we have $\bar h_2(t,\xi)=h_2(t,\xi)-l_0(\xi)-l_1(\xi)$ and $\bar{\tilde{h}}_2(t,\xi)=\tilde h_2(t,\xi)-l_0(\xi)$. Hence it follows that 
  \begin{equation}
    \vert \bar h_2(t,\xi)-\bar{\tilde{h}}_2(t,\xi)\vert \leq \vert h_2(t,\xi)-\tilde h_2(t,\xi)\vert +l_1(\xi).
  \end{equation}
  Since \eqref{eq:G1} implies $0\leq l_1(\xi)\leq h_2(t,\xi)-l_0(\xi)$ for all $t\in [\tau_2(\xi),\tilde\tau_2(\xi)]$, we get 
  \begin{align}\label{eq:3.68}
    \int_{\tau_2}^{\tilde\tau_2} l_1(\xi)dt & \leq \int_{\tau_2}^{\tilde\tau_2} (h_2(t,\xi)-l_0(\xi))dt\\ \nn 
    & \leq \int_{\tau_2}^{\tilde\tau_2} \vert h_2(t,\xi)-\tilde h_2(t,\xi)\vert dt+\int_{\tau_2}^{\tilde\tau_2} (\tilde h_2(t,\xi)-l_0(\xi))dt\\ \nn 
    & \leq \int_{\tau_2}^{\tilde\tau_2}\vert h_2(t,\xi)-\tilde h_2(t,\xi)\vert dt +\int_{\tau_2}^{\tilde\tau_2} \bar{\tilde{h}}_2(t,\xi)dt. 
  \end{align}
  Since $\tilde \Theta_2=\mathcal{P}(\tilde \Theta)$ for $\tilde \Theta\in C([0,T], B_{\bar M})$, we get using \eqref{eq:sysfix2}, for $t\in[\tau_2(\xi),\tilde\tau_2(\xi)]$ that 
  \begin{equation}
    \tilde w_2(t,\xi) =\tilde w_2 (\tau_2(\xi),\xi)+\frac12 \int_{\tau_2}^t \bar{\tilde{h}}_2(t',\xi)dt'+\int_{\tau_2}^t (\tilde U^2-P(\tilde \Theta))\tilde q_2(t',\xi)dt'.
  \end{equation}
  According to Lemma~\ref{lem:G}, since $\xi\in\kappa_{1-\gamma}$, we have $\tilde \Theta_2(t,\xi)\in
  \Omega_1$ for all $t\in [0,\min(\tilde\tau_2(\xi),T)]$. Moreover, $\tilde w_2(t,\xi)\leq 0$ on the
  interval $[0, \min(\tilde\tau_2(\xi),T)]$ while $\tilde q_2(t,\xi)$ is decaying.  Furthermore,
  $\norm{\tilde U^2-P(\tilde \Theta)}_{L^\infty_TE}\leq C(\bar M)$,
  $w_2(\tau_2(\xi),\xi)=q_2(\tau_2(\xi),\xi)=0$ and $w_2(t,\xi)\geq 0$ for all $t\in[\tau_2(\xi),
  T]$. Thus we get that 
      \begin{align}\nn
      \frac12 \int_{\tau_2}^{\tilde\tau_2} \bar{\tilde{h}}_2(t',\xi)dt'
      & = \tilde w_2(\tilde\tau_2(\xi),\xi)-\tilde w_2(\tau_2(\xi),\xi)-\int_{\tau_2}^{\tilde\tau_2}(\tilde U^2-P(\tilde \Theta))\tilde q_2(t',\xi)dt'\\  \nn
      & \leq -\tilde w_2(\tau_2(\xi),\xi) +C(\bar M)\int_{\tau_2}^{\tilde\tau_2}\tilde q_2(\tau_2(\xi),\xi) dt'\\ \nn
      & \leq w_2(\tau_2(\xi),\xi)-\tilde w_2(\tau_2(\xi),\xi)\\ \nn
      & \quad +C(\bar M)\int_{\tau_2}^{\tilde\tau_2} \tilde q_2(\tau_2(\xi),\xi)-q_2(\tau_2(\xi),\xi)dt'\\ \label{eq:3.70b}
      &\leq w_2(\tau_2(\xi),\xi)-\tilde w_2(\tau_2(\xi),\xi)+ C(\bar M)T \norm{X_2(\dott,\xi)-\tilde X_2(\dott,\xi)}_{L^\infty_{\tau_2}} \\   
      & \leq w_2(\tau_2(\xi),\xi)-\tilde w_2(\tau_2(\xi),\xi)+ C(\bar M)T \norm{X_2-\tilde X_2}_{L^\infty_T \bar V}.  \nn
    \end{align}
  Combining the above estimates  yields
  \begin{align}\label{eq:estttau2}
    \int_{\tau_2}^{\tilde\tau_2}\vert \bar h_2(t,\xi)-\bar{\tilde{h}}_2(t,\xi)\vert dt & \leq \int_{\tau_2}^{\tilde\tau_2} \vert h_2(t,\xi)-\tilde h_2(t,\xi)\vert dt+\int_{\tau_2}^{\tilde\tau_2} l_1(\xi)dt\\ \nn 
    & \leq 2\int_{\tau_2}^{\tilde\tau_2} \vert h_2(t,\xi)-\tilde h_2(t,\xi)\vert dt+ \int_{\tau_2}^{\tilde\tau_2} \bar{\tilde{h}}_2(t,\xi)dt\\ \nn 
    & \leq 2 (w_2(\tau_2(\xi),\xi)-\tilde w_2(\tau_2(\xi),\xi))\\ \nn 
    & \quad+2\int_{\tau_2}^{\tilde\tau_2} \vert h_2(t,\xi)-\tilde h_2(t,\xi)\vert dt+C(\bar M)T\norm{X_2-\tilde X_2}_{L^\infty_T \bar V}.
  \end{align}

  For $t\in[\tilde\tau_2(\xi),T]$, we have $\bar h_2(t,\xi)=h_2(t,\xi)-l_0(\xi)-l_1(\xi)$ and $\bar{\tilde{h}}_2(t,\xi)=\tilde h_2(t,\xi)-l_0(\xi)-\tilde l_1(\xi)$, and, in particular,  
  \begin{equation}
    \vert \bar h_2(t,\xi)-\bar{\tilde{h}}_2(t,\xi)\vert \leq \vert h_2(t,\xi)-\tilde h_2(t,\xi)\vert +\vert l_1(\xi)-\tilde l_1(\xi)\vert,
  \end{equation}
  where $l_1(\xi)=\alpha (h_2(\tau_2(\xi),\xi)-l_0(\xi))$ and $\tilde l_1(\xi)=\alpha(\tilde h_2(\tilde\tau_2(\xi),\xi)-l_0(\xi))$. Thus we can write
  \begin{align}\label{eq:estttau9}
    \vert l_1(\xi)-\tilde l_1(\xi)\vert & =\alpha\vert h_2(\tau_2(\xi),\xi)-\tilde h_2(\tilde\tau_2(\xi),\xi)\vert\\ \nn 
    & \leq \alpha( \vert h_2(\tau_2(\xi),\xi)-h_2(\tilde\tau_2(\xi),\xi)\vert +\vert h_2(\tilde\tau_2(\xi),\xi)-\tilde h_2(\tilde\tau_2(\xi),\xi)\vert).
  \end{align}
  The first term on the right-hand side can be estimated, using \eqref{eq:sysfix2}, as follows
  \begin{align}\label{eq:estttau8}
    \vert h_2(\tau_2(\xi),\xi)-h_2(\tilde\tau_2(\xi),\xi)\vert & \leq \int_{\tau_2}^{\tilde\tau_2} 2\vert U^2-P(\Theta)\vert w_2(t,\xi)dt\\ \nn 
    & \leq C(\bar M) \int_{\tau_2}^{\tilde\tau_2} w_2(t,\xi)-\tilde w_2(t,\xi) dt \\ \nn 
    & \leq C(\bar M)T \norm{X_2-\tilde X_2}_{L^\infty_T\bar V},
  \end{align}
  where we used that $w_2(t,\xi)\geq 0$ and $\tilde w_2(t,\xi)\leq 0$ for all $t\in[\tau_2(\xi),\tilde\tau_2(\xi)]$.
  Combining the above estimates yields
  \begin{align}\label{eq:estttau3}
    \int_{\tilde\tau_2}^T \vert \bar h_2(t,\xi)-\bar{\tilde{h}}_2(t,\xi)\vert dt 
    & \leq \int_{\tilde\tau_2}^T \vert h_2(t,\xi)-\tilde h_2(t,\xi)\vert dt +T\vert l_1(\xi)-\tilde l_1(\xi)\vert \\ \nn 
    & \leq \int_{\tilde\tau_2}^T \vert h_2(t,\xi)-\tilde h_2(t,\xi)\vert dt\\ \nn 
    & \quad + \alpha T(\vert h_2(\tau_2(\xi),\xi)-h_2(\tilde\tau_2(\xi),\xi)\vert \\ \nn
    &\qquad\quad +\vert h_2(\tilde\tau_2(\xi),\xi)-\tilde h_2(\tilde \tau_2(\xi),\xi)\vert )\\ \nn
    & \leq \int_{\tilde\tau_2}^T \vert h_2(t,\xi)-\tilde h_2(t,\xi)\vert dt + C(\bar M)T\norm{X_2-\tilde X_2}_{L^\infty_T \bar V}.
  \end{align}

  Adding \eqref{eq:estttau1},  \eqref{eq:estttau2}, and \eqref{eq:estttau3}, we obtain 
  \begin{multline}\label{eq:estttau4}
    \int_0^T \vert \bar h_2(t,\xi)-\bar{\tilde{h}}_2(t,\xi)\vert dt\\ 
    \leq 2\int_0^T \vert h_2(t,\xi)-\tilde h_2(t,\xi)\vert dt+C(\bar M)\norm{X_2-\tilde X_2}_{L^\infty_T \bar V}.
  \end{multline}
  Note that this inequality is true for all $\xi\in\kappa_{1-\gamma}$. Since $\meas(\kappa_{1-\gamma})\leq C(\bar M)$, we can apply Fubini's theorem and use \eqref{eq:estttau4} to obtain 
  \begin{equation}\label{eq:3.77}
    \int_0^T \int_{\kappa_{1-\gamma}} \vert \bar h_2(t,\xi)-\bar{\tilde{h}}_2(t,\xi)\vert d\xi dt \leq C(\bar M)\norm{X_2-\tilde X_2}_{L^\infty_T \bar V}.
  \end{equation}
  Combining \eqref{eq:3.65} and \eqref{eq:3.77} finally yields \eqref{eq:contGamma}.

  (\textit{iii}): A close inspection of the proof of (\textit{ii}) reveals that we only need to improve  \eqref{eq:estttau2}. Let us consider $\xi\in \kappa_{1-\gamma}$ and assume for the moment that $0<\tau_2(\xi)<\tilde\tau_2(\xi)\leq T$, since all other cases can be derived from this one. For $t\in[\tau_2(\xi),\tilde\tau_2(\xi))$, we have 
  \begin{equation*}
    \int_{\tau_2}^{\tilde\tau_2}\vert \bar h_2(t,\xi)-\bar{\tilde{h}}_2(t,\xi)\vert dt\leq 2 (w_2(\tau_2(\xi),\xi)-\tilde w_2(\tau_2(\xi),\xi))+C(\bar M)T\norm{X_2-\tilde X_2}_{L^\infty_T \bar V}.
  \end{equation*}
  In order to improve this estimate we will use that $\Theta$  not only is an element of $C([0,T], \bar V)$ like in (\textit{ii}), but also belongs to $\Ima(\mathcal{P})$. From \eqref{eq:sysfix2}, we get that 
  \begin{equation}\label{eq:3.79}
    \begin{aligned}
    &  w_2(\tau_2(\xi),\xi)-\tilde w_2(\tau_2(\xi),\xi) \\
      &\qquad \leq \frac12 \int_0^{\tau_2} \vert \bar h_2(t,\xi)-\bar{\tilde{h}}_2(t,\xi) \vert dt 
      +\int_0^{\tau_2}\vert U^2-P(\Theta)\vert(t,\xi)\vert q_2(t,\xi)-\tilde q_2(t,\xi)\vert dt\\ 
      &\qquad \quad + \int_0^{\tau_2} \vert U^2-P(\Theta)-\tilde U^2+P(\tilde \Theta)\vert (t,\xi)\tilde q_2(t,\xi)dt\\ 
      &\qquad \leq \frac12 \int_0^{\tau_2} \vert h_2(t,\xi)-\tilde{h}_2(t,\xi) \vert dt +C(\bar M)T \norm{X_2-\tilde X_2}_{L^\infty_T \bar V}\\  
      &\qquad \quad +C(\bar M)\gamma \norm{U^2-P(\Theta)-\tilde U^2+P(\tilde \Theta)}_{L^1_T E}\\  
      &\qquad \leq C(\bar M) \Big(T(\norm{X_2-\tilde X_2}_{L^\infty_T \bar V}+\norm{X-\tilde X}_{L^\infty_T \bar V})+\gamma \Gamma(\Theta,\tilde \Theta)\Big),
    \end{aligned}
  \end{equation}
  where we used \eqref{eq:estPQGamma} and that $\frac{q_2}{q_2+\bar h_2}(t,\xi)\leq \gamma$ for all $t\in[0,\tau_2(\xi)]$, and therefore $q_2(t,\xi)=(q_2+\bar h_2)(t,\xi)\frac{q_2}{q_2+\bar h_2}(t,\xi)\leq C(\bar M)\gamma$ for all $t\in[0,\tau_2(\xi)]$. Thus 
  \begin{align*}
  &  \int_{\tau_2}^{\tilde\tau_2}\vert \bar h_2(t,\xi)-\bar{\tilde{h}}_2(t,\xi)\vert dt \\
    &\qquad\qquad\qquad \leq 2 (w_2(\tau_2(\xi),\xi)-\tilde w_2(\tau_2(\xi),\xi))+C(\bar M)T\norm{X_2-\tilde X_2}_{L^\infty_T \bar V}\\ \nn 
    &\qquad\qquad\qquad \leq  C(\bar M)  \Big(T(\norm{X_2-\tilde X_2}_{L^\infty_T \bar V}+\norm{X-\tilde X}_{L^\infty_T \bar V})+\gamma \Gamma(\Theta,\tilde \Theta) \Big). 
  \end{align*}

  As in (\textit{ii}) we can conclude that for all $\xi\in\kappa_{1-\gamma}$
  \begin{equation}\label{eq:estttau5}
    \int_0^T \vert \bar h_2(t,\xi)-\bar{\tilde{h}}_2(t,\xi)\vert dt\leq C(\bar M)  \Big(T(\norm{X_2-\tilde X_2}_{L^\infty_T \bar V}+\norm{X-\tilde X}_{L^\infty_T \bar V})+\gamma \Gamma(\Theta,\tilde \Theta) \Big).
  \end{equation}
  Since $\meas(\kappa_{1-\gamma})\leq C(\bar M)$, we can apply Fubini's theorem and use \eqref{eq:estttau5} to obtain 
  \begin{multline}
    \int_0^T \int_{\kappa_{1-\gamma}} \vert \bar h_2(t,\xi)-\bar{\tilde{h}}_2(t,\xi)\vert d\xi dt \\
    \leq  C(\bar M)  \Big(T(\norm{X_2-\tilde X_2}_{L^\infty_T \bar V}+\norm{X-\tilde X}_{L^\infty_T \bar V})+\gamma \Gamma(\Theta,\tilde \Theta) \Big).
  \end{multline}
  
  (\textit{iv}): First we estimate $\norm{Z_2-\tilde Z_2}_{L^\infty_T \bar W(\kappa_{1-\gamma}^c)}$. For $\xi\in \kappa_{1-\gamma}^c$ we have $Z_2-\tilde Z_2=\bar Z_2-\bar{\tilde{Z}}_2$, and, in particular,  $\bar Z_{2,t}=F(\Theta)\bar Z_2$ and $\bar{\tilde{Z}}_{2,t}=F(\tilde \Theta)\bar{\tilde{Z}}_2$ for all $t\in[0,T]$. Hence 
  \begin{align}\label{eq:estdifzgc2}
    \norm{(Z_2-\tilde Z_2)(t,\dott)}_{\bar W(\kappa_{1-\gamma}^c)}& = \norm{(\bar Z_2-\bar{\tilde{Z}}_2)(t,\dott)}_{\bar W(\kappa_{1-\gamma}^c)}\\ \nn 
    & \leq \int_0^t \norm{(F(\Theta)-F(\tilde \Theta))\bar Z_2(t',\dott)}_{\bar W(\kappa_{1-\gamma}^c)}dt'\\ \nn 
    & \quad +\int_0^t \norm{F(\tilde \Theta)(\bar Z_2-\bar{\tilde{Z}}_2)(t',\dott)}_{\bar W(\kappa_{1-\gamma}^c)}dt'.
  \end{align}
  We have that
  \begin{align}\label{eq:estttau6}
    \left(F(\Theta)-F(\tilde
      \Theta)\right)\bar Z_2& =\Big(0,\big((U^2-P(\Theta))-(\tilde
    U^2-P(\tilde
    \Theta))\big)q_2,\\ \nn
    &\qquad 2\big((U^2-P(\Theta))-(\tilde
    U^2-P(\tilde \Theta))\big)w_2,0\Big),
  \end{align}
  and therefore
  \begin{equation}
    \label{eq:Fbd2}
    \norm{(F(\Theta)-F(\tilde \Theta))\bar Z_2}_{L_T^1\bar W}
    \leq C(\bar M)\norm{\big(U^2-P(\Theta)\big)-\big(\tilde U^2-
      P(\tilde \Theta)\big)}_{L_T^1E}. 
  \end{equation}
  Applying Gronwall's lemma to \eqref{eq:estdifzgc2}, as $\norm{F(\tilde \Theta)}_{L^\infty_T L^\infty}\leq C(\bar M)$, we get 
  \begin{equation}
    \label{eq:gronapp2}
    \norm{Z_2-\tilde Z_2}_{L^\infty_T \bar W(\kappa_{1-\gamma}^c)}\leq
    C(\bar M)\norm{(F(\Theta)-F(\tilde \Theta))\bar Z_2}_{L_T^1\bar W}. 
  \end{equation}
  Hence, we get by \eqref{eq:Fbd2} that
  \begin{equation}
    \label{eq:ltl2est2}
    \norm{Z_2-\tilde Z_2}_{L_T^\infty
      \bar W(\kappa_{1-\gamma}^c)}\leq C(\bar M)\norm{\big(P(\Theta)-U^2\big)- \big(P(\tilde \Theta)-\tilde U^2\big)}_{L_T^1 E}. 
  \end{equation}
  Thus, we have by \eqref{eq:estPQGamma} that 
  \begin{equation}
    \label{eq:pliptau22}
    \norm{Z_2-\tilde Z_2}_{L_T^\infty \bar W(\kappa_{1-\gamma}^c)}\leq C(\bar M)\big(T\norm{X-\tilde X}_{L^\infty_T\bar V}+\Gamma(\Theta,\tilde \Theta)\big).
  \end{equation}
  To estimate $\norm{Z_2-\tilde Z_2}_{L^\infty_T \bar W(\kappa_{1-\gamma})}$, we fix
  $\xi\in\kappa_{1-\gamma}$ and assume without loss of generality that
  $0<\tau_2(\xi)<\tilde\tau_2(\xi)\leq T$.  For $t\in[0,\tau_2(\xi)]$ we can conclude as for
  $\xi\in\kappa_{1-\gamma}^c$ to obtain
  \begin{equation}\label{eq:estttau7}
    \vert Z_2(t,\xi)-\tilde Z_2(t,\xi)\vert \leq C(\bar M) \Big(T\norm{X-\tilde X}_{L^\infty_T\bar V}+\Gamma(\Theta,\tilde \Theta) \Big). 
  \end{equation}
  For $t\in [\tau_2(\xi),\tilde\tau_2(\xi))$ we have $\bar Z_{2,t}=F(\Theta)\bar Z_2$ and
  $\bar{\tilde{Z}}_{2,t}=F(\tilde \Theta)\bar{\tilde{Z}}_2$, but $h_2(t,\xi)-\tilde h_2(t,\xi)=\bar
  h_2(t,\xi)-\bar{\tilde{h}}_2(t,\xi)+l_1(\xi)$. Thus it follows, using \eqref{eq:3.68}, that
  \begin{align}
    \vert (Z_2-\tilde Z_2)(t,\xi)\vert & \leq \vert (Z_2-\tilde Z_2)(\tau_2(\xi),\xi)\vert +\int_{\tau_2}^t \frac 12 l_1(\xi)dt' \\ \nn 
    & \quad +\int_{\tau_2}^t \vert (F(\Theta)-F(\tilde \Theta))Z_2(t',\xi)\vert dt'
 +\int_{\tau_2}^t \vert F(\tilde \Theta)(Z_2-\tilde Z_2)(t',\xi)\vert dt'\\ \nn 
&\leq \vert (Z_2-\tilde Z_2)(\tau_2(\xi),\xi)\vert +\int_{\tau_2}^t \frac 12 \bar{\tilde {h}}_2(t',\xi)dt' \\ \nn 
& \quad +\int_{\tau_2}^t \vert (F(\Theta)-F(\tilde \Theta))Z_2(t',\xi)\vert dt'\\ \nn 
& \quad  +\int_{\tau_2}^t (\vert F(\tilde \Theta)\vert +1)\vert(Z_2-\tilde Z_2)(t',\xi)\vert dt' ,
  \end{align}
  where $(F(\Theta)-F(\tilde \Theta))Z_2=(F(\Theta)-F(\tilde \Theta))\bar Z_2$ is given by
  \eqref{eq:estttau6}, which depends neither on $h_2$ and $\tilde h_2$ nor on $\bar h_2$ and $\bar{\tilde{h}}_2$.  
Applying Gronwall's inequality then yields
  \begin{align}
    \vert (Z_2-\tilde Z_2)(t,\xi)\vert \leq C(\bar M)&  \Big(\vert (Z_2-\tilde Z_2)(\tau_2(\xi),\xi)\vert +\int_{\tau_2}^t \frac 12 \bar{\tilde{h}}_2(t',\xi)dt'\\ \nn 
    & \qquad +\int_{\tau_2}^t \vert (F(\Theta)-F(\tilde \Theta))Z_2(t',\xi)\vert dt' \Big). 
  \end{align}
  Then we get from \eqref{eq:estttau7} and \eqref{eq:Fbd2}, 
  together with 
  \begin{align}
    \int_{\tau_2}^t \frac 12 \bar{\tilde{h}}_2(t',\xi)dt'& \leq w_2(\tau_2(\xi),\xi)-\tilde w_2(\tau_2(\xi),\xi)+C(\bar
    M)T\norm{X_2(\dott,\xi)-\tilde X_2(\dott,\xi)}_{L^\infty_{\tau_2}}\\ 
    \nn & \leq C(\bar M)
    \Big(T(\norm{X_2(\dott,\xi)-\tilde X_2(\dott,\xi)}_{L^\infty_{\tau_2}}\\
    &\qquad\qquad\qquad\qquad +\norm{X-\tilde
      X}_{L^\infty_T\bar V})+\gamma \Gamma(\Theta,\tilde \Theta) \Big) \nn\\
       \nn & \leq C(\bar M)(T\norm{X-\tilde
      X}_{L^\infty_T\bar V}+(T+\gamma)\Gamma(\Theta,\tilde \Theta)),
  \end{align}
  where we used \eqref{eq:3.70b}, \eqref{eq:3.79}, and \eqref{eq:estttau7}, that 
  \begin{equation}\label{eq:estttau10}
    \vert (Z_2-\tilde Z_2)(t,\xi)\vert \leq C(\bar M) \big(T\norm{X-\tilde X}_{L^\infty_T\bar V}+\Gamma(\Theta,\tilde \Theta) \big).
  \end{equation}

  For $t\in[\tilde\tau_2(\xi),T]$, we have $\bar Z_{2,t}=F(\Theta)\bar Z_2$ and $\bar{\tilde{Z}}_{2,t}=F(\tilde \Theta)\bar{\tilde{Z}}_2$, but $h_2(t,\xi)-\tilde h_2(t,\xi)=\bar h_2(t,\xi)-\bar{\tilde{h}}_2(t,\xi)+l_1(\xi)-\tilde l_1(\xi)$. Thus it follows that 
  \begin{align}
    \vert (Z_2-\tilde Z_2)(t,\xi)\vert & \leq \vert (Z_2-\tilde Z_2)(\tilde\tau_2(\xi),\xi)\vert +\int_{\tilde\tau_2}^t \frac 12 \vert l_1(\xi)-\tilde l_1(\xi)\vert dt' \\ \nn 
    & \quad +\int_{\tilde\tau_2}^t \vert (F(\Theta)-F(\tilde \Theta))Z_2(t',\xi)\vert dt'\\ \nn 
    &  \quad +\int_{\tilde\tau_2}^t \vert F(\tilde \Theta)(Z_2-\tilde Z_2)(t',\xi)\vert dt',
  \end{align}
  where $(F(\Theta)-F(\tilde \Theta))Z_2=(F(\Theta)-F(\tilde \Theta))\bar Z_2$ is given by \eqref{eq:estttau6}. Applying Gronwall's inequality then  yields
  \begin{align}
    \vert (Z_2-\tilde Z_2)(t,\xi)\vert& \leq C(\bar M)  \Big(\vert (Z_2-\tilde Z_2)(\tilde\tau_2(\xi),\xi)\vert +\int_{\tilde\tau_2}^t \frac 12 \vert l_1(\xi)-\tilde l_1(\xi)\vert dt'\\ \nn 
    & \qquad +\int_{\tilde \tau_2}^t \vert (F(\Theta)-F(\tilde \Theta))Z_2(t',\xi)\vert dt' \Big). 
  \end{align}
  Then we get from \eqref{eq:estttau10} and \eqref{eq:Fbd2}, together with 
  \begin{align}
    \vert l_1(\xi)-\tilde l_1(\xi)\vert & \leq \alpha \vert h_2(\tilde\tau_2(\xi),\xi)-\tilde h_2(\tilde\tau_2(\xi),\xi)\vert +C(\bar M) T \norm{X_2(\dott,\xi)-\tilde X_2(\dott,\xi)}_{L^\infty_{\tilde\tau_2}}\\ \nn 
    & \leq C(\bar M)\big( T\norm{X-\tilde X}_{L^\infty_T\bar V}+\Gamma(\Theta,\tilde \Theta)\big),
  \end{align}
  where we used \eqref{eq:estttau9}, \eqref{eq:estttau8}, and \eqref{eq:estttau10}, 
  that 
  \begin{equation}\label{eq:estttau11}
    \vert (Z_2-\tilde Z_2)(t,\xi)\vert \leq C(\bar M)\big(T\norm{X-\tilde X}_{L^\infty_T\bar V}+\Gamma(\Theta,\tilde \Theta)\big).
  \end{equation}

  Combining \eqref{eq:estttau7}, \eqref{eq:estttau10}, and \eqref{eq:estttau11}, we get 
  \begin{equation}\label{eq:estttau12}
    \vert (Z_2-\tilde Z_2)(t,\xi)\vert \leq C(\bar M) \big(T\norm{X-\tilde X}_{L^\infty_T\bar V}+\Gamma(\Theta,\tilde \Theta) \big),
  \end{equation}
  for all $t\in[0,T]$. Since $\meas(\kappa_{1-\gamma})\leq C(\bar M)$, the estimate \eqref{eq:estttau12} implies 
  \begin{equation}\label{eq:estttau13}
    \norm{Z_2-\tilde Z_2}_{L^\infty_T\bar W(\kappa_{1-\gamma})} \leq C(\bar M) \big(T\norm{X-\tilde X}_{L^\infty_T\bar V}+\Gamma(\Theta,\tilde \Theta) \big).
  \end{equation}
  Combining \eqref{eq:pliptau22} and \eqref{eq:estttau13}, we get 
  \begin{equation}\label{eq:estdiffZ2}
    \norm{Z_2-\tilde Z_2}_{L^\infty_T\bar W} \leq C(\bar M) \big(T\norm{X-\tilde X}_{L^\infty_T\bar V}+\Gamma(\Theta,\tilde \Theta) \big).
  \end{equation}

  From \eqref{eq:sysfix2} we obtain
  \begin{equation} \label{eq:estdiffU2}
    \begin{aligned}
      \norm{U_2-\tilde U_2}_{L_T^\infty E}&\leq \norm{Q(\Theta)- Q(\tilde \Theta)}_{L_T^1 E} \\
      &  \leq C(\bar
      M) \big(T\norm{X-\tilde X}_{L^\infty_T\bar V}+\Gamma(\Theta,\tilde \Theta) \big)
    \end{aligned}
  \end{equation}
  and 
  \begin{equation}
    \label{eq:estdiffzet2}
    \begin{aligned}
      \norm{\zeta_2-\tilde \zeta_2}_{L_T^\infty
        L^\infty}&\leq T\norm{U_2-\tilde U_2}_{L_T^\infty L^\infty}\\
      &\leq C(\bar M)\big(T\norm{X-\tilde X}_{L^\infty_T\bar V}+\Gamma(\Theta,\tilde \Theta)\big).
    \end{aligned}
  \end{equation}

  Thus adding 
  \eqref{eq:estdiffZ2}, \eqref{eq:estdiffU2}, and \eqref{eq:estdiffzet2},
  we conclude that
  \begin{equation}\label{eq:estdc1}
    \norm{X_2-\tilde X_2}_{L^\infty_T\bar V}\leq C(\bar M)\big(T\norm{X-\tilde X}_{L^\infty_T\bar V}+\Gamma(\Theta,\tilde \Theta)\big).
  \end{equation}
\end{proof}

\begin{remark}
  Recall that in the case of conservative solutions, i.e.,  $\alpha=0$, we have that $\bar
  h(t,\xi)=h(t,\xi)$ for all $\xi\in\Real$ and $t\in\Real$, and hence the above proof simplifies
  considerably in that case. In particular, it suffices to prove (\textit{iv}) since one can
  conclude that $\Gamma(\Theta,\tilde \Theta)\leq C(\bar M)T\norm{X-\tilde X}_{L^\infty_T \bar V}$ as in
  \eqref{eq:gammanorm}.
\end{remark}

\begin{theorem}[Short time solution]
  \label{th:short0}
  Given $M>0$, for any initial data $\Theta_0=(y_0,U_0,y_{0,\xi}, U_{0,\xi}, \bar h_0,h_0,r_0)\in\G\cap B_M$, there exists a
  time $T>0$, which only depends on $M$, such that there exists a unique solution $\Theta=(y,U,y_\xi,U_\xi,\bar h,
  h,r)\in C([0,T],\bar V)$ of \eqref{eq:sysdiss} with $\Theta(0)=\Theta_0$.  Moreover $\Theta(t)\in\G$ for all
  $t\in[0,T]$.
\end{theorem}

\begin{proof} In order to prove the existence and uniqueness of the solution we use an iteration
  argument. By Lemma~\ref{lem:estshort} there exist $T$ and $\bar M$ such that $\mathcal{P}$ is a mapping from $C([0,T], B_{\bar M})$ to $C([0,T], B_{\bar M})$. Now let $\Theta_1(t,\xi)=\Theta_0(\xi)$ for all $t\in[0,T]$ and set 
  $\Theta_{n+1}=\mathcal{P}(\Theta_n)$ and $\Theta_n(0,\xi)=\Theta_0(\xi)$ for all
  $n\in\Natural$. 
Then $\Theta_n$ belongs to $\Ima(\mathcal{P})$ for all $n\in\Natural$, and, in particular, $\Theta_n(t)\in B_{\bar M}$ for all $t\in[0,T]$.   
We  have
  \begin{align*}
    \norm{X_{n+1}-X_n}_{L^\infty_T\bar V}&\leq
    C(\bar M)\big(T\norm{X_n-X_{n-1}}_{L^\infty_T\bar V}+\Gamma(\Theta_n,\Theta_{n-1})\big)\\
    &\leq C(\bar M)\Big(T\big(\norm{X_n-X_{n-1}}_{L^\infty_T\bar V}+\norm{X_{n-1}-X_{n-2}}_{L^\infty_T\bar V}\big) \\
    &\qquad\qquad+\gamma\Gamma(\Theta_{n-1},\Theta_{n-2})\Big)\\
    &\leq C(\bar M)(T+\gamma)\big(\norm{X_n-X_{n-1}}_{L^\infty_T\bar V}+\norm{X_{n-1}-X_{n-2}}_{L^\infty_T\bar V}\big)
  \end{align*}
  where we used Lemma~\ref{lem:contrgamma}.
  Hence, for $T$
  and $\gamma$ small enough, we have
  \begin{equation*}
    \norm{X_{n+1}-X_n}_{L^\infty_T\bar V}
    \leq \frac14 (\norm{X_n-X_{n-1}}_{L^\infty_T\bar V}+\norm{X_{n-1}-X_{n-2}}_{L^\infty_T\bar V}) \quad \text{ for } n\geq 3.
  \end{equation*}
  Summation over all $n\geq 3$ on the left-hand
  side then yields
  \begin{align*}
    \sum_{n=3}^N   \norm{X_{n+1}-X_n}_{L^\infty_T\bar V}&\leq\frac14(\sum_{n=2}^{N-1}\norm{X_{n+1}-X_n}_{L^\infty_T\bar V}+\sum_{n=1}^{N-2}\norm{X_{n+1}-X_n}_{L^\infty_T\bar V}) \\
 \intertext{and}
    \frac12\sum_{n=1}^N   \norm{X_{n+1}-X_n}_{L^\infty_T\bar V}&\leq \norm{X_2-X_1}_{L^\infty_T\bar V}+\norm{X_3-X_2}_{L^\infty_T\bar V}
  \end{align*}
  independently of $N$. Therefore it implies that $\{X_n\}_{n=1}^\infty$ is a Cauchy sequence which
  converges to a unique limit $X$. 
In addition,  Lemma~\ref{lem:contrgamma} (\textit{i})-(\textit{ii}) implies that $h_n(t)$ converges to a unique limit $h(t)$ in $L^1(\Real)\cap L^2(\Real)$ and that $\bar h_n(t)$ converges to a unique limit $\bar h(t)$ in $L^1(\Real)\cap L^2(\Real)$. 
 
Next we want to show that for almost every $\xi\in\Real$ such that $\tau_1(\xi)\leq T$, we have 
\begin{equation*}
 X(\tau_1(\xi),\xi)=X(\tau_1(\xi)-0,\xi) \quad \text{ and }\quad \bar h(\tau_1(\xi),\xi)=(1-\alpha)\bar h(\tau_1(\xi)-0,\xi).
\end{equation*}
Let $\mathcal{A}$ be the following set 
\begin{align}
 \mathcal{A}=\{\xi\in\Real\mid & \, \vert \zeta_0(\xi)\vert\leq \norm{\zeta_0}_{L^\infty}, \vert U_0(\xi)\vert \leq \norm{U_0}_{L^\infty}, \vert q_0(\xi)-1\vert \leq \norm{q_0-1}_{L^\infty},\\ \nn 
&\quad  \vert w_0(\xi)\vert \leq\norm{w_0}_{L^\infty}, \vert h_0(\xi)\vert \leq \norm{h_0}_{L^\infty}, \vert r_0(\xi)\vert \leq \norm{r_0}_{L^\infty}\}.
\end{align}
We have that $\mathcal{A}$ has full measure, that is, $\meas(\mathcal{A}^c)=0$. Recall that
\begin{equation*}
\norm{X(\dott,\xi)}_{L^\infty_T}= \sup_{t\in[0,T]}\{ \vert y(t,\xi)-\xi\vert +\vert U(t,\xi)\vert +\vert q(t,\xi)\vert +\vert w(t,\xi)\vert +\vert h(t,\xi)\vert +\vert r(t,\xi)\vert\}.
\end{equation*}
Since  both $P(\Theta_n)(t,\xi)$ and $Q(\Theta_n)(t,\xi)$ belong to $H^1(\Real)$ for all $t\in [0,T]$, we have that 
\begin{equation*}
 \sup_{t\in[0,T]}\{\vert P(\Theta_n)(t,\xi)\vert +\vert Q(\Theta_n)(t,\xi)\vert \}\leq \norm{P(\Theta_n)}_{L^\infty_TE}+\norm{Q(\Theta_n)}_{L^\infty_T E}\leq  C(\bar M).
\end{equation*}
Moreover, following closely the proof of Lemma~\ref{lem:estshort}, we obtain that for any $\xi\in\mathcal{A}$ 
\begin{align*}
\norm{X_n(\dott,\xi)}_{L^\infty_T}\leq \norm{X_n}_{L^\infty_T\bar V}\leq C(\bar M),
\end{align*}
which  implies that $X_n(t,\xi)$ is continuous with respect to time. In particular, one obtains that for any $\xi\in\mathcal{A}$, 
\begin{equation}
 \norm{X_n(\dott,\xi)-X_{n-1}(\dott,\xi)}_{L^\infty_T}\leq \norm{X_n-X_{n-1}}_{L^\infty_T \bar V}.
\end{equation}
Thus 
$\norm{X_n(\dott,\xi)-X_{n-1}(\dott,\xi)}_{L^\infty_T} \to 0$ and $X_n(\xi)$ converges to the unique limit $X(\xi)$ for almost every $\xi\in\Real$. Thus if we can show that 
\begin{equation}\label{esttauest}
 \bar h(\tau_1(\xi),\xi)=\bar h(\tau_1(\xi)+0,\xi)=(1-\alpha)\bar h(\tau_1(\xi)-0,\xi),
\end{equation}
for all $\xi\in\mathcal{A}$ that experience wave breaking within $[0,T]$, $\Theta$ will be a solution of \eqref{eq:sysdiss} in the sense of Definition~\ref{def:sol}. Recall that for any $n\in\mathbb{N}$ we have  that 
\begin{equation}
 \bar h_n(\tau_{1,n}(\xi),\xi)=\bar h_n(\tau_{1,n}(\xi)+0,\xi)=(1-\alpha)\bar h_n(\tau_{1,n}(\xi)-0,\xi).
\end{equation}
Thus if we can show that $\tau_{1,n}(\xi)$ converges to a unique limit $\tau_1(\xi)$, the claim will follow since $\norm{X_n(\dott,\xi)-X(\dott,\xi)}_{L^\infty_T} \to 0$. We assume without loss of generality that $0< \tau_{1,n-1}(\xi)<\tau_{1,n}(\xi)\leq T$, since all other possible cases can be handled similarly. Moreover, we assume that $q_{0}(\xi)+h_0(\xi)-l_0(\xi)=C>0$, since otherwise $\bar h_0(\xi)=0=q_0(\xi)$, and \eqref{esttauest} is obviously satisfied. Then, as in the proof of Lemma~\ref{lem:G} (\textit{ii}), we can find a strictly positive constant $C_T$ such that $C_T<q(t,\xi)+h(t,\xi)-l_0(\xi)$ for all $t\in[0,T]$. In particular, we get 
\begin{equation}
 \tau_{1,n}(\xi)-\tau_{1,n-1}(\xi)=\int_{\tau_{1,n-1}(\xi)}^{\tau_{1,n}(\xi)} \frac{q_{n}(s,\xi)+\bar h_n(s,\xi)}{q_n(s,\xi)+\bar h_n(s,\xi)}ds\leq C_T^{-1}\int_{\tau_{1,n-1}(\xi)}^{\tau_{1,n}(\xi)} (q_{n}+\bar h_n)(s,\xi) ds,
\end{equation}
where we used that $\bar h_n(t,\xi)=h_n(t,\xi)-l_0(\xi)$ for $t\in[0,\tau_{1,n}(\xi))$. We split the integral on the right-hand side into two and study them separately.
For the first integral we get
\begin{align}
 \int_{\tau_{1,n-1}(\xi)}^{\tau_{1,n}(\xi)} q_n(s,\xi)ds& =\int_{\tau_{1,n-1}(\xi)}^{\tau_{1,n}(\xi)} \big(q_n(s,\xi)-q_n(\tau_{1,n-1}(\xi),\xi)+q_n(\tau_{1,n-1}(\xi),\xi)\big)ds\\ \nn 
& \leq\int_{\tau_{1,n-1}(\xi)}^{\tau_{1,n}(\xi)} q_n(\tau_{1,n-1}(\xi),\xi)ds\\ \nn
& =\int_{\tau_{1,n-1}(\xi)}^{\tau_{1,n}(\xi)} \big(q_n(\tau_{1,n-1}(\xi),\xi)-q_{n-1}(\tau_{1,n-1}(\xi),\xi)\big)ds\\ \nn
& \leq T \norm{q_n(\dott,\xi)-q_{n-1}(\dott,\xi)}_{L^\infty_T}\\ \nn 
& \leq T \norm{X_n-X_{n-1}}_{L^\infty_T \bar V}
\end{align}
where we used that $q_{n}(s,\xi)$ is decreasing on the interval $[\tau_{1,n-1}(\xi),\tau_{1,n}(\xi)]$, since $q_{n,t}(s,\xi)=w_n(s,\xi)\leq 0$ for all $s\in[\tau_{1,n-1}(\xi),\tau_{1,n}(\xi)]$ and that $q_{n-1}(\tau_{1,n-1}(\xi),\xi)=0$.
As far as the second integral is concerned, we can conclude as follows
\begin{align}
 \int_{\tau_{1,n-1}(\xi)}^{\tau_{1,n}(\xi)} \bar h_n(s,\xi)ds 
& = 2(w_{n}(\tau_{1,n}(\xi),\xi)-w_n(\tau_{1,n-1}(\xi),\xi))\\ \nn 
& \quad -2\int_{\tau_{1,n-1}(\xi)}^{\tau_{1,n}(\xi)} (U^2_{n-1}-P(\Theta_{n-1}))q_n(s,\xi)ds\\ \nn 
& \leq -2w_{n}(\tau_{1,n-1}(\xi),\xi)+C(\bar M)Tq_n(\tau_{1,n-1}(\xi),\xi)\\ \nn 
& \leq 2(w_{n-1}(\tau_{1,n-1}(\xi),\xi)-w_n(\tau_{1,n-1}(\xi),\xi))\\ \nn 
& \quad +C(\bar M)T(q_n(\tau_{1,n-1}(\xi),\xi)-q_{n-1}(\tau_{1,n-1}(\xi),\xi))\\ \nn 
& \leq \norm{w_n(\dott,\xi)-w_{n-1}(\dott,\xi)}_{L^\infty_T}\\ \nn
&\quad+C(\bar M)T\norm{q_n(\dott,\xi)-q_{n-1}(\dott,\xi)}_{L^\infty_T}\\ \nn 
& \leq (1+TC(\bar M))\norm{X_n-X_{n-1}}_{L^\infty_T \bar V},
\end{align}
where we used $w_n(\tau_{1,n}(\xi),\xi)=w_{n-1}(\tau_{1,n-1}(\xi),\xi)=q_{n-1}(\tau_{1,n-1}(\xi),\xi)=0$. Thus the sequence $\tau_{1,n}(\xi)$ converges to a unique limit $\tau_{1}(\xi)$ for every $\xi\in\mathcal{A}$, and, in particular, $\lim_{n\to\infty}h_n(\tau_{1,n}(\xi),\xi)=h(\tau_1(\xi),\xi)$ for all $\xi\in\mathcal{A}$. 
This implies since $\bar h(t,\xi)=\lim_{n\to\infty}\bar h_n(t,\xi)$ for $t\not =\tau_1(\xi)$, that 
\begin{align}
 \bar h(\tau_1(\xi),\xi)& =\lim_{s\downarrow \tau_1(\xi)} \bar h(s,\xi)\\ \nn
& =\lim_{s\downarrow \tau_1(\xi)}\lim_{n\to\infty} \bar h_n(s,\xi)\\ \nn 
& =\lim_{s\downarrow \tau_1(\xi)}\lim_{n\to\infty} \Big(h_n(s,\xi)-l_{1,n}(\xi)-l_0(\xi)\Big)\\ \nn
& =\lim_{s\downarrow \tau_1(\xi)}\lim_{n\to\infty}\Big(h_n(s,\xi)-\alpha(h_n(\tau_{1,n}(\xi), \xi)-l_0(\xi))-l_0(\xi)\Big)\\ \nn 
& =\lim_{s\downarrow \tau_1(\xi)} \Big(h(s,\xi)-\alpha(h(\tau_1(\xi),\xi)-l_0(\xi))-l_0(\xi)\Big)\\ \nn
& =h(\tau_1(\xi),\xi)-\alpha(h(\tau_1(\xi),\xi)-l_0(\xi))-l_0(\xi)\\ \nn
& =(1-\alpha)(h(\tau_1(\xi),\xi)-l_0(\xi))
\end{align}
and 
\begin{align}
 \bar h(\tau_1(\xi)-0,\xi)& =\lim_{s\uparrow \tau_1(\xi)}\bar h(s,\xi)\\ \nn
& =\lim_{s\uparrow \tau_1(\xi)} \lim_{n\to\infty} \bar h_n(s,\xi) \\ \nn 
& =\lim_{s\uparrow \tau_1(\xi)}\lim_{n\to \infty} (h_n(s,\xi)-l_0(\xi))\\ \nn 
& =\lim_{s\uparrow \tau_1(\xi)} h(s,\xi)-l_0(\xi)\\ \nn 
& =h(\tau_1(\xi),\xi)-l_0(\xi).
\end{align}
Thus 
\begin{equation}
 \bar h(\tau_1(\xi),\xi)=(1-\alpha)(h(\tau_1(\xi),\xi)-l_0(\xi))=(1-\alpha)\bar h(\tau_1(\xi)-0,\xi),
\end{equation}
and, in particular,
\begin{equation}
 \bar h(t,\xi)=\begin{cases}
                h(t,\xi)-l_0(\xi), & \quad \text{ for }t<\tau_1(\xi),\\
		h(t,\xi)-l_0(\xi)-l_1(\xi), & \quad \text{ otherwise},
               \end{cases}
\end{equation}
where $l_1(\xi)=\alpha(h(\tau_1(\xi),\xi)-l_0(\xi))=\lim_{n\to\infty}l_{1,n}(\xi)$.

It is left to prove that $U$ and $y$ are differentiable and that $U_\xi=w$ and $y_\xi=q$. Recall that $Q(\Theta)$ is defined via \eqref{eq:Qlag3} and choose $\xi_1$, $\xi_2\in\Real$ such that $\xi_1<\xi_2$, then we have 
\begin{align}\label{eq:lipschitz}
 & \int_{-\infty}^{\xi_2} e^{-\vert y(t,\xi_2)-y(t,\eta)\vert} (2U^2q+\bar h)(t,\eta)d\eta-\int_{-\infty}^{\xi_1}e^{-\vert y(t,\xi_1)-y(t,\eta)\vert}(2U^2q+\bar h)(t,\eta)d\eta\\ \nn
&\qquad=\int_{-\infty}^{\xi_1} (e^{-\vert y(t,\xi_2)-y(t,\eta)\vert}-e^{-\vert y(t,\xi_1)-y(t,\eta)\vert}) (2U^2q+\bar h)(t,\eta)d\eta\\ \nn 
&\quad\qquad +\int_{\xi_1}^{\xi_2}e^{-\vert y(t,\xi_2)-y(t,\eta) \vert }(2U^2q+\bar h)(t,\eta)d\eta\\ \nn 
& \qquad \leq C(\bar M)(\vert y(t,\xi_2)-y(t,\xi_1)\vert +\xi_2-\xi_1).
\end{align}
Here we used that 
\begin{align}
 \vert e^{-\vert y(t,\xi_2)-y(t,\eta)\vert}-e^{-\vert y(t,\xi_1)-y(t,\eta)\vert}\vert
& =\vert \int_{-\vert y(t,\xi_1)-y(t,\eta)}^{-\vert y(t,\xi_2)-y(t,\eta)\vert} e^x dx\vert \\ \nn
& \leq \vert \vert y(t,\xi_1)-y(t,\eta)\vert -\vert y(t,\xi_2)-y(t,\eta)\vert \vert \\ \nn
& \leq \vert y(t,\xi_1)-y(t,\xi_2)\vert.
\end{align}
Thus \eqref{eq:lipschitz} implies that $Q(\Theta)$ is differentiable almost everywhere according to Rademacher's theorem if $y(t,\xi)$ is Lipschitz continuous, 
since 
\begin{equation}\label{est:QQQ}
 Q(\Theta)(t,\xi_2)-Q(\Theta)(t,\xi_1)\leq C(\bar M)(\vert y(t,\xi_2)-y(t,\xi_1)\vert +\xi_2-\xi_1).
\end{equation}
Therefore observe that 
\begin{equation}
 y(t,\xi_2)-y(t,\xi_1)=\int_0^t \big(U(s,\xi_2)-U(s,\xi_1)\big)ds+y(0,\xi_2)-y(0,\xi_1)
\end{equation}
and 
\begin{equation}\label{eq:UUU}
 U(t,\xi_2)-U(t,\xi_1)=-\int_0^t \big(Q(\Theta)(s,\xi_2)-Q(\Theta)(s,\xi_1)\big)ds+U(0,\xi_2)-U(0,\xi_1),
\end{equation}
where $U(0,\xi)$ and $y(0,\xi)$ are Lipschitz continuous due to the assumptions on the initial data. 
Combining these two inequalities and \eqref{est:QQQ} yields
\begin{align}
 \vert  y(t,\xi_2)-y(t,\xi_1)\vert & \leq \int_0^t \vert U(s,\xi_2)-U(s,\xi_1)\vert ds+\vert y(0,\xi_2)-y(0,\xi_1)\vert \\ \nn 
& \leq\int_0^t\int_0^s \vert Q(\Theta)(r,\xi_2)-Q(\Theta)(r,\xi_1)\vert dr ds\\ \nn 
& \quad +\int_0^t \vert U(0,\xi_2)-U(0,\xi_1) \vert ds+\vert y(0,\xi_2)-y(0,\xi_1)\vert\\ \nn 
& \leq C(\bar M)(\int_0^t \vert y(s,\xi_2)-y(s,\xi_1)\vert ds +\xi_2-\xi_1).
\end{align}
Applying Gronwall's inequality yields 
\begin{equation}
 \vert y(t,\xi_2)-y(t,\xi_1)\vert \leq C(\bar M)\vert \xi_2-\xi_1\vert.
\end{equation}
Thus $y(t,\xi)$ is Lipschitz continuous and differentiable almost everywhere. As an immediate consequence, we get from \eqref{est:QQQ} and \eqref{eq:UUU} that also $Q$ and $U$ are Lipschitz continuous and therefore differentiable almost everywhere.

We are now ready to show that $w=U_\xi$ and $q=y_\xi$. Therefore recall that $Q(\Theta)$ is defined via
  \eqref{eq:Qlag3}, and note that $Q(\Theta)$ is differentiable since $y$ is differentiable. A direct
  computation gives us that
  \begin{equation}  \label{eq:derQxi}
    Q_\xi(\Theta) =-\frac12 \bar h-U^2q+P(\Theta)y_\xi.
  \end{equation}
 In addition, as $q(t,\xi)$ and $w(t,\xi)$ are both continuous with respect to time,
  we have 
  \begin{subequations}
    \label{eq:qQxi}
    \begin{align}
      (q-y_\xi)_t&=(w-U_\xi),\\
      (w-U_\xi)_t&=-P(\Theta)(q-y_\xi). 
    \end{align}
  \end{subequations}
  In particular, this means that  if $q_0=y_{0,\xi}$ and $w_0=U_{0,\xi}$, then
  \begin{multline*}
    \norm{(q-y_\xi)(t,\dott)}_{E}+\norm{(w-U_\xi)(t,\dott)}_{E}\\
    \leq C(\bar M) \int_0^t (\norm{(q-y_\xi)(t',\dott)}_{E}+\norm{(w-U_\xi)(t',\dott)}_{E})dt',
  \end{multline*}
  and thus using Gronwall's inequality yields that $y_\xi=q$ and $U_\xi=w$.

  Let us
  prove that $X(t)\in\G$ for all $t$. From
  \eqref{eq:G1} and \eqref{eq:G2} we get
  $q(t,\xi)\geq0$, $h(t,\xi)\geq0$ and
  $q\bar h=w^2+r^2$ for all $t$ and almost all $\xi$, 
  and therefore, since $U_\xi=w$ and $y_\xi=q$,
  conditions \eqref{eq:lagcoord3} and
  \eqref{eq:lagcoord6} are fulfilled and $y$ is an increasing function. Since
  $\zeta(t,\xi)=\zeta(0,\xi)+\int_0^t
  U(t,\xi)\,dt$, we obtain by the Lebesgue dominated
  convergence theorem that 
  $\lim_{\xi\to-\infty}\zeta(t,\xi)=0$ because
  $U\in H^1(\Real)$. Hence, since in addition
  $X(t)\in B_{\bar M}$, the function $X(t)$ fulfills all the
  conditions listed in \eqref{eq:lagcoord}, and thus $X(t)\in\G$. 
\end{proof}

  Note that the set $\G\cap B_M$ is closed with respect to the topology of $\bar V$. We have
  \begin{align*}
    y_{\xi,t}&=U_{\xi}, \\
    h_t&=2(U^2-P(\Theta))U_\xi,\\
    r_t&=0,
  \end{align*}
  for all $\xi\in\Real$ and $t\in\Real_+$.  In particular, this means that $y_\xi$, $h$, and $r$ are
  differentiable with respect to time in the classical sense almost everywhere.

In order to obtain global solutions, we want to apply  Theorem~\ref{th:short0} iteratively, which is
possible if we can show that $\norm{X}_{\bar V}+\norm{h}_{L^1}+\norm{\frac{1}{y_\xi+h}}_{L^\infty}$ does not blow up within finite time. The
corresponding estimate is contained in the following lemma.

\begin{lemma}
 \label{lem:globest}
Given $M$ and $T_0$, then there exists a constant $M_0$ which only depends on $M$ and $T_0$ such that, for any $\Theta_0=(y_0,U_0,y_{0,\xi},U_{0,\xi},\bar h_0, h_0, r_0)\in B_{M}$, the following holds for all $t\in[0,T_0]$,
\begin{equation}
\norm{X(t)}_{\bar V}+\norm{h(t)}_{L^1}+\norm{\frac{1}{y_\xi+h}(t)}_{L^\infty}\leq M_0
\end{equation}
and 
\begin{equation}
 \int_\Real \big(U^2y_\xi(t,\xi)+h(t,\xi)\big)d\xi=\int_\Real \big(U_0^2y_{0,\xi}(\xi)+h_0(\xi)\big)d\xi.
\end{equation}
\end{lemma}

\begin{proof} This proof follows the same lines as the one in \cite{HolRay:07}. To simplify the
notation we will generically denote by $C$ constants and by $C(M,T_0)$ constants which in addition
depend on $M$ and $T_0$. Let us introduce
  \begin{equation*}
    \Sigma(t)=\int_\Real (U^2y_\xi+h)(t,\xi)d\xi.
  \end{equation*}
  Since $h\geq 0$, we have $\norm{h}_{L^1_\Real}=\int_\Real
  h\,d\xi<\infty$. After some computation, \eqref{eq:sysdiss} yields that
  \begin{equation}\label{eq:Sigma}
    \Sigma(t)=\Sigma(0) \text{ for all }t\in\Real_+,
  \end{equation}
which implies 
\begin{equation}
 \norm{h(t,\dott)}_{L^1}\leq \Sigma(0).
\end{equation}
  Moreover we have  
  \begin{subequations}
    \label{eq:derivULinf}
    \begin{align*}
      U^2(t,\xi)&=2\int_{-\infty}^{\xi}UU_\xi(t,\eta)\,d\eta\\
      &\leq \int_{\{\eta\mid  y_\xi(\eta)>0\}}\Big(
      U^2y_\xi+\frac{U_\xi^2}{y_\xi}\Big)(t,\eta)\,d\eta\\
      &\leq \int_{\{\eta\mid  y_\xi(\eta)>0\}}(U^2y_\xi+h)(t,\eta)\,d\eta\\
      &\leq \Sigma(t)=\Sigma(0),
    \end{align*}
  \end{subequations}
  where we used that $y_\xi(\eta)=0$ implies $U_\xi(\eta)=0$, and therefore the integrand in the
  integral in the first line vanishes whenever $y_\xi(\xi)=0$. Thus it suffices to integrate
  over $\{\eta\in\Real\mid y_\xi(\eta)>0\} \cap \{\eta\leq\xi\}$ which justifies the subsequent
  estimate. Thus
  \begin{equation}\label{eq:2951}
    \norm{U(t,\dott)}_{L^\infty}^2\leq\Sigma(0).
  \end{equation}
  Moreover, $P$ and $Q$ satisfy
  \begin{equation}\label{eq:2952}
    \norm{P(\Theta)(t,\dott)}_{L^\infty}\leq 2\Sigma(0)\quad \text{ and }\quad \norm{Q(\Theta)(t,\dott)}_{L^\infty}\leq 2\Sigma(0).
  \end{equation}
  From \eqref{eq:sysdiss}, we obtain that 
  \begin{equation}
    \vert \zeta(t,\xi)\vert \leq \vert \zeta(0,\xi)\vert +\int_0^t\abs{U(t',\xi)}dt', 
  \end{equation}
  and hence 
  \begin{equation}
    \norm{\zeta(t,\dott)}_{L^\infty}\leq\norm{\zeta(0,\dott)}_{L^\infty}+T(1+\Sigma(0)).
  \end{equation}
  Applying Young's inequality to \eqref{eq:Plag1} and \eqref{eq:Qlag1} and following the proof of
  Lemma~\ref{lem:PQ} we get
  \begin{equation}
    \norm{P(\Theta)(t,\dott)}_{L^2}+\norm{Q(\Theta)(t,\dott)}_{L^2}\leq
    Ce^{2\norm{\zeta(t,\dott)}_{L^\infty}}\Sigma(0).
  \end{equation}
  Let 
  \begin{equation*}
    \alpha(t)=\norm{U(t,\dott)}_{E}+\norm{\zeta_\xi(t,\dott)}_{E}+\norm{U_\xi(t,\dott)}_{E}+\norm{h(t,\dott)}_{E}+\norm{r(t,\dott)}_{E}.
  \end{equation*}
  Then
  \begin{equation}
    \alpha(t)\leq \alpha(0)+C(M,T_0)+C(M,T_0)\int_0^t \alpha(t') dt'.
  \end{equation}
  Hence Gronwall's lemma gives us $\alpha(t)\leq C(M,T_0)$. It remains to prove that
  $\norm{\frac{1}{y_\xi+h}}_{L^\infty_TL^\infty}$ can be bounded by some constant depending on $M$
  and $T_0$, but this follows immediately form \eqref{eq:G3}. This completes the proof.
\end{proof}

We can now prove  global existence of solutions.
\begin{theorem}[Global solution] 
  \label{th:global}
  For any initial data $\Theta_0=(y_0,U_0,y_{0,\xi}, U_{0,\xi}, h_0, \bar h_0, r_0)\in\G$, there exists a unique global
  solution $\Theta=(y,U,y_{\xi}, U_{\xi},\bar h,h,r)\in C(\Real_+,\G)$ of \eqref{eq:sysdiss} with $\Theta(0)=\Theta_0$.
\end{theorem}

\begin{proof}
By assumption $\Theta_0\in\G$, and therefore there exists a constant $M$ such that $\Theta_0\in B_{M}$. By Theorem~\ref{th:short0} there exists a $T>0$, dependent on $M$, such that we can find a unique solution
  $\Theta(t)\in \G$ on $[0,T]$. Thus we can find a global solution to \eqref{eq:sysdiss} if and only if $\norm{X(t)}_{\bar V}+\norm{h(t)}_{L^1}+\norm{\frac{1}{y_\xi+h}(t)}_{L^\infty}$ does not blow up within a
  finite time interval, but this follows from Lemma~\ref{lem:globest}.
\end{proof}

Observe that
$(\zeta,U,\zeta_\xi,U_\xi,\bar h, h,r)$ is a fixed point of
$\PP$, and the results of Lemma~\ref{lem:G}
hold for $\Theta=\tilde \Theta=(\zeta,U,\zeta_\xi,U_\xi,\bar h,h,r)$. 
Since this lemma contains important information about which points will experience wave breaking in the near future, we rewrite it for the fixed point solution $\Theta$. For
this purpose, we redefine $B_M$ and
$\kappa_{1-\gamma}$, see \eqref{eq:defBMfix} and \eqref{eq:defKgamma}, as
\begin{equation*}
  B_M=\{ \Theta \mid \norm{X}_{\bar V}+\norm{h}_{L^1}+\norm{\frac{1}{y_\xi+h}}_{L^\infty}\leq M\},
\end{equation*}
where $X=(\zeta,U,\zeta_\xi,U_\xi,h,r)$, and
\begin{equation}
  \label{eq:defKgamma2}
  \kappa_{1-\gamma}=\{\xi\in\Real\mid \frac{\bar h_0}{y_{0,\xi}+\bar h_0}(\xi)\geq 1-\gamma\text{, } U_{0,\xi}(\xi)\leq 0, \text{ and } r_0(\xi)=0\}, \quad \gamma\in[0,\frac12].
\end{equation}
Note that every condition imposed on points $\xi\in \kappa_{1-\gamma}$ is motivated by what is known about wave breaking. If wave breaking occurs at some time $t_b$, then energy is concentrated on sets of measure zero in Eulerian coordinates, which correspond to the sets where $\frac{h}{y_\xi+h}(t_b,\xi)=1=\frac{\bar h}{y_\xi+\bar h}(t_b-0,\xi)$ in Lagrangian coordinates. Furthermore, it is well-known that wave breaking in the context of the 2CH system means that the spatial derivative becomes unbounded from below and hence $U_\xi(t,\xi)\leq 0$ for $t_b-\delta\leq t\leq t_b$ for such  points, see \cite{ConstantinIvanov:2008, GY}. Finally, it has been shown in \cite[Theorem 6.1]{GHR4} that wave breaking within finite time can only occur at points 
$\xi$ where $r_0(\xi)=0$.

Lemma \ref{lem:G} 
rewrites, due to \eqref{eq:2951} and \eqref{eq:2952}, as follows.

\begin{corollary}\label{lem:G2}
  Let  $M_0$ be a constant, and consider initial data $\Theta_0\in \mathcal{G}\cap B_{M}$. Denote by $\Theta=(\zeta, U, \zeta_\xi, U_\xi,\bar h,h,r)\in C(\Real_+,\mathcal{G})$ the global solution of \eqref{eq:sysdiss} with initial data $\Theta_0$.  
  Then the following statements hold:

  (\textit{i}) We have 
  \begin{equation}\label{eq:G3sol}
    \norm{\frac{1}{y_\xi+h}(t,\dott)}_{L^\infty}\leq 2e^{C(M)T}\norm{\frac{1}{y_{0,\xi}+h_0}}_{L^\infty},
  \end{equation}
  and 
  \begin{equation}
    \label{eq:G3bsol}
    \norm{(y_\xi+h)(t,\dott)}_{L^\infty}\leq 2e^{C(M)T}\norm{y_{0,\xi}+h_0}_{L^\infty}
  \end{equation}
  for all $t\in[0,T]$ and a constant $C(M)$ which depends on $M$. 

  (\textit{ii}) There exists a $\gamma\in(0,\frac12)$ depending only on $M$ such that if $\xi\in \kappa_{1-\gamma}$, then $\Theta(t,\xi)\in \Omega_1$ for all $t\in [0,\min(\tau_1(\xi),T)]$, $\frac{y_\xi}{y_\xi+\bar h}(t,\xi)$ is a decreasing function and $\frac{U_\xi}{y_\xi+\bar h}(t,\xi)$ is an increasing function, both with respect to time for $t\in[0,\min(\tau_1(\xi),T)]$.  Therefore we have 
  \begin{equation}
    \frac{U_{0,\xi}}{y_{0,\xi}+\bar h_0}(\xi)\leq\frac{U_\xi}{y_\xi+\bar h}(t,\xi)\leq 0\quad \text{and} \quad 0\leq \frac{y_\xi}{y_\xi+\bar h}(t,\xi)\leq \frac{y_{0,\xi}}{y_{0,\xi}+\bar h_0}(\xi),
  \end{equation}
  for $t\in[0,\min(\tau_1(\xi),T)]$.
  In addition, for $\gamma$ sufficiently small, depending only on $M$ and $T$, we have 
  \begin{equation}\label{eq:G5sol}
    \kappa_{1-\gamma}\subset \{\xi\in\Real\mid 0\leq \tau_1(\xi)<T\}.
  \end{equation}

  (\textit{iii})
  Moreover, for any given $\gamma>0$, there exists $\hat T>0$ such that 
  \begin{equation}\label{eq:G4sol}
    \{\xi\in\Real\mid 0<\tau_1(\xi)<\hat T\}\subset \kappa_{1-\gamma}.
  \end{equation}
\end{corollary}

Although we have now constructed a new class of solutions in Lagrangian coordinates, there is one
more fact we want to point out. The construction of $\alpha$-dissipative solutions involves the
sequence of breaking times $\{\tau_j(\xi)\}$ for every point $\xi$. At first sight it is not clear
that this possibly infinite sequence does not accumulate.

\begin{corollary}\label{cor:tau}
  Denote by $\Theta(t)=(y,U,y_\xi, U_\xi,\bar h,h,r)(t)$ the global solution of \eqref{eq:sysdiss} with $\Theta(0)=\Theta_0\in\G\cap B_M$ in $C(\Real_+,\G)$.  For any $\xi\in\Real$ the possibly infinite sequence
  $\tau_j(\xi)$ cannot accumulate.

In particular, there exists a time $\hat T$ depending on $M$ such that any point $\xi$ can experience
  wave breaking at most once within the time interval $[T_0,T_0+\hat T]$ for any $T_0\ge0$. More
  precisely, given $\xi\in\Real$, we have
  \begin{equation}
  \tau_{j+1}(\xi)-\tau_j(\xi)>\hat T \text{ for all } j. 
  \end{equation}
  In addition,
  for $\hat T$ sufficiently small, we get that in this case $U_\xi(t,\xi)\geq 0$ for all
  $t\in[\tau_j(\xi), \tau_j(\xi)+\hat T]$.
\end{corollary}

\begin{proof}
  In the proof of Lemma~\ref{lem:globest}, we showed  
  \begin{equation}\label{eq:conen}
    \norm{P(\Theta)}_{L^{\infty}_{\infty} L^\infty}+\norm{Q(\Theta)}_{L^\infty_\infty L^\infty}+\norm{U}_{L^\infty_\infty L^\infty}^2\leq 5\Sigma(0),
  \end{equation}
  where $L^{\infty}_{\infty}=L^{\infty}_{T=\infty}$.
  This means, in particular, that the constant $C(\bar M)$ in the proof of Lemma~\ref{lem:break}, for the global solution, can be chosen to be independent of time. Thus we can conclude from Lemma~\ref{lem:break}, that there exists a constant $\hat T$, such that $\tau_{j+1}(\xi)-\tau_j(\xi)> \hat T$ for all $j$.
\end{proof}

\section{From Eulerian to Lagrangian variables and vice versa}

Let us define in detail our variables in Eulerian coordinates. As explained in the introduction, the energy distribution can concentrate and therefore our set of Eulerian variables does not only contain  the functions $u(t)$ and $\rho(t)$ but also a measure $\mu(t)$, which properly describes the concentrated amount of energy at breaking times. This measure $\mu(t)$, which describes only part of the energy in general, is treated as an independent variable, but still remains strongly connected to $u(t)$ and $\rho(t)$ through its absolutely continuous part, see \eqref{eq:abspart} below. In addition, in order to enable the construction of the semigroup, we add to the set of Eulerian variables the measure $\nu(t)$, which allows us, together with $\mu(t)$, to determine how much energy has been dissipated. For the solution we construct, see Section \ref{sec:semigroup}, the measure $\mu(t)$
is in general discontinuous in time while $\nu(t)$ remains continuous.

\begin{definition} [Eulerian coordinates]\label{def:D}
 The set $\D$ is
  composed of all $(u,\rho,\mu,\nu)$ such that 
  \begin{enumerate}
  \item $u\in H^1(\Real)$,
  \item $\rho\in L^2(\Real)$,
  \item $\mu$ is a positive finite
    Radon measure whose absolutely 
    continuous part,  $\muac$,  satisfies 
    \begin{equation}\label{eq:abspart}
      \muac=(u_x^2+\rho^2)\,dx,
    \end{equation}
  \item $\nu$ is a positive finite Radon measure such that $\mu\leq \nu$.
  \end{enumerate}
\end{definition}
Note that $\mu\leq \nu$ implies that $\mu$ is absolutely continuous with respect to $\nu$ and
therefore there exists a measureable function $f$ such that
\begin{equation}\label{eq:radon}
  \mu = f\nu\quad \text{ and }\quad 0\leq f\leq 1.
\end{equation}

\begin{remark}
 At first sight it might seem surprising that we need two measures to be able to construct a semigroup of solutions, but both of them play an essential role. 

The measure $\mu$, on the one hand, describes the concentrated amount of energy at breaking times, and is therefore, in general, discontinuous with respect to time. Moreover, it helps to measure the total energy $E(t)$ at any time, since
\begin{equation}
 E(t)=\int_\Real u^2(t,x)dx+\mu(t,\Real).
\end{equation}
Thus also the energy is in general a discontinuous function, and, in particular, drops suddenly at breaking times if $\alpha\not =0$, while it is preserved for all times in the conservative case. 

The measure $\nu$, on the other hand, is continuous with respect to time, and plays a key role when identifying equivalence classes. Moreover, it enables us to determine how much energy has dissipated from the system up to a certain time, since 
\begin{equation}
 \int_\Real u^2(t,x) dx+\nu(t,\Real)
\end{equation}
is independent of time. 

For conservative solutions no energy vanishes from the system, and therefore it is natural to impose that $\mu=\nu$. In the case of dissipative solutions all the energy that concentrates at isolated points where wave breaking takes place, vanishes from the system. The measure $\mu$, which corresponds to the energy, is purely absolutely continuous, while $\nu-\mu$ describes how much energy we already lost. If $\alpha\in (0,1]$, we can initially choose  the two measures to be equal, $\nu_0=\mu_0$, but as soon as wave breaking takes place, they will differ. In particular, $\nu$ does not coincide with the measure $\mu_{\rm cons}$ for conservative solutions. 

\end{remark}

\begin{definition}[Relabeling functions]\label{def:Grelab}
  We denote by $G$ the subgroup of the group of homeomorphisms from $\Real$ to $\Real$ such that 
  \begin{subequations}
    \label{eq:Gcond}
    \begin{align}
      \label{eq:Gcond1}
      f-\id \text{ and } f^{-1}-\id &\text{ both belong to } W^{1,\infty}(\Real), \\
      \label{eq:Gcond2}
      f_\xi-1 &\text{ belongs to } L^2(\Real),
    \end{align}
  \end{subequations}
  where $\id$ denotes the identity function. Given $\kappa>0$, we denote by $G_\kappa$ the subset $G_\kappa$ of $G$ defined by 
  \begin{equation}
    G_\kappa=\{ f\in G\mid  \norm{f-\id}_{W^{1,\infty}}+\norm{f^{-1}-\id}_{W^{1,\infty}}\leq\kappa\}. 
  \end{equation}
\end{definition}

\begin{definition}[Lagrangian coordinates]
  The subsets $\F$ and $\F_\kappa$ of $\G$ are defined as
  \begin{equation*}
    \mathcal{F}_\kappa=\{\Theta=(y,U,y_\xi, U_\xi, \bar h,h,r)\in\mathcal{G}\mid  y+H\in G_\kappa\},
  \end{equation*}
  and
  \begin{equation*}
    \mathcal{F}=\{\Theta=(y,U,y_\xi, U_\xi,\bar h,h,r)\in\mathcal{G}\mid  y+H\in G\},
  \end{equation*}
  where $H(t,\xi)$ is defined by 
  \begin{equation*}
    H(t,\xi)=\int_{-\infty}^\xi h(t,\tilde\xi)d\tilde\xi.
  \end{equation*}
\end{definition}

In addition, it should be pointed out that the condition on $y+H$ is closely linked to
$\norm{\frac{1}{y_\xi+h}}_{L^\infty}$ as the following lemma shows.

\begin{lemma}[{\cite[Lemma 3.2]{HolRay:07}}] 
  \label{lem:charH}
  Let $\kappa\geq0$. If $f$ belongs to $G_\kappa$,
  then $1/(1+\kappa)\leq f_\xi\leq 1+\kappa$ almost
  everywhere. Conversely, if $f$ is absolutely
  continuous, $f-\id\in W^{1,\infty}(\Real)$, $f$ satisfies
  \eqref{eq:Gcond2} and there exists $d\geq 1$ such
  that $1/d\leq f_\xi\leq d$ almost everywhere, then
  $f\in G_\kappa$ for some $\kappa$ depending only
  on $d$ and $\norm{f-\id}_{W^{1,\infty}}$.
\end{lemma}

An immediate consequence of \eqref{eq:lagcoord5} is therefore the following result.

\begin{lemma}\label{lem:Fpres}
  The space $\mathcal{G}$ is preserved by the governing equations \eqref{eq:sysdiss}.
\end{lemma}

For the sake of simplicity, for any $\Theta=(y, U,y_\xi, U_\xi, \bar h,h,r)\in \mathcal{F}$ and any function
$f\in\Gr$, we denote $(y\circ f, U\circ f, y_\xi\circ ff_\xi, U_\xi\circ ff_\xi,\bar h\circ ff_\xi, h\circ ff_\xi, r\circ f f_\xi)$ by $\Theta\circ f$.

\begin{proposition} \label{prop_action} The map 
  from $\Gr\times\F$ to $\F$ given by $(f,\Theta)\mapsto
  \Theta\circ f$ defines an action of the group $\Gr$ on
  $\F$.
\end{proposition}

Since $\Gr$ is acting on $\F$, we can consider the
quotient space $\quot$ of $\F$ with respect to the
action of the group $G$. The equivalence relation
on $\F$ is defined as follows: For any
$\Theta,\Theta'\in\F$, we say that $\Theta$ and $\Theta'$ are equivalent if there
exists a relabeling function $f\in\Gr$ such that $\Theta'=\Theta\circ f$. We
denote by $\Pi(\Theta)=[\Theta]$ the projection of $\F$ into the
quotient space $\quot$, and introduce the mapping
$\Lambda\colon\F\rightarrow\F_0$ given by
\begin{equation*}
  \Lambda(\Theta)=\Theta\circ (y+H)\inv
\end{equation*}
for any $\Theta=(y,U,y_\xi, U_\xi,\bar h,h,r)\in\F$. We have $\Lambda(\Theta)=\Theta$ when
$\Theta\in\F_0$. It is not hard to prove that $\Lambda$ is
invariant under the $\Gr$ action, that is,
$\Lambda(\Theta\circ f)=\Lambda(\Theta)$ for any $\Theta\in\F$ and
$f\in\Gr$. Hence, there corresponds to $\Lambda$ a
mapping $\tilde\Lambda$ from the quotient space $\quot$ to
$\F_0$ given by $\tilde\Lambda([\Theta])=\Lambda(\Theta)$ where
$[\Theta]\in\quot$ denotes the equivalence class of
$\Theta\in\F$. For any $\Theta\in\F_0$, we have
$\tilde\Lambda\circ\Pi(\Theta)=\Lambda(\Theta)=\Theta$. Hence, $\tilde\Lambda\circ
\Pi|_{\F_0}=\id|_{\F_0}$. 

Denote by $S\colon\F\times [0,\infty)\rightarrow \F$
the semigroup which to any initial
data $\Theta_0\in \F$ associates the solution $\Theta(t)$
of the system of differential equations
\eqref{eq:sysdiss} at time $t$. As indicated earlier,
the two-component Camassa--Holm system is invariant with
respect to relabeling.  More precisely, using our
terminology, we have the following result.

\begin{theorem} 
  \label{th:sgS} 
  For any $t>0$, the mapping $S_t\colon\F\rightarrow\F$
  is $\Gr$-equivariant, that is,
  \begin{equation}
    \label{eq:Hequivar}
    S_t(\Theta\circ f)=S_t(\Theta)\circ f
  \end{equation}
  for any $\Theta\in\F$ and $f\in\Gr$. Hence, the mapping
  $\tilde S_t$ from $\quot$ to $\quot$ given by
  \begin{equation*}
    \tilde S_t([\Theta])=[S_t\Theta]
  \end{equation*}
  is well-defined and generates a
  semigroup.
\end{theorem}

We have the
following diagram:
\begin{equation}
  \label{eq:diag}
  \xymatrix{
    \F_0\ar[r]^{\Pi}&\quot\\
    \F_\alpha\ar[u]^{\Lambda}&\\
    \F_0\ar[u]^{S_t}\ar[r]^\Pi&\quot\ar[uu]_{\tilde S_t}
  }
\end{equation}
Next we describe the correspondence between Eulerian coordinates (functions in $\D$) and Lagrangian
coordinates (functions in $\quot$). In order to do so, we have to take into account the fact that
the set $\D$ allows the energy density to have a singular part and a positive amount of energy can
concentrate on a set of Lebesgue measure zero.

We first define the mapping $L$ from $\D$ to $\F$ which to any initial data in $\D$ associates an
initial data for the equivalent system in $\F$.

\begin{definition} 
  \label{th:Ldef}
  For any $(u,\rho,\mu,\nu)$ in $\D$, let
  \begin{subequations}
    \label{eq:Ldef}
    \begin{align}
      \label{eq:Ldef1}
      y(\xi)&=\sup\left\{y\mid \nu((-\infty,y))+y<\xi\right\},\\
      \label{eq:Ldef2}
      h(\xi)&=1-y_\xi(\xi),\\
      \label{eq:Ldef3}
      U(\xi)&=u\circ{y(\xi)},\\
      \label{eq:Ldef4}
      r(\xi)&=\rho\circ{y(\xi)}y_\xi(\xi),\\
      \label{eq:Ldef5}
      \bar h(\xi)&=f\circ{y(\xi)} h(\xi),
    \end{align}
  \end{subequations}
where $f$ is given through \eqref{eq:radon}.
  Then $(y,U,y_\xi, U_\xi, \bar h,h,r)\in\F$. We denote by $L\colon \D\rightarrow \F$ the mapping which to any
  element $(u,\rho,\mu,\nu)\in\D$ associates $\Theta=(y,U,y_\xi, U_\xi,\bar h,h,r)\in \F$ given by \eqref{eq:Ldef}.
\end{definition}
\begin{proof}[Well-posedness of Definition \ref{th:Ldef}]
  We have to prove that $(y,U,y_\xi, U_\xi,\bar h,h,r)\in\F$. The proof follows the same lines as in
  \cite[Theorem 3.8]{HolRay:07}. The properties \eqref{eq:lagcoord1} to \eqref{eq:lagcoord5}
  are proved in the same way and we do not reproduce the proofs here. It remains to prove
  \eqref{eq:lagcoord6} and \eqref{eq:lagcoord7}. Since $f\leq 1$, see \eqref{eq:radon}, we have
  that $\bar h\leq h$ follows from \eqref{eq:Ldef5}. Let us prove \eqref{eq:lagcoord6}. First,
  we show that
  \begin{equation}
    \label{eq:pushforwards}
    \nu = y_\#(h(\xi)\,d\xi)\text{ and }\mu = y_\#(\bar h(\xi)\,d\xi).
  \end{equation}
  For any given $x\in\Real$, let us define $\xi$ as
  \begin{equation*}
    \xi = \sup\{\bar\xi\mid y(\bar\xi)=x\}.
  \end{equation*}
  We know that $y$ is  increasing and Lipschitz (we refer to \cite{HolRay:07}) so that $y$ is
  continuous. Hence, $y(\xi)=x$.  
  Moreover, by \eqref{eq:Ldef1} and \eqref{eq:Ldef2}, the definition of $y$ and $h$, we have for $\xi\in\Real$, that 
  \begin{equation}
  y(\xi)=\sup\left\{y\mid \nu((-\infty,y))+y<\xi\right\},
\end{equation}
and 
\begin{equation}
y(\xi)+\int_{-\infty}^\xi h(\bar \xi)d\bar \xi=\xi.
\end{equation}
Thus we have for any $\bar y> y(\xi)$
\begin{equation}
y(\xi)+\int_{-\infty}^\xi h(\bar\xi)d\bar\xi\leq \nu((-\infty,\bar y))+\bar y. 
\end{equation}
Letting $\bar y$ tend to $y(\xi)=x$, then yields
 \begin{equation}
    \label{eq:upboundnu}
     \int_{-\infty}^{\xi} h(\bar \xi)\,d\bar\xi\leq\nu((-\infty,x]) .
  \end{equation}
  For any $\epsi>0$, by the definition of $\xi$, we have that $y(\xi+\epsi)>y(\xi)$. Hence,
  following the same lines as before, we get
  \begin{equation*}
  y(\xi)+\nu((-\infty,y(\xi)])\leq y(\xi+\epsi)+\int_{-\infty}^{\xi+\epsi} h(\bar\xi)d\bar\xi
  \end{equation*}
   which, after letting $\epsi$ tend to zero, yields
  \begin{equation}  \label{eq:loboundnu}
  \nu((-\infty, x])\leq \int_{-\infty}^\xi h(\bar\xi)d\bar\xi.
  \end{equation}
  Combining \eqref{eq:upboundnu}, \eqref{eq:loboundnu} and the definition of $\xi$, we get 
  \begin{equation*}
    \nu((-\infty,x]) = \int_{y^{-1}((-\infty, x])} h(\bar \xi)\,d\bar \xi,
  \end{equation*}
  which proves the first identity in \eqref{eq:pushforwards}. Let us prove the second one. For
  any Borel set $A$, we have
  \begin{equation*}
    \mu(A)=\int_{A}f\,d\nu=\int_{y^{-1}(A)}f(y(\xi))\, h(\xi)\,d\xi
  \end{equation*}
  because $\nu = y_\#(h(\xi)\,d\xi)$. Then, using \eqref{eq:Ldef5}, we get
  $\mu(A)=\int_{y^{-1}(A)}\bar h(\xi)\,d\xi$, which concludes the proof of
  \eqref{eq:pushforwards}. We introduce the sets
  \begin{equation*}
    B=\{x\in\Real\mid \lim_{\delta\to 0}\frac{1}{2\delta}\mu(x-\delta, x+\delta) = (u_x^2 + \rho^2)(x)\}
  \end{equation*}
  and
  \begin{equation*}
    A = \{\xi\in y^{-1}(B)\mid y_\xi(\xi)>0\}.
  \end{equation*}
  From Besicovitch's derivation theorem \cite{Ambrosio}, we have $\meas(B^c)=0$. For almost every $\xi\in A$,
  we denote $x=y(\xi)$ and define $\xi_-^\delta$ and $\xi_+^\delta$ as
  \begin{equation}
    \label{eq:defxirho}
    \xi_{-}^\delta = \sup\{\bar\xi\mid y(\bar\xi)=x - \delta\}\text{ and }     \xi_{+}^\delta = \inf\{\bar\xi\mid y(\bar\xi)=x + \delta\},
  \end{equation}
  for any $\delta>0$. The continuity of $y$ implies $y(\xi_-^\delta) = x - \delta$ and $
  y(\xi_+^\delta) = x + \delta$.  From \eqref{eq:pushforwards}, we obtain
  \begin{align*}
    \mu(x-\delta, x+\delta) = \int_{\xi_-^\delta}^{\xi_+^\delta}\bar h(\bar\xi)\,d\bar\xi,
  \end{align*}
  as the definition \eqref{eq:defxirho} implies
  $y^{-1}((x-\delta,x+\delta))=(\xi_-^\delta,\xi_+^\delta)$. Since $y_\xi(\xi)>0$, we have
  $\xi_-^\delta<\xi_+^\delta$, $\lim_{\delta\to0}\xi_+^\delta=\lim_{\delta\to0}\xi_+^\delta=\xi$ and
  \begin{equation*}
    \frac{1}{2\delta}\mu(x-\delta, x+\delta) = \frac{\int_{\xi_-^\delta}^{\xi_+^\delta}\bar h(\bar\xi)\,d\bar\xi}{\xi_+^{\delta}-\xi_-^{\delta}}\, \frac{\xi_+^{\delta}-\xi_-^{\delta}}{\int_{\xi_-^{\delta}}^{\xi_+^{\delta}}y_\xi(\bar \xi)\,d\bar\xi}.
  \end{equation*}
  Letting $\delta$ tend to zero, we get
  \begin{equation*}
    u_x^2(y(\xi)) + \rho^2(y(\xi)) = \frac{\bar h(\xi)}{y_\xi(\xi)}.
  \end{equation*}
  As $U_{\xi} = u_x\circ y\, y_\xi$ and $r=\rho\circ y\,y_\xi$ almost everywhere, we obtain that 
  \begin{equation}
    \label{eq:lagcoord6bis}
    y_\xi\bar h = U_\xi^2+r^2,
  \end{equation}
  for almost every $\xi\in A$. However, as $\meas(B^c)=0$, we can prove that
  $\meas(\{\xi\in\Real\mid y_\xi(\xi)>0\text{ and }y(\xi)\in B^c\})=0$, see \cite[Lemma
  3.9]{HolRay:07}, and therefore \eqref{eq:lagcoord6bis} holds also for almost every
  $\xi\in\Real$ such that $y_\xi(\xi)>0$. 
  
It is left to show that \eqref{eq:lagcoord6bis} is also true for almost all $\xi$ such that $y_\xi(\xi)=0$. Following closely the proof of \cite[Theorem 3.8]{HolRay:07}, one obtains that the function
\begin{equation}
\xi\mapsto\int_{-\infty}^{y(\xi)} \big(u_x^2+\rho^2\big)dx
\end{equation}
is Lipschitz continuous with Lipschitz constant at most one. Thus we have, for all $\xi$, $\tilde \xi\in\Real$, using the Cauchy--Schwarz inequality,
\begin{align}\label{eq:extra}
\vert U(\tilde \xi)-U(\xi)\vert&=\vert \int_{y(\xi)}^{y(\tilde\xi)} u_x dx\vert \\ \nn
& \leq \sqrt{\vert y(\tilde \xi)-y(\xi)\vert}\sqrt{\vert\int_{y(\xi)}^{y(\tilde\xi)} u_x^2dx\vert}\\ \nn
& \leq\sqrt{\vert y(\tilde\xi)-y(\xi)\vert}\sqrt{\vert\int_{y(\xi)}^{y(\tilde\xi)}u_x^2+\rho^2dx\vert} \\ \nn
& \leq \vert \xi-\tilde\xi\vert,
\end{align}
because $y$ and $\xi\mapsto \int_{-\infty}^{y(\xi)} \big(u_x^2+\rho^2\big)dx$ are Lipschitz with Lipschitz constant at  most one. Hence, $U$ is Lipschitz and therefore differentiable almost everywhere. Let 
\begin{equation}
B_2=\{x\in B \mid \lim_{\delta\to 0} \frac{1}{\delta}\int_{x-\delta}^{x+\delta} u_x(s)ds=u_x(x)\}.
\end{equation}
From Besicovitch's derivation theorem we have that $\{\xi\mid y_\xi(\xi)=0\}\subset B_2^c$ and $\meas(B_2^c)=0$. Then \eqref{eq:extra} implies
\begin{equation}
\vert \frac{U(\tilde\xi)-U(\xi)}{\tilde\xi-\xi}\vert\leq\sqrt{\frac{y(\tilde\xi)-y(\xi)}{\tilde\xi-\xi}},
\end{equation}
due to the Lipschitz continuity with Lipschitz constant of at most one of $y$ and $\xi\mapsto\int_{-\infty}^{y(\xi)} \big(u_x^2+\rho^2\big)dx$. Hence, for almost every $\xi$ in $y^{-1}(B_2^c)$, we have 
\begin{equation}
\vert U_\xi(\xi)\vert \leq\sqrt{y_\xi(\xi)}.
\end{equation}
A similar argument yields that 
\begin{equation}
\vert r(\xi)\vert\leq \sqrt{y_\xi(\xi)}.
\end{equation}
Since $\meas(B_2^c)=0$, we have by \cite[Lemma 3.9]{HolRay:07}, that $y_\xi=0$ almost everywhere on $y^{-1}(B_2^c)$. Hence $U_\xi=0$ and $r=0$ almost everywhere on $y^{-1}(B_2^c)$. Thus $y_\xi \bar h=U_\xi^2+r^2$ almost everywhere on $y^{-1}(B_2^c)$, which is \eqref{eq:lagcoord6bis}. This finishes the proof of \eqref{eq:lagcoord7}.
\end{proof}

In fact, $L$ is a mapping from $\D$ to the set $\F_0\subset \F$, which contains exactly
one element of each equivalence class. 

 On the other hand, to any element in $\F$ there corresponds a unique element in $\D$ which is
 given by the mapping $M$ defined below.

\begin{definition}
  \label{th:umudef} 
  Given any element $\Theta=(y,U,y_\xi, U_\xi,\bar h,h,r)\in\F$. Then, the
  measure $y_{\#}( r(\xi)\,d\xi)$ is absolutely
  continuous, and we define $(u,\rho,\mu,\nu)$ as follows
  \begin{subequations}
    \label{eq:umudef}
    \begin{align}
      \label{eq:umudef1}
      u(x)&=U(\xi)\text{ for any }\xi\text{ such that  }  x=y(\xi),\\
      \label{eq:umudef2}
      \mu&=y_\#(\bar h(\xi)\,d\xi),\\
      \label{eq:umudef5}
      \nu&=y_\#(h(\xi)\,d\xi),\\
      \label{eq:umudef3}
      \rho(x)\,dx&=y_\#(r(\xi)\,d\xi).
    \end{align}
  \end{subequations}
  We have that $(u,\rho,\mu,\nu)$ belongs to $\D$. We
  denote by $M\colon \F\rightarrow\D$ the mapping
  which to any $\Theta$ in $\F$ associates the element $(u, \rho,\mu,\nu)\in \D$ as given
  by \eqref{eq:umudef}. In particular, the mapping $M$ is
  invariant under relabeling.
\end{definition}

Finally, we identify the connection between the equivalence classes in Lagrangian coordinates
and the set of Eulerian coordinates. The proof is similar to the one found in \cite{HolRay:07},  and
we do not reproduce it here.

\begin{theorem}\label{th:LMinv}The mappings $M$ and
  $L$ are invertible. We have
  \begin{equation*}
    L\circ M=\id_\quot\text{ and }M\circ L=\id_\D.
  \end{equation*}
\end{theorem}

\section{Semigroup of solutions}
\label{sec:semigroup}
In the last section we defined the connection between Eulerian and Lagrangian coordinates, which is the main tool when defining weak solutions of the 2CH system. The aim of this section is to show that we obtained a semigroup of solutions. Accordingly we define $T_t$ as
\begin{equation}
  T_t=M\circ S_t\circ L. 
\end{equation}

\begin{definition}
  \label{eq:defweakconssol}
  Assume that $u\colon[0,\infty)\times\Real \to \Real$  and $\rho\colon [0,\infty) \times\Real \to\Real$ satisfy \\
  (i) $u\in L^\infty([0,\infty), H^1(\Real))$ and $\rho\in L^\infty([0,\infty), L^2(\Real))$, \\
  (ii) the equations
  \begin{multline}\label{eq:weak1}
    \iint_{[0,\infty)\times\Real}\Big[
    -u(t,x)\phi_t(t,x)
    +\big(u(t,x)u_x(t,x)+P_x(t,x)\big)\phi(t,x)\Big]dxdt\\
    =\int_\Real u(0,x)\phi(0,x)dx,
  \end{multline}
  \begin{multline}\label{eq:weak2}
    \iint_{[0,\infty)\times\Real} \Big[(P(t,x)-u^2(t,x)-\frac{1}{2}u_x^2(t,x)-\frac{1}{2}\rho^2(t,x))\phi(t,x)\\
    +P_x(t,x)\phi_x(t,x)\Big]dxdt=0,
  \end{multline}
  and
  \begin{equation}
    \label{eq:weak3}
    \iint_{[0,\infty)\times\Real}\Big[ -\rho(t,x)\phi_t(t,x)-u(t,x)\rho(t,x)\phi_x(t,x)\Big]dxdt=\int_\Real \rho(0,x)\phi(0,x)dx,
  \end{equation}
  hold for all $\phi\in
  C^\infty_0([0,\infty)\times\Real)$. Then we say that
  $(u,\rho)$ is a weak global solution of the two-component
  Camassa--Holm system. 
\end{definition}

\begin{theorem} \label{th:energi}
  The mapping $T_t$ is a semigroup of solutions of the 2CH system. Given some initial data $(u_0,\rho_0,\mu_0,\nu_0)\in \D$, let 
$(u(t,\dott),\rho(t,\dott),\mu(t,\dott),\nu(t,\dott))=T_t(u_0,\rho_0,\mu_0,\nu_0)$. Then $(u,\rho)$ is a weak solution to \eqref{eq:chsys2} and $(u,\rho,\mu)$ is a weak solution to  
  \begin{equation}\label{eq:weak5}
    (u^2+\mu)_t+(u(u^2+\mu))_x\leq(u^3-2Pu)_x.
  \end{equation}
  The function
  \begin{equation}
    F(t)=\int_\Real d(\nu(t,x)-\mu(t,x))-\int_\Real d(\nu(0,x)-\mu(0,x)),
  \end{equation}
  which is an increasing, semi-continuous function, equals the amount of energy that has vanished
  from the solution up to time $t$.
\end{theorem}

\begin{proof}
  This proof follows essentially the same lines as the one in \cite{GHR5} and of the proof of
  \eqref{eq:weak5}, which we present here. Let $\phi\in C^\infty_0((0,\infty)\times\Real)$ such that
  $\phi(t,x)\geq 0$. Since $U(t,\xi)$ is continuous with respect to time, we have
\begin{align}\label{eq:iint1}
 &\iint_{\Real_+\times\Real}u^2\phi_t(t,x)dx dt \\ \nn 
 & \qquad =-\iint_{\Real_+\times\Real} (U^2(t,\xi)U_\xi(t,\xi)-2U(t,\xi)Q(t,\xi)y_\xi(t,\xi))\phi(t,y(t,\xi))d\xi dt\\ \nn 
    & \qquad \quad-\iint_{\Real_+\times\Real} U^3(t,\xi)\phi_\xi(t,y(t,\xi))d\xi dt 
\end{align}

The measure $\mu$ in Eulerian coordinates corresponds to the function $\bar h(t,\xi)$, which is discontinuous with respect to time, in Lagrangian coordinates. Thus it is important for the following calculations to keep in mind that $\bar h(t,\xi)$ can be rewritten as 
$h(t,\xi)=\bar h(t,\xi)+\sum_{j=0}^\infty \chi_{\{\tau_j(\xi)\leq t\}}(\xi)l_j(\xi)$, where $h(t,\xi)$ is continuous with respect to time and corresponds to the measure $\nu$ in Eulerian coordinates by Definition~\ref{th:umudef} and $\tau_0(\xi)=0$. Thus
  \begin{align}\label{eq:iint2}
    & \iint_{\Real_+\times\Real} \phi_t(t,x)d\mu(t,x)dt= \iint_{\Real_+\times\Real}\phi_t(t,y(t,\xi))\bar h(t,\xi)d\xi dt\\ \nn 
    & \qquad =\iint_{\Real_+\times\Real} [(\phi(t,y(t,\xi)))_t-\phi_x(t,y(t,\xi))y_t(t,\xi)]\bar h(t,\xi)d\xi dt\\ \nn 
    & \qquad =\iint_{\Real_+\times \Real}(\phi(t,y(t,\xi)))_t\bar h(t,\xi)dt d\xi-\iint_{\Real_+\times\Real}U(t,\xi)\bar h(t,\xi)\phi_x(t,y(t,\xi))d\xi dt\\ \nn
    & \qquad =\iint_{\Real_+\times\Real}(\phi(t,y(t,\xi)))_th(t,\xi)dt d\xi -\iint_{\Real_+\times\Real}(\phi(t,y(t,\xi)))_t\sum_{j=0}^\infty \chi_{\{\tau_j(\xi)\leq t\}}(\xi) l_j(\xi)dt d\xi\\ \nn 
    & \qquad\quad -\iint_{\Real_+\times\Real}U(t,\xi)\bar h(t,\xi)\phi_x(t,y(t,\xi))d\xi dt\\ \nn 
    & \qquad =-\iint_{\Real_+\times\Real} 2(U^2(t,\xi)-P(t,\xi))U_\xi(t,\xi)\phi(t,y(t,\xi))d\xi dt\\ \nn 
    & \qquad \quad-\iint_{\Real_+\times\Real}U(t,\xi)\bar h(t,\xi)\phi_x(t,y(t,\xi))d\xi dt\\ \nn 
    & \qquad \quad +\int_{\Real}\sum_{j=1}^\infty \phi(\tau_j(\xi),y(\tau_j(\xi),\xi))l_j(\xi)d\xi.
  \end{align}
Note that the second integral in the fourth line is well defined since, by construction, $0\leq\sum_{j=0}^\infty \chi_{\{\tau_j(\xi)\leq t\}}(\xi)l_j(\xi)\leq h(t,\xi)$ and therefore the integrand belongs to $L^1(\Real)$ for each fixed $t$. In addition, $\phi\in C^\infty_0((0,\infty)\times\Real)$ and hence this integral exists. Similar conclusions hold for the integral with respect to $\xi$ in the last line.

Observe that we have 
\begin{align}\label{eq:iint3}
 \iint_{\Real_+\times\Real}  (2PU_\xi+&2UQy_\xi -3U^2U_\xi) (t,\xi)\phi(t,y(t,\xi))d\xi dt \\ \nn
& =-\iint_{\Real_+\times\Real} (2PU-U^3)(t,\xi)\phi_\xi(t,y(t,\xi))d\xi dt.
\end{align}
  Gathering \eqref{eq:iint1}, \eqref{eq:iint2} and applying \eqref{eq:iint3} yields
\begin{align*}
  \iint_{\Real_+\times\Real}&u^2\phi_t(t,x)dx dt +\iint_{\Real_+\times\Real} \phi_t(t,x)d\mu(t,x)dt\\ \nn 
 & =-\iint_{\Real_+\times\Real} 2P(t,x)u(t,x)\phi_x(t,x)dtdx -\iint_{\Real_+\times\Real} u(t,x)\phi_x(t,x) d\mu(t,x) dt\\ \nn 
    & \quad +\int_{\Real}\sum_{j=0}^\infty \phi(\tau_j(\xi),y(\tau_j(\xi),\xi))l_j(\xi)d\xi.
\end{align*}
  The integral in the last line is finite and positive, hence we proved \eqref{eq:weak5}.
\end{proof}

We have now shown that this new solution concept yields global weak solutions of the 2CH system. However, there is one more question which is of great interest.   Recall that $(u,\rho)$ satisfies the same equation, namely \eqref{eq:chsys2},  independently of the value $\alpha$, yet we have carefully constructed a solution for a given 
$\alpha$. One can turn this around and ask: Given a solution $(u,\rho)$, can we determine $\alpha$?
The answer is contained in the following theorem. 

\begin{theorem} \label{thm:find_alpha}
Let $(u,\rho,\mu,\nu)$ be a weak solution of the 2CH system.  The limits from the \emph{future} and the \emph{past} of the measure $\mu$ exist for all times and we denote them as follows
  \begin{equation*}
    \mu^-(t)=\lim_{t'\uparrow t} \mu(t')\quad\text{ and }\quad  \mu^+(t)=\lim_{t'\downarrow t} \mu(t').
  \end{equation*}
  We have that the measure $\mu$ is continuous backward in time, that is,
  \begin{equation*}
    \mu^+ = \mu
  \end{equation*}
  for all $t$.  In the other direction, going forward in time, we have that
  \begin{equation}
    \label{eq:decompmu}
    \mu = \muac^- + (1-\alpha)\mus^-,
  \end{equation} 
  for all $t$, that is,
  \begin{equation*}
    \muac = \muac^-\quad\text{ and }\quad\mus = (1-\alpha)\mus^-.
  \end{equation*}
  Moreover, we have that, for almost every time $t$,
  \begin{equation}
    \label{eq:consenergal}
    \mu^+(t) = \mu^-(t) = \mu(t) = \muac(t).
  \end{equation}
\end{theorem}

\begin{proof} We prove the theorem first for $\alpha<1$. Given $t\in\Real_+$, we define
  \begin{equation}
    \label{eq:defbarhminus}
 {\tilde{\bar h}}(t,\xi)  =
    \begin{cases}
      \frac{1}{1-\alpha}\bar h(t,\xi),&\text{ if }y_{\xi}(t,\xi)=0,\\
      \bar h (t,\xi),&\text{ otherwise.}
    \end{cases}
  \end{equation}
  We claim that, for almost every $\xi$,
  \begin{align}
    \label{eq:barhlimits}
   \bar h(t-0,\xi) &=  \lim_{t'\uparrow t}\bar h(t',\xi) = {\tilde{\bar h}}(t,\xi) &\text{and}&& \lim_{t'\downarrow t}\bar h(t',\xi) &= \bar h(t,\xi).
  \end{align}
  Indeed, if $y_\xi(t,\xi)>0$ then $\tau_n(\xi)<t<\tau_{n+1}(\xi)$ and $\bar h(t',\xi)$ is differentiable in $t'$. It is therefore continuous and we have
  \begin{equation*}
    \lim_{t'\uparrow t}\bar h(t',\xi) = \lim_{t'\downarrow t}\bar h(t',\xi) = \bar h(t,\xi).
  \end{equation*}
  If $y_\xi(t,\xi)=0$, there exists $n$ such that $t=\tau_n(\xi)$. In $[\tau_{n-1}(\xi),\tau_{n}(\xi))$, the function $t'\mapsto \bar h(t',\xi)$ satisfies the ordinary differential equation \eqref{eq:sysdiss-5}. Hence, $\lim_{t'\uparrow t}\bar h(t',\xi)$ exists and, by the jump conditions \eqref{eq:defderatjump}, \eqref{eq:defderatjumpA}, we have 
  $\ {\tilde{\bar h}}(t-0,\xi) = \lim_{t'\uparrow t}\bar h(t',\xi)=\frac{1}{1-\alpha}\bar
  h(\tau_n(\xi),\xi)$. For $t'\in[\tau_n(\xi), \tau_{n+1}(\xi))$, $\bar h(t',\xi)$ also satisfies \eqref{eq:sysdiss-5} and we then directly have $\lim_{t'\downarrow t}\bar h(t',\xi)=\bar
  h(\tau_n(\xi),\xi)$. This concludes the proof of \eqref{eq:barhlimits}. Let us now define the measure $\mu^{-}$ as
  \begin{equation}
    \label{eq:defmuminus}
    \mu^{-}(t)=y_{\#}(\bar h(t-0)\,d\xi).
  \end{equation}
  We claim that
  \begin{equation}
    \label{eq:limsmu}
    \lim_{t'\uparrow t}\mu(t') = \mu^{-}(t)\quad\text{ and }\quad  \lim_{t'\downarrow t}\mu(t') = \mu(t).
  \end{equation}
  Here, we use the weak star topology for the measure. For any continuous function $\phi\in C(\Real)$ with compact support, we have
 \begin{equation}
    \label{eq:phidmu}
    \int_\Real \phi(x)\,d\mu(t',x)=\int_\Real\phi(y(t',\xi))\,\bar h(t',\xi)\,d\xi.
  \end{equation}
  For almost every given $\xi$, we have that $\lim_{t'\to t} y(t',\xi) = y(t,\xi)$, and, from
  \eqref{eq:barhlimits}, we have $\lim_{t'\uparrow t}\bar h(t',\xi)= {\tilde{\bar h}}(t,\xi)$. Hence, the
  integrand in \eqref{eq:phidmu} tends to $\phi(y(t,\xi))\, {\tilde{\bar h}}(t,\xi)$ when $t'$ converges to
  $t$ from below. Moreover, since $\norm{y(t,\dott) - \id}_{L^\infty}$ is bounded and $\phi$ has compact
  support, we can restrict the integration domain in \eqref{eq:phidmu} to a bounded domain. Then,
  the first proposition in \eqref{eq:limsmu} follows from the Lebesgue dominated convergence theorem
  applied to \eqref{eq:phidmu} by letting $t'$ tend to $t$. The second proposition is proved in a
  similar way. Let us define 
  \begin{equation}
    \label{eq:defsetB}
    B=\{\xi\in\Real\mid y_\xi(t,\xi)>0\}    
  \end{equation}
  and $A=y(t,B)$. Let us prove that
  \begin{equation}
    \label{eq:defacspart}  
    \muac^{-} (t)= \mu^{-}|_{A} (t)\quad\text{ and }\quad\mus^{-} (t)= \mu^{-}|_{A^c}(t).
  \end{equation}
  Here $\mu|_{A}(t)$ denotes the restriction of $\mu(t)$ to $A$, that is, $\mu|_{A}(t,E) =
  \mu(t,E\cap A)$ for any Borel set $E$. We have $\meas(A^c)=0$. Indeed, since $y$ is surjective,
  $A^c\subset y(t,B^c)$ and $\meas(y(t,B^c))=\int_{B^c}y_\xi(t,\xi)\,d\xi = 0$, from the definition of
  $B$. Let us prove that $\mu^-|_{A}(t)$ is absolutely continuous. We consider a set $E$ of
  zero measure. 
  We have
  \begin{align*}
    \mu^{-}|_{A}(t,E) &= \int_{y^{-1}(t,A\cap E)}\bar h(t,\xi)\,d\xi
  \end{align*}
  We define $K_M = \{\xi\in\Real\mid \frac{\bar h(t,\xi)}{y_\xi(t,\xi)} \leq M\}$. Let us prove that
  $y_\xi(t,\xi)>0$ for almost every $\xi\in y^{-1}(t,A)$. Assume the opposite, then, since $y$ is
  surjective, there exist $\bar\xi\in B^{c}$ and $\xi\in B$ such that $y(t,\xi)=y(t,\bar \xi)$. Since
  $y$ is increasing, it implies $y_\xi(t,\xi)=y_\xi(t,\bar\xi)=0$, which is a contradiction to the fact that
  $\xi\in B$. Thus, the indicator function of the set $K_M$, which we denote $\chi_{K_M}$, converges
  to one, almost everywhere in $y^{-1}(t,A\cap E)$, as $M$ tends to infinity. We have
  \begin{equation*}
    \int_{y^{-1}(t,A\cap E)}\chi_{K_M}(\xi)\bar h(t,\xi)\,d\xi \leq M\int_{y^{-1}(t,A\cap E)} y_\xi(t,\xi)\,d\xi=M\meas(A\cap E)=0,
  \end{equation*}
  and, by the monotone convergence theorem, it follows that $\int_{y^{-1}(t,A\cap E)}\bar
  h(t,\xi)\,d\xi=0$. Let us now prove \eqref{eq:decompmu}. We have, for any Borel set $E$,
  \begin{align}
    \label{eq:mutE}
    \mu(t,E)&=\int_{y^{-1}(t,E)}\bar h(t,\xi)\,d\xi\\
    &=\int_{y^{-1}(t,E)\cap B}\bar h(t,\xi)\,d\xi + \int_{y^{-1}(t,E)\cap B^c}\bar h(t,\xi)\,d\xi \notag\\
    &=\int_{y^{-1}(t,E\cap A)} {\tilde{\bar h}}(t,\xi)\,d\xi + \int_{y^{-1}(t,E\cap A^c)}(1-\alpha) {\tilde{\bar h}}(t,\xi)\,d\xi,\notag
  \end{align}
  by the definition \eqref{eq:defbarhminus} of $\bar h(t-0)$ and the fact that $y^{-1}(t,A)=B$. Then,
  \begin{align*}
    \mu(t,E)&=\int_{y^{-1}(t,E\cap A)} {\tilde{\bar h}}(t,\xi)\,d\xi + (1-\alpha)\int_{y^{-1}(t,E\cap A^c)}\ {\tilde{\bar h}}(t,\xi)\,d\xi\\
    &=\muac^{-}(t,E) + (1-\alpha)\mus^{-}(t,E),
  \end{align*}
  by \eqref{eq:defacspart} and \eqref{eq:defmuminus}. This concludes the proof of the theorem for $\alpha<1$. In the case where $\alpha=1$, the definition \eqref{eq:defbarhminus} cannot be
  used. However, the limit $\lim_{t'\uparrow t}\bar h(t',\xi)$ still exists, and we denote it by $ {\tilde{\bar h}}(t,\xi)$. The rest of the proof is the same up to \eqref{eq:mutE} which is replaced by
  \begin{align}
    \label{eq:mutE2}
    \mu(t,E)&=\int_{y^{-1}(t,E)}\bar h(t,\xi)\,d\xi\\
    &=\int_{y^{-1}(t,E)\cap B}\bar h(t,\xi)\,d\xi + \int_{y^{-1}(t,E)\cap B^c}\bar h(t,\xi)\,d\xi \notag\\
    &=\int_{y^{-1}(t,E\cap A)} {\tilde{\bar h}}(t,\xi)\,d\xi, \notag
  \end{align}
  because, in the fully dissipative case $\alpha=1$, we have $\bar h(t,\xi) = 0$ when
  $y_\xi(t,\xi)=0$. Then, we obtain that $\mu(t,E)=\muac^{-}(t,E)$. We turn to the proof of
  \eqref{eq:consenergal}. Let us introduce the set  
  \begin{multline*}
    \mathcal A = \{t\in\Real_+\ | \text{ for almost every $\xi$, either} \ y_\xi(t,\xi) > 0 \text{ or }\\
    (y_\xi(t,\xi) = 0 \text{ and }\bar h(t-0, \xi) = 0)\}.
  \end{multline*}
  For $t\in\mathcal{A}$,
  using \eqref{eq:defacspart}, we get
  \begin{align*}
    \mus^{-}(t)(\Real) &= \int_{y^{-1}(y(t,B)^c)} \bar h(t-0,\xi)\,d\xi  \\ &= \int_{{y^{-1}(y(t,B)^c)}\cap B} \bar h(t-0,\xi)\,d\xi \\ &\leq \int_{B^c\cap B} \bar h(t-0,\xi)\,d\xi = 0.
  \end{align*}
  Thus, \eqref{eq:consenergal} will be proved once we have proved that $\mathcal{A}$ has full
  measure. For a given $\xi\in\Real$, we know from Corollary \ref{cor:tau} that the collision times
  do not accumulate. For $\alpha<1$, it means that $y_\xi(t,\xi)=0$ only at isolated times $t$. For
  $\alpha=1$, assuming a collusion occurs at the point $\xi$, we have $y_\xi(t,\xi)>0$ for
  $t<\tau_1(\xi)$, $y_\xi(t,\xi)=0$ for $t\geq\tau_1(\xi)$ but $\bar h(t,\xi)=0$ for all
  $t\geq\tau_1(\xi)$. Hence, in both cases, we have
  \begin{equation*}
    \meas(\{t\in\Real_+\ |\ y_\xi(t,\xi) = 0 \text{ and } \bar h(t-0,\xi) > 0\}) = 0.
  \end{equation*}
  Using Fubini theorem, we get
  \begin{multline*}
    \int_{\Real^+}\meas(\{\xi \in \Real \ |\ y_\xi(t,\xi) = 0 \text{ and } \bar h(t-0,\xi) >
    0\})\,dt \\ = \int_\Real \meas(\{t\in\Real_+\ |\ y_\xi(t,\xi)=0\text{ and } \bar h(t-0,\xi) > 0\})\,d\xi =0,
  \end{multline*}
  so that $\meas(\{\xi \in \Real \ |\ y_\xi(t,\xi) = 0 \text{ and } \bar h(t-0,\xi) > 0\})=0$ for
  almost every time. It follows that $\mathcal{A}$ has full measure, which concludes the proof of
  \eqref{eq:consenergal}.
\end{proof}

\section{The peakon-antipeakon example} \label{sec:AP}

The most well-known explicit solution, and key example for the dichotomy between conservative and dissipative solutions, as well as  a source for intuition in the general case, is the peakon-antipeakon solution. We here present the detailed analysis in this paper applied to this example.  See, e.g.,  
\cite{BealsSattingerSzm:01,HolRay:06b,HolRay:07B,wahlen}.

\begin{figure}\centering
      \includegraphics[width=.45\textwidth]{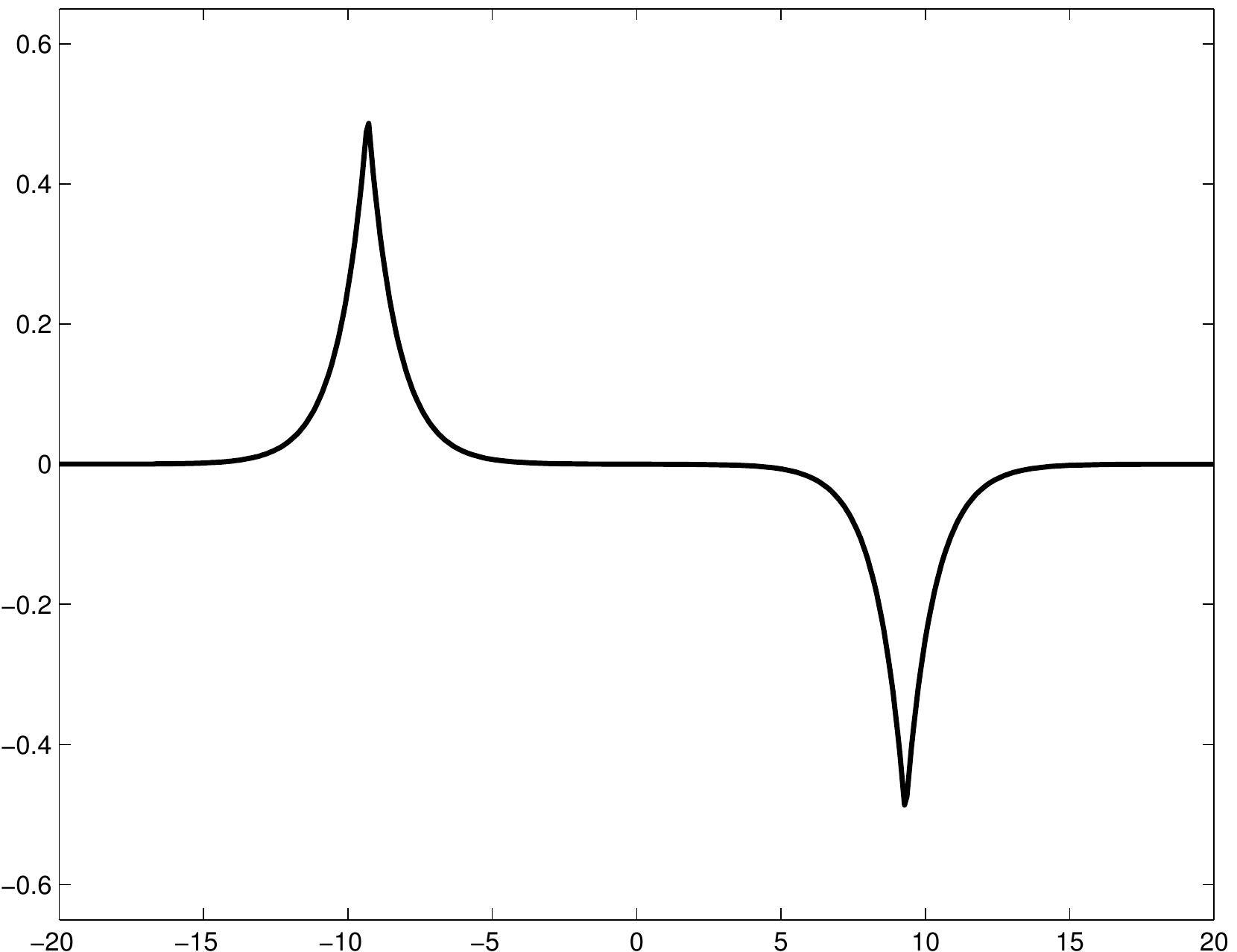}
       \includegraphics[width=.45\textwidth]{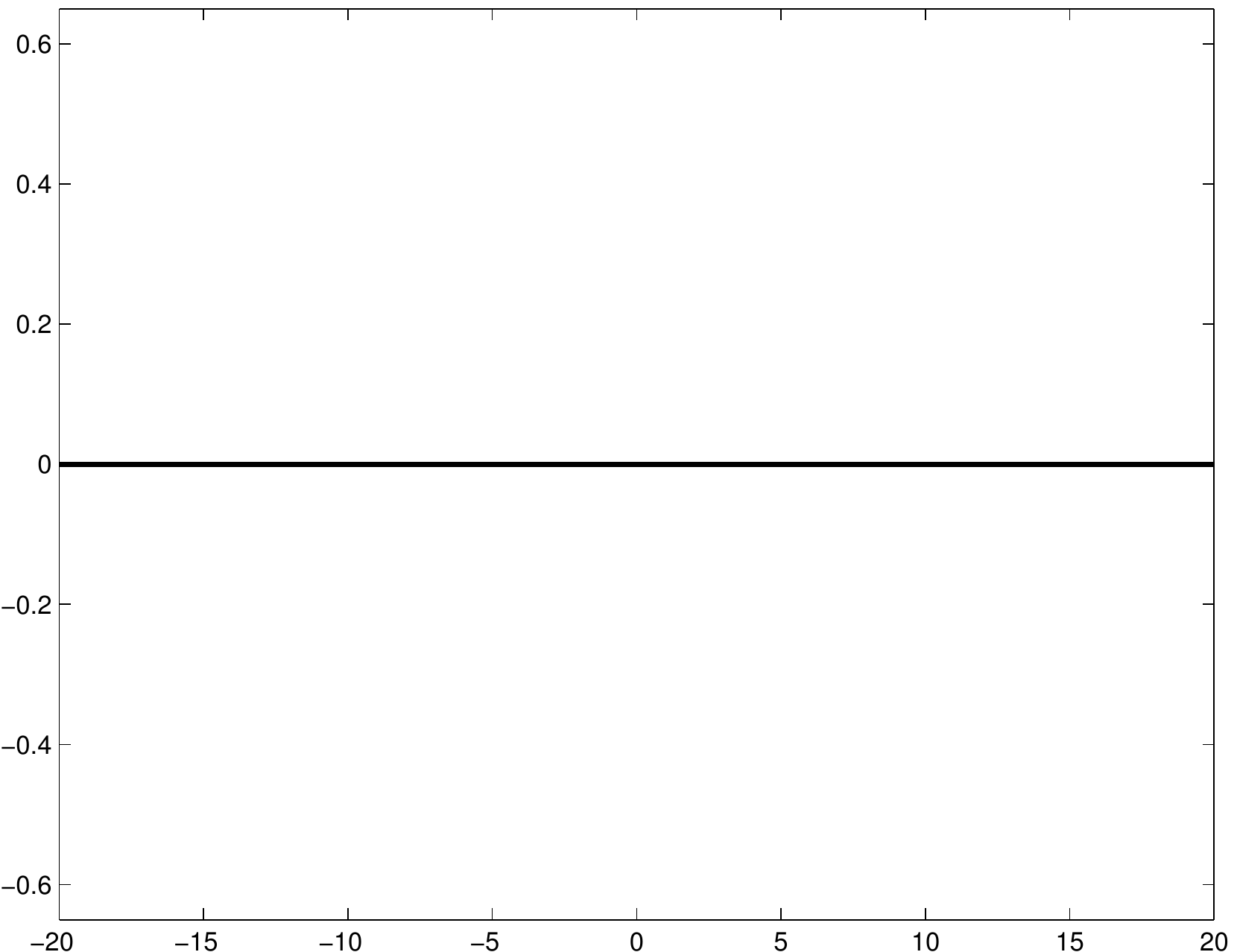}
      \caption{The peakon-antipeakon solution at times $0$ and $t_0$ for three different values of $\alpha$. (The curves coincide.)} \label{fig:PA}
    \end{figure}

Consider the initial data:
\begin{equation}
u(0,x)= \begin{cases} 
\sgn(x)A(0)e^{-\abs{x}}, & \text{for $\abs{x}>\gamma(0)$},\\
B(0)\sinh x, &
\text{for $\abs{x}\le\gamma(0)$}, 
\end{cases} 
\end{equation}
where we have introduced\footnote{Here $\sinh^{-1}$ denotes the multiplicative inverse of  $\sinh$. Similar conventions apply to $\cosh^{-1}$ and $\tanh^{-1}$.}
\begin{equation}
A(t)= \frac{E}{2}\sinh(\frac{E}{2}(t-t_0)), \quad 
B(t)=E\sinh^{-1}(\frac{E}{2}(t-t_0)), \quad 
\gamma(t)=\ln\cosh(\frac{E}{2}(t-t_0)). 
\end{equation}
Here $t_0>0$ is a given time where the wave breaking will occur.  For $t<t_0$ the function 
\begin{equation}
u(t,x)= \begin{cases} 
\sgn(x)A(t)e^{-\abs{x}}, & \text{for $\abs{x}>\gamma(t)$},\\
B(t)\sinh x, &
\text{for $\abs{x}\le\gamma(t)$}, 
\end{cases} 
\end{equation}
will be the peakon-antipeakon solution of the Camassa--Holm equation \eqref{eq:chsys2A} with $\kappa=0$ and $\rho$ identically zero, see 
\cite[ Thm.~4.1, Ex.~4.2 (ii)]{HolRay:06b}.  Define the two Radon measures by
\begin{equation}
\mu(t)=\nu(t)= u_x^2(t,x)\dx, \quad 
u_x(t,x)= \begin{cases} 
-A(t)e^{-\abs{x}}, & \text{for $\abs{x}>\gamma(t)$},\\
B(t)\cosh x, &
\text{for $\abs{x}\le\gamma(t)$}. 
\end{cases}
\end{equation} 
Observe that
\begin{equation}\label{eq:energyPA}
\int_\Real (u^2(t,x)+u_x^2(t,x)) \, dx= E^2 \text{  for all $t< t_0$.}
\end{equation} 
At $t=t_0$ we see that $A(t_0)=\gamma(t_0)=0$, and thus $u(t_0,x)=u_x(t_0,x)=0$ almost everywhere, while $\mu(t), \nu(t)\to E^2\delta_0$ as $t\uparrow t_0$. Let namely $M\subset \Real$ be a measurable set. Then
\begin{equation}
\mu(t)(M)=\int_M u_x^2(t,x) \, dx\underset{t\uparrow t_0}{\to} 
\begin{cases} 
E^2, &\text{for $0\in M$}, \\
0, & \text{for $0\not\in M$},
\end{cases}
\end{equation}
since $\gamma(t)\to 0$ and $u_x(t,x)\to 0$ ($x\neq0$) as $t\uparrow t_0$, and
\begin{equation}
\begin{aligned}
\int_{-\gamma(t)}^{\gamma(t)}u_x^2(t,x)\, dx&=B^2(t)\int_{-\gamma(t)}^{\gamma(t)}\cosh^2(x) \, dx \\
&= B^2(t)\big(\gamma(t)+\frac12\sinh(2\gamma(t)) \big)\underset{t\uparrow t_0}{\to} E^2. 
\end{aligned}
\end{equation}

Next we turn to the Lagrangian variables, which are solutions of the following system of ordinary differential equations (cf.~\eqref{eq:sysdiss}) for $t<t_0$, 
\begin{subequations}\label{eq:sysdiss1PA}
  \begin{align}
    y_t& =U , \\
     U_t&=-Q,\\
    y_{t,\xi}&=U_\xi,\\
    U_{t,\xi}&= \frac12 h+(U^2-P)y_\xi,\\
    h_t& = 2(U^2-P)U_\xi, \\
    \bar h_t&=h_t,
  \end{align}
\end{subequations}
where $P$ and $Q$ are given by  \eqref{eq:Plag1} and \eqref{eq:Qlag1}, respectively.  However, this system is difficult to solve directly, even in the case of a peakon-antipeakon solution. The initial data have to be judiciously chosen, and we will return to this shortly.  Instead of solving \eqref{eq:sysdiss1PA} directly, we will determine the solution by using the connection between Eulerian and Lagrangian variables directly. The key relations are
\begin{equation} \label{eq_lagrange_key}
y_t=u\circ y, \quad U=u\circ y, \quad h=u_x^2\circ y\, y_\xi.
\end{equation}
We have to determine the characteristics initially (here denoted by $\bar y_0$), given by   \eqref{eq:Ldef1}, that is,  
$ \bar y_0(\xi)=\sup\left\{y\mid \nu((-\infty,y))+y<\xi\right\}$ (we write  $\bar y_0$ rather than $y_0$ as we will modify it shortly). In this case, where the measure $\nu$ is absolutely continuous, we find that the characteristics is given by
\begin{equation}\label{eq:initial}
 \int_{-\infty}^{\bar y_0(\xi)} u_x^2(0,x)dx +\bar y_0(\xi)=\xi,
\end{equation} 
which appears to be difficult to solve, even in this case. 
Fortunately, its derivative is straightforward:
\begin{equation}
\begin{aligned}
\bar y_0'(\xi)&=\frac{1}{1+u_x^2(0,\bar y_0(\xi))}\\
&=\begin{cases}
\big(1+A^2(0)e^{-\sgn(\xi)2\bar y_0(\xi)}\big)^{-1}, & \text{for $\xi\not\in [\xi_-,\xi_+]$},\\
\big(1+B^2(0)\cosh^2(\bar y_0(\xi))\big)^{-1}, & \text{for $\xi\in [\xi_-,\xi_+]$},
\end{cases}
\end{aligned}
\end{equation} 
where we introduced $\xi_\pm$, the solution of $\bar y_0(\xi_\pm)=\pm\gamma(0)$.
For reasons that will become clear later, we will benefit from  having characteristics that satisfy $y_0(\pm\gamma(0))=\pm\gamma(0)$, which is not automatically satisfied by \eqref{eq:initial}. We use the freedom given to us by relabeling to modify $\bar y_0$. To that end define
\begin{equation}
f(z)=\int_{-\infty}^{z} u_x^2(0,x)dx +z.
\end{equation}
Then $f$ is a relabeling function in the sense of Definition \ref{def:Grelab}. Observe that with this definition $\xi_\pm=f(\pm\gamma(0))$ and $f'(z)=u_x^2(0,z) +1$. Introduce
\begin{equation}
 y_0(z)=\bar y_0(f(z)),
\end{equation}
which implies
\begin{equation}
 y_0(\pm\gamma(0))=\bar y_0(f(\pm\gamma(0)))=\bar y_0(\xi_\pm)=\pm\gamma(0).
\end{equation}
Hence
\begin{equation}
\begin{aligned}
y_0'(\xi)&=\bar y'_0\circ f(\xi) f'(\xi)  
=1.
\end{aligned}
\end{equation} 
Thus the relabeled initial characteristics is simply $y_0(\xi)=\xi$. Clearly, we could have chosen this function immediately, and the above argument shows that one can always use the identity as the initial characteristics when the initial data contains no singular part.  However, the above argument illustrates the possible use of relabeling. 

The Lagrangian variables are then given, using \eqref{eq_lagrange_key}  for $t<t_0$, by
\begin{equation} \label{eq:lagrangeAP}
\begin{aligned}
y(t,\xi)&= \begin{cases}
\xi+\sgn(\xi)\ln\big(1+(\cosh(\frac{E}{2}(t-t_0))-\cosh(\frac{E}{2}t_0)) e^{-\abs{\xi}} \big), 
& \text{for $\abs{\xi}\ge \gamma(0)$}, \\
2\artanh\Big(\tanh(\frac{\xi}{2}) \frac{\tanh^2(\frac{E}{4}(t-t_0))}{\tanh^2(-\frac{E}{4}t_0)}\Big), 
& \text{for $\abs{\xi}\le\gamma(0)$}, 
\end{cases} \\[3mm]
U(t,\xi)&=
 \begin{cases}
\begin{matrix}\sgn(\xi)A(t) e^{-\abs{\xi}}\\
\times\big(1+(\cosh(\frac{E}{2}(t-t_0))-\cosh(\frac{E}{2}t_0)) e^{-\abs{\xi}} \big)^{-1}\end{matrix}, & \text{for $\abs{\xi}\ge \gamma(0)$}, \\
\begin{matrix}2B(t) \tanh(\frac{\xi}{2}) \frac{\tanh^2(\frac{E}{4}(t-t_0))}{\tanh^2(-\frac{E}{4}t_0)}\\
\times\big(1-\tanh^2(\frac{\xi}{2}) \frac{\tanh^4(\frac{E}{4}(t-t_0))}{\tanh^4(-\frac{E}{4}t_0)} \big)^{-1}\end{matrix}, & \text{for $\abs{\xi}\le\gamma(0)$}, \end{cases}\\[3mm]
h(t,\xi)&=
\begin{cases}
\begin{matrix}A(t)^2 e^{-2\abs{\xi}}\\
\times\big(1+(\cosh(\frac{E}{2}(t-t_0))-\cosh(\frac{E}{2}t_0)) e^{-\abs{\xi}} \big)^{-3}\end{matrix}, & \text{for $\abs{\xi}\ge \gamma(0)$}, \\
\begin{matrix} B(t)^2\Big(1+\tanh^2(\frac{\xi}{2})\frac{\tanh^4(\frac{E}{4}(t-t_0))}{\tanh^4(-\frac{E}{4}t_0)}\Big)^2\\ 
\times\Big(1-\tanh^2(\frac{\xi}{2})\frac{\tanh^4(\frac{E}{4}(t-t_0))}{\tanh^4(-\frac{E}{4}t_0)}\Big)^{-3}\\
\times \cosh^{-2}(\frac{\xi}{2})\frac{\tanh^2(\frac{E}{4}(t-t_0))}{\tanh^2(-\frac{E}{4}t_0)}\end{matrix}, & \text{for $\abs{\xi}\le\gamma(0)$},
\end{cases} \\[3mm]
\bar h(t,\xi)&=h(t,\xi),\\[3mm]
P(t,\xi)&=\frac14 \int_{\Real} e^{-\vert
      y(t,\xi)-y(t,\eta)\vert}(2U^2y_\xi+ h)(t,\eta)d\eta, \\[3mm]
Q(t,\xi)&=-\frac14 \int_{\Real} \sgn{(\xi-\eta)}e^{-\vert
      y(t,\xi)-y(t,\eta)\vert}(2U^2y_\xi+h)(t,\eta)d\eta. 
\end{aligned}
\end{equation}
With the choice of initial characteristics we obtain 
\begin{equation}
y(t, \pm\gamma(0))=\pm \gamma(t),
\end{equation}
and hence at the peaks
\begin{equation}
U(t, \pm\gamma(0))=u(t,\pm\gamma(t))= \pm\frac{E}{2}\tanh\big(\frac{E}{2}(t-t_0)\big).
\end{equation}

As expected
\begin{equation}
\tau(\xi)=\begin{cases}
\infty, & \text{for $\abs{\xi}\ge\gamma(0)$}, \\
t_0, & \text{for $\abs{\xi}<\gamma(0)$}.
\end{cases}
\end{equation}
The important quantity is the first time there is wave breaking. By construction
\begin{equation}
t_0= \inf_\xi\tau(\xi).
\end{equation}
Next we consider the limits of these variables as  $t\uparrow t_0$:
\begin{equation} \label{eq:AP11}
\begin{aligned}
\lim_{t\uparrow t_0}y(t,\xi)&=
\begin{cases}
\xi+\sgn(\xi)\ln\big(1+(1-\cosh(\frac{E}{2})t_0) e^{-\abs{\xi}} \big), & \text{for $\abs{\xi}\ge \gamma(0)$},\\
0, & \text{for $\abs{\xi}< \gamma(0)$},
\end{cases}  \\
\lim_{t\uparrow t_0} U(t,\xi)&= 0, \\
\lim_{t\uparrow t_0} h(t,\xi)&= \begin{cases} 
0,  & \text{for $\abs{\xi}\ge \gamma(0)$}, \\
\frac{E^2}{4}\cosh^{-2}(\frac{\xi}{2})\tanh^{-2}(-\frac{E}{4}t_0),  & \text{for $\abs{\xi}\le\gamma(0)$}, 
\end{cases} \\
\lim_{t\uparrow t_0}P(t,\xi)&=\begin{cases}
\frac{E^2}{4}(1+(1-\cosh(\frac{E}{2}t_0))e^{-\abs{\xi}})^{-1}e^{-\abs{\xi}},  & \text{for $\abs{\xi}\ge \gamma(0)$},\\
\frac{E^2}{4}, & \text{for $\abs{\xi}\le\gamma(0)$}, 
\end{cases}\\
\lim_{t\uparrow t_0}Q(t,\xi)&=-\begin{cases}
\sgn(\xi)\frac{E^2}{4} (1+(1-\cosh(\frac{E}{2}t_0))e^{-\abs{\xi}})^{-1}e^{-\abs{\xi}},  & \text{for $\abs{\xi}\ge \gamma(0)$},\\
\frac{E^2}{4}\tanh^{-2}(-\frac{E}{4}t_0)\tanh(\frac{\xi}{2}), & \text{for $\abs{\xi}\le\gamma(0)$}.
 \end{cases}
\end{aligned}
\end{equation}
\medskip
At $t=t_0$ we introduce the parameter $\alpha\in[0,1]$, and define
\begin{equation}
\bar h(t_0,\xi)= (1-\alpha) \lim_{t\uparrow t_0}h(t,\xi), \quad 
h(t_0,\xi)=  \lim_{t\uparrow t_0}h(t,\xi).
\end{equation}

This implies that in Eulerian variables
\begin{equation} \label{eq:APet0A}
u(t_0,x)=0, \quad \mu(t_0)=(1-\alpha)E^2\delta_0, \quad \text{and} \quad \nu(t_0)=E^2 \delta_0,
\end{equation}
using the definitions    \eqref{eq:umudef2}, namely $\mu=y_\#(\bar h(\xi)\,d\xi)$, and 
      \eqref{eq:umudef5}, that is, 
      $\nu=y_\#(h(\xi)\,d\xi)$.

We will show that for  $t>t_0$ the solution  coincides with the peakon-antipeakon solution with the energy $E$ replaced by 
\begin{equation}
 \tilde E=\sqrt{1-\alpha}\, E.
\end{equation}

For $t>t_0$ the Lagrangian system reads (cf.~\eqref{eq:sysdiss})
\begin{subequations} \label{eq:lagrangePA}
    \begin{align} 
    y_t& =U, \\
     U_t&=-Q,\\
      y_{t,\xi}&=U_\xi,\\ U_{t,\xi}&= \frac12 \bar
      h+(U^2-P)y_\xi,\\
      h_t& =2(U^2-P)U_\xi, \\
       \bar h_t& =h_t,
    \end{align}
\end{subequations}
where  $P$ and $Q$ are given by \eqref{eq:Plag1} and \eqref{eq:Qlag1}, respectively.

\begin{figure}\centering
      \includegraphics[width=.45\textwidth]{mpeakon_1}
       \includegraphics[width=.45\textwidth]{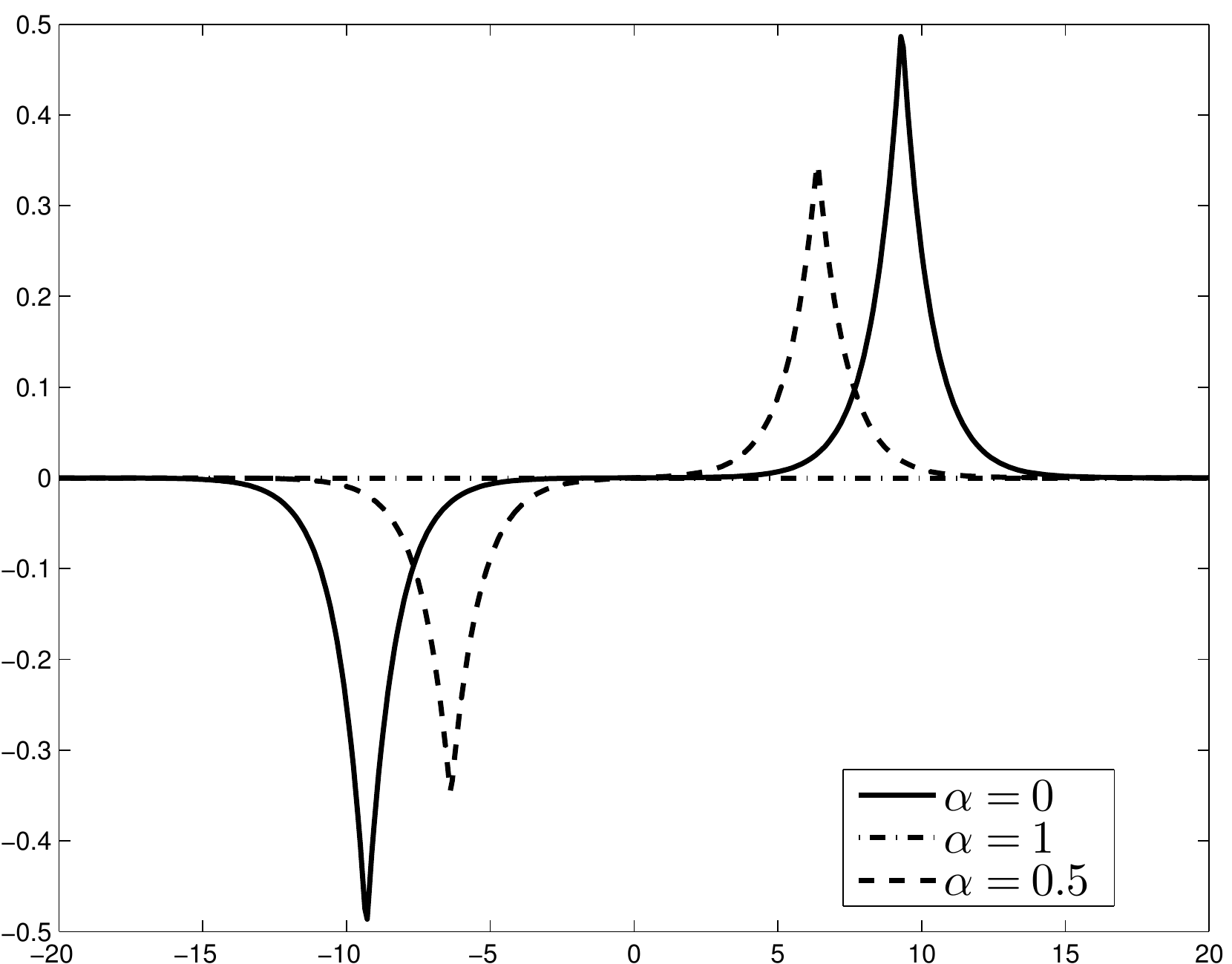}
      \caption{The peakon-antipeakon solution at times $0$ and $t>t_0$ for three different values of $\alpha$.} \label{fig:PA2}
    \end{figure}

In the fully dissipative case with $\alpha=1$, we get $\bar h(t_0)=0$, but also $(U^2-P)(t_0)=U_\xi(t_0)=Q(t_0)=0$, and hence
we have  for $t>t_0$:
\begin{equation}
\begin{aligned}
y(t,\xi)&=y(t_0,\xi), \\
U(t,\xi)&=0, \\
h(t,\xi)&=h(t_0,\xi), \\
\bar h(t,\xi)&=0. 
\end{aligned}
\end{equation}
This implies that in Eulerian variables
\begin{equation} \label{eq:APet0}
u(t,x)=0, \quad \mu(t)=0, \quad \nu(t)=E^2 \delta_0, \qquad t>t_0.
\end{equation}

In the general case $\alpha\in[0,1)$, it is difficult, as it was for $t<t_0$,  to solve the system \eqref{eq:lagrangePA} explicitly.  
However, we proceed as follows. Given \eqref{eq:APet0A}, we use \eqref{eq:Ldef1}, denoting the characteristics by 
$\tilde y(t_0)$, to determine the new initial characteristics. We find
\begin{equation}
\tilde y(t_0,\xi)= \begin{cases}
\xi, & \text{for $\xi\le 0$},\\
0, & \text{for $0\le \xi< \tilde E^2$}, \\
\xi -\tilde E^2, & \text{for $\xi\ge \tilde E^2$}.
\end{cases} 
\end{equation}
 Note that this function is related by relabeling to the characteristics we already have at $t=t_0$,  given by \eqref{eq:AP11}, namely
\begin{equation}\label{eq:AP23}
y(t_0,\xi)=
\begin{cases}
\xi+\sgn(\xi)\ln\big(1+(1-\cosh(\frac{E}{2}t_0)) e^{-\abs{\xi}} \big), & \text{for $\abs{\xi}\ge \gamma(0)$},\\
0, & \text{for $\abs{\xi}< \gamma(0)$}.
\end{cases} 
\end{equation}
To that end define
\begin{equation}
\begin{aligned}
g(\xi)&=y(t_0,\xi)+\bar H(t_0,\xi)=y(t_0,\xi)+  \int_{-\infty}^\xi \bar h(t_0,\eta) \, d\eta  \\
&= \begin{cases}
\xi-\ln\big(1+(1-\cosh(\frac{E}{2}t_0)) e^{\xi} \big), & \text{for $\xi\le -\gamma(0)$},\\
\frac{\tilde E^2}{2}\big(\tanh^{-2}(-\frac{E}{4}t_0)\tanh(\frac{\xi}{2})+1 \big), & \text{for $-\gamma(0)\le \xi< \gamma(0)$}, \\
\xi+\tilde E^2+\ln\big(1+(1-\cosh(\frac{E}{2}t_0)) e^{-\xi} \big) , & \text{for $\xi\ge \gamma(0)$}.
\end{cases} 
\end{aligned}
\end{equation}
Observe that $g$ is a monotonically increasing relabeling function  that satisfies 
\begin{equation}
\lim_{\xi\to-\gamma(0)}g(\xi)=0, \qquad \lim_{\xi\to\gamma(0)}g(\xi)=\tilde E^2,
\end{equation}
and thus
\begin{equation}
y(t_0,\xi)=\tilde y(t_0, g(\xi)).
\end{equation}
We are now given initial data $y(t_0)$, as well as $U(t_0)=0$ and $h(t_0)$. We claim that the solution, in Eulerian variables, is
\begin{equation}
u(t,x)= \begin{cases} 
\sgn(x)\tilde A(t)e^{-\abs{x}}, & \text{for $\abs{x}>\tilde\gamma(t)$},\\
\tilde B(t)\sinh x, &
\text{for $\abs{x}\le\tilde\gamma(t)$}, 
\end{cases} \qquad t>t_0, 
\end{equation}
where
\begin{equation}
\tilde A(t)= \frac{\tilde E}{2}\sinh(\frac{\tilde E}{2}(t-t_0)), \quad 
\tilde B(t)=\tilde E\sinh^{-1}(\frac{\tilde E}{2}(t-t_0)), \quad 
\tilde\gamma(t)=\ln\cosh(\frac{\tilde E}{2}(t-t_0)).
\end{equation}

\begin{figure}\centering
      \includegraphics[width=.45\textwidth]{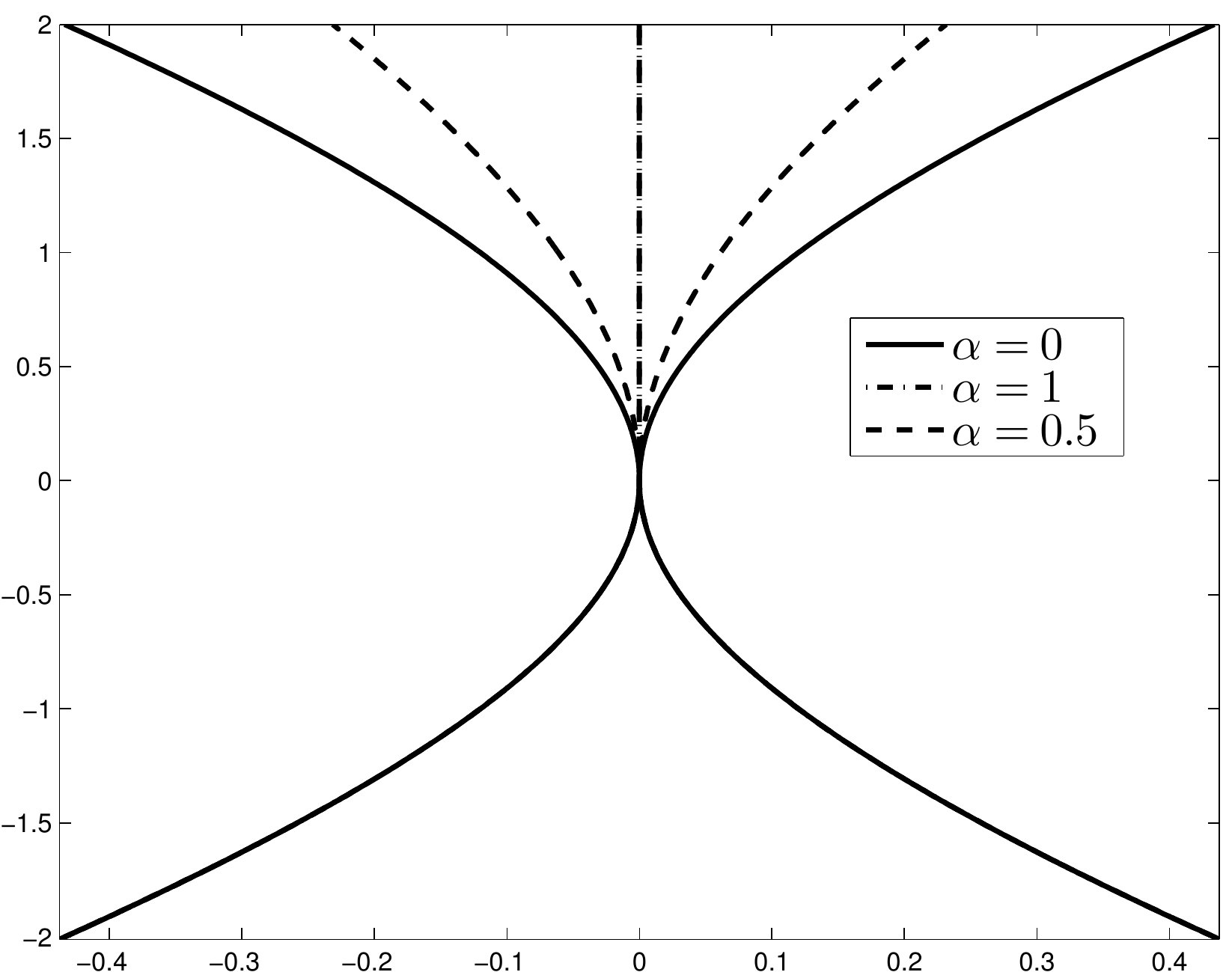}
      \caption{The characteristics $y(t,\pm \gamma(0))$ describing the location of the peaks of the peak-antipeakon solution for three different values of $\alpha$.} \label{fig:PA3}
    \end{figure}

To determine the characteristics we solve the equation $y_t=u\circ y$. We provide some details. Consider first the case 
$\xi\le -\gamma(0)$. Integrating we find
\begin{equation}
e^{-y(t,\xi)}- e^{-y(t_0,\xi)}= \cosh(\frac{\tilde E}{2}(t-t_0))-1.
\end{equation}
Inserting the expression \eqref{eq:AP23}, we find 
\begin{equation}
y(t,\xi)=
\xi-\ln\big(1+(\cosh(\frac{\tilde E}{2}(t-t_0))-\cosh(\frac{E}{2}t_0)) e^{\xi} \big).
\end{equation}
A similar calculation determines the case $\xi\ge \gamma(0)$. Assume now that $-\gamma(0)\le\xi\le \gamma(0)$.  Integrating the equation we find, for any small, positive $\varepsilon$, that
\begin{multline}
\ln\tanh(\frac12y(t,\xi)) - \ln\tanh(\frac12y(t_0+\varepsilon,\xi)) \\
= 2\big(\ln\tanh(\frac{\tilde E}{4}(t-t_0))- \ln\tanh(\frac{\tilde E}{4}\varepsilon) \big),
\end{multline}
which rewrites to
\begin{equation}
y(t,\xi)= 2\artanh\Big( \tanh(\frac12y(t_0+\varepsilon,\xi)) \frac{\tanh^2(\frac{\tilde E}{4}(t-t_0))}{\tanh^2(\frac{\tilde E}{4}\varepsilon)} \Big). 
\end{equation}
Taking $\varepsilon\downarrow 0$ we find that
\begin{equation}
y(t,\xi)= 2\artanh\Big(\tanh(\frac{\xi}{2}) \frac{\tanh^2(\frac{\tilde E}{4}(t-t_0))}{\tanh^2(-\frac{E}{4}t_0)}\Big).
\end{equation}
Note that this limit is rather delicate. As it involves repeated use of L'H\^opital's rule, one has to invoke the equations \eqref{eq:lagrangePA}  in order to compute the limit. 
We can now  determine the remaining Lagrangian quantities:
\begin{equation} \label{eq:lagrangeAP2}
\begin{aligned}
y(t,\xi)&= \begin{cases}
\xi+\sgn(\xi)\ln\big(1+(\cosh(\frac{\tilde E}{2}(t-t_0))-\cosh(\frac{E}{2}t_0)) e^{-\abs{\xi}} \big), 
& \text{for $\abs{\xi}\ge \gamma(0)$}, \\
2\artanh\Big(\tanh(\frac{\xi}{2}) \frac{\tanh^2(\frac{\tilde E}{4}(t-t_0))}{\tanh^2(-\frac{E}{4}t_0)}\Big), 
& \text{for $\abs{\xi}\le\gamma(0)$}, 
\end{cases} \\[3mm]
U(t,\xi)&=
 \begin{cases}
\begin{matrix}\sgn(\xi)\tilde A(t) e^{-\abs{\xi}}\\
\times\big(1+(\cosh(\frac{\tilde E}{2}(t-t_0))-\cosh(\frac{E}{2}t_0)) e^{-\abs{\xi}} \big)^{-1}\end{matrix}, & \text{for $\abs{\xi}\ge \gamma(0)$}, \\
\begin{matrix}2\tilde B(t) \tanh(\frac{\xi}{2}) \frac{\tanh^2(\frac{\tilde E}{4}(t-t_0))}{\tanh^2(-\frac{E}{4}t_0)}\\
\times\big(1-\tanh^2(\frac{\xi}{2}) \frac{\tanh^4(\frac{\tilde E}{4}(t-t_0))}{\tanh^4(-\frac{E}{4}t_0)} \big)^{-1}\end{matrix}, & \text{for $\abs{\xi}\le\gamma(0)$}, \end{cases}\\[3mm]
\bar h(t,\xi)&=
\begin{cases}
\begin{matrix}\tilde A(t)^2e^{-2\abs{\xi}}\\
\times\big(1+(\cosh(\frac{\tilde E}{2}(t-t_0))-\cosh(\frac{E}{2}t_0)) e^{-\abs{\xi}}\big)^{-3}\end{matrix}, & \text{for $\abs{\xi}\ge \gamma(0)$}, \\
\begin{matrix}\tilde  B(t)^2\Big(1+\tanh^2(\frac{\xi}{2})\frac{\tanh^4(\frac{\tilde E}{4}(t-t_0))}{\tanh^4(-\frac{E}{4}t_0)}\Big)^2\\
\times\Big(1-\tanh^2(\frac{\xi}{2})\frac{\tanh^4(\frac{\tilde E}{4}(t-t_0))}{\tanh^4(-\frac{E}{4}t_0)}\Big)^{-3}\\
 \times \cosh^{-2}(\frac{\xi}{2})\frac{\tanh^2(\frac{\tilde E}{4}(t-t_0))}{\tanh^2(-\frac{E}{4}t_0)} \end{matrix}, & \text{for $\abs{\xi}\le\gamma(0)$},
\end{cases} \\[3mm]
h(t,\xi)& = \begin{cases}
\bar h(t,\xi),  & \text{for $\abs{\xi}\ge \gamma(0)$}, \\  
\bar h(t,\xi)+\alpha\frac{E^2}{4}\cosh^{-2}\big(\frac{\xi}{2}\big)\tanh^{-2}\big(-\frac{E}{4}t_0\big),& \text{for $\abs{\xi}<\gamma(0)$}, 
            \end{cases}\\[3mm]
P(t,\xi)&=\frac14 \int_{\Real} e^{-\vert
      y(t,\xi)-y(t,\eta)\vert}(2U^2y_\xi+ \bar h)(t,\eta)d\eta, \\[3mm]
Q(t,\xi)&=-\frac14 \int_{\Real} \sgn{(\xi-\eta)}e^{-\vert
      y(t,\xi)-y(t,\eta)\vert}(2U^2y_\xi+\bar h)(t,\eta)d\eta. 
\end{aligned}
\end{equation}

To complete the calculation of the Eulerian variables, we use the definitions  \eqref{eq:umudef2} and  \eqref{eq:umudef5} to determine the measures. 
First we find
\begin{equation}
\mu(t)= u_x^2(t,x)\dx,\quad t>t_0.
\end{equation}
To determine $\nu(t)$ we write $h=\bar h+l_1$ (see \eqref{eq:deflj}) with $\bar h=u_x^2\circ y y_\xi$. We want to determine a function $l$ such that $l_1=l\circ y y_\xi$, which implies  
$\nu=(\bar h+l)dx$.  From \eqref{eq:lagrangeAP2} we see that $l_1(t,\xi)=0$ for $\abs{\xi}\ge \gamma(0)$. For  $\abs{\xi}< \gamma(0)$ we first observe  from  \eqref{eq:lagrangeAP2} that
\begin{align*}
y_\xi(t,\xi)&= \Big(1-\big(\tanh(\frac{\xi}{2}) \frac{\tanh^2(\frac{\tilde E}{4}(t-t_0))}{\tanh^2(-\frac{E}{4}t_0)}\big)^2 \Big)^{-1} 
\cosh^{-2}(\frac{\xi}{2}) \frac{\tanh^2(\frac{\tilde E}{4}(t-t_0))}{\tanh^2(-\frac{E}{4}t_0)} \\
&= \big(1-\tanh^{2}(\frac{y(t,\xi)}{2})  \big)^{-1} \cosh^{-2}(\frac{\xi}{2}) \frac{\tanh^2(\frac{\tilde E}{4}(t-t_0))}{\tanh^2(-\frac{E}{4}t_0)}, 
\end{align*}
using
\begin{equation*}
\tanh(\frac{y(t,\xi)}{2})= \tanh(\frac{\xi}{2})  \frac{\tanh^2(\frac{\tilde E}{4}(t-t_0))}{\tanh^2(-\frac{E}{4}t_0)}.
\end{equation*}
Thus for  $\abs{\xi}< \gamma(0)$
\begin{align*}
l(t,y(t,\xi))= \frac{l_1(t,\xi)}{y_\xi(t,\xi)}&= \alpha\frac{E^2}{4}\frac{\big(1-\tanh^{2}(\frac{y(t,\xi)}{2})\big) \cosh^{2}(\frac{\xi}{2})\tanh^{2}(-\frac{E}{4}t_0)}{\cosh^{2}\big(\frac{\xi}{2}\big)\tanh^{2}(-\frac{E}{4}t_0)\tanh^{2}(\frac{\tilde E}{4}(t-t_0))}\\
&= \alpha\frac{E^2}{4}\big(1-\tanh^{2}(\frac{y(t,\xi)}{2})\big) \tanh^{-2}(\frac{\tilde E}{4}(t-t_0)),
\end{align*}
thus we infer that
\begin{equation*}
l(t,x)=\begin{cases}
0,  & \text{for $\abs{x}\ge \tilde\gamma(t)$}, \\  
\alpha\frac{E^2}{4}\big(1-\tanh^{2}(\frac{x}{2})\big) \tanh^{-2}(\frac{\tilde E}{4}(t-t_0)),  & \text{for $\abs{x}\le\tilde\gamma(t)$}.
\end{cases}
\end{equation*}
Finally, we get the following expression 
\begin{equation}
\nu(t)=\begin{cases}
 u_x^2(t,x)\dx, & \text{for $\abs{x}\ge\tilde \gamma(t)$}, \\ 
\Big(u_x^2(t,x)+ \alpha\frac{E^2}{4}\tanh^{-2}\big(\frac{\tilde E}{4}(t-t_0)\big)\big(1-\tanh^2\big(\frac{x}{2}\big)\big)\Big)\dx, & \text{for $\abs{x}\le\tilde \gamma(t)$},
                                     \end{cases}
\end{equation} 
for $t>t_0$.

%---------- references ------------------

\end{document}